\documentclass[11pt,reqno]{amsart}

\usepackage[T1]{fontenc}
\usepackage{lmodern}
\usepackage{microtype}
\usepackage{amsmath,amssymb,mathtools,mathrsfs}
\usepackage{enumitem}
\usepackage{booktabs}
\usepackage{aliascnt}
\usepackage{xcolor}
\usepackage[a4paper,margin=29mm]{geometry}
\usepackage[colorlinks=true,linkcolor=blue!45!black,citecolor=blue!45!black,urlcolor=blue!45!black]{hyperref}
\usepackage[nameinlink,capitalize,noabbrev]{cleveref}

\theoremstyle{plain}
\newtheorem{theorem}{Theorem}[section]
\newaliascnt{proposition}{theorem}
\newtheorem{proposition}[proposition]{Proposition}
\aliascntresetthe{proposition}
\newaliascnt{lemma}{theorem}
\newtheorem{lemma}[lemma]{Lemma}
\aliascntresetthe{lemma}
\newaliascnt{corollary}{theorem}
\newtheorem{corollary}[corollary]{Corollary}
\aliascntresetthe{corollary}

\theoremstyle{definition}
\newaliascnt{definition}{theorem}
\newtheorem{definition}[definition]{Definition}
\aliascntresetthe{definition}
\newaliascnt{example}{theorem}
\newtheorem{example}[example]{Example}
\aliascntresetthe{example}
\theoremstyle{remark}
\newaliascnt{remark}{theorem}
\newtheorem{remark}[remark]{Remark}
\aliascntresetthe{remark}
\crefname{theorem}{Theorem}{Theorems}
\crefname{proposition}{Proposition}{Propositions}
\crefname{lemma}{Lemma}{Lemmas}
\crefname{corollary}{Corollary}{Corollaries}
\crefname{definition}{Definition}{Definitions}
\crefname{example}{Example}{Examples}
\crefname{remark}{Remark}{Remarks}
\numberwithin{equation}{section}

\DeclareMathOperator{\Var}{Var}
\DeclareMathOperator{\Cov}{Cov}
\DeclareMathOperator{\HS}{HS}
\newcommand{\R}{\mathbb R}
\newcommand{\C}{\mathbb C}
\newcommand{\T}{\mathbb T}
\newcommand{\N}{\mathbb N}
\newcommand{\Nzero}{\mathbb N_0}
\newcommand{\Z}{\mathbb Z}
\newcommand{\E}{\mathbb E}
\newcommand{\Pp}{\mathbb P}
\newcommand{\eps}{\varepsilon}
\newcommand{\dd}{\mathrm d}
\newcommand{\id}{\mathrm{id}}
\newcommand{\Ran}{\operatorname{Ran}}
\newcommand{\Ker}{\operatorname{Ker}}
\newcommand{\BCH}{\operatorname{BCH}}
\newcommand{\ad}{\operatorname{ad}}
\newcommand{\supp}{\operatorname{supp}}
\newcommand{\co}{\operatorname{co}}
\newcommand{\Res}{\operatorname{Res}}
\newcommand{\cA}{\mathcal A}
\newcommand{\cB}{\mathcal B}
\newcommand{\cD}{\mathcal D}
\newcommand{\cI}{\mathcal I}

\newcommand{\cQ}{\mathcal Q}
\newcommand{\cR}{\mathcal R}
\newcommand{\cS}{\mathcal S}

\newcommand{\F}{\mathcal F}
\newcommand{\Hh}{\mathfrak H}
\newcommand{\G}{\mathcal G}
\newcommand{\one}{\mathbf 1}
\newcommand{\A}{\mathbb A}

\newcommand{\Rect}{\mathsf R}
\newcommand{\norm}[1]{\lVert #1\rVert}
\newcommand{\abs}[1]{\lvert #1\rvert}

\title[Rough Paths Below the Young Threshold]{Rough Paths Below the Young Threshold:\\An Exact Scale Calculus and the Locality Phase Transition at One Quarter}
\author{Zongjian Han}
\address{School of Mathematical Sciences, Tongji University, Shanghai 200092, China}
\email{dbln@tongji.edu.cn}
\date{}
\subjclass[2020]{Primary 60L20; Secondary 60G22, 42B25, 60H10, 46H25}
\keywords{rough path, fractional Brownian motion, measurable locality, one-quarter threshold, exact scale calculus, corona derivation, Fourier normal ordering, Heisenberg equation}
\hypersetup{
  pdftitle={Rough Paths Below the Young Threshold: An Exact Scale Calculus and the Locality Phase Transition at One Quarter},
  pdfauthor={Zongjian Han},
  pdfsubject={Exact scale calculus and the one-quarter measurable-locality threshold for fractional Brownian rough paths},
  pdfkeywords={rough paths, fractional Brownian motion, measurable locality, one-quarter threshold, exact scale calculus, corona derivation, Fourier normal ordering, Heisenberg equation}
}

\begin{document}

\begin{abstract}
For a continuous path $x:[0,T]\to\R^d$, a second-level coordinate $\A^{ab}_{s,t}$ is interval-local when it is measurable from the increments of components $a,b$ on $[s,t]$.  Compatibility with subdivision is expressed by Chen's identity
\[
 \A^{ab}_{s,t}=\A^{ab}_{s,u}+\A^{ab}_{u,t}+x^a_{s,u}x^b_{u,t},
 \qquad x^a_{s,t}:=x^a(t)-x^a(s).
\]
We determine their compatibility for fractional Brownian motion and develop exact scale coordinates for deterministic rough enhancement.

Let $B^H$ be fractional Brownian motion of Hurst index $H$, with independent components in dimension $d\ge2$.  For $0<H\le1/4$, no Borel domain of positive fractional-Brownian probability supports even one finite off-diagonal local Chen coordinate.  No moment, H\"older, or geometricity assumption is required.  For $1/4<H\le1/2$, put $m=\lfloor1/H\rfloor$ and choose $1/(m+1)<\eta<H$.  Every $\eta$-H\"older rough-path lift on a full-measure Borel domain, with each tensor word measurable from its corresponding interval increments, has square-integrable coordinates through level $m$.  Its difference from the canonical lift consists of deterministic higher-level increments and their forced Gaussian cross terms.  For $H>1/2$, graded H\"older bounds with exponent greater than $1/2$ determine the Young signature uniquely.

For a bounded function $f$ with increasing concave modulus of continuity $\omega$, let $\Delta_j f$ be its smooth frequency block at wavelength $r_j=2^{-j}$ and set $\lambda_j=\omega(r_j)$.  Centered block differences at separation $\varepsilon r_j$, divided by $\varepsilon\lambda_j$, converge uniformly in $j$, with sharp error $O(\varepsilon^2)$, to
\[
 d_j^\omega f=(r_j/\lambda_j)\partial_x\Delta_jf.
\]
The low-frequency block together with the full derivative tower reconstructs $f$ exactly.  Passing instead to bounded towers modulo towers converging uniformly to zero turns the asymptotic product and chain rules into exact identities; normalized higher product defects yield multiderivations.

A fixed Fourier-normal-ordering scheme yields Borel geometric lifts at every positive regularity, with explicit cutoff rates at lower rough exponents.  Thus one quarter separates measurable locality from lift existence.  The obstruction already excludes interval-local multiplicative flows for a single linear Heisenberg equation on positive-probability Borel domains when $H\le1/4$.
\end{abstract}

\maketitle

\section{Introduction}\label{sec:introduction}

\subsection{From the Young threshold to an information threshold}

Fix $T>0$ and a finite-dimensional real normed space $E$.  For a path $x:[0,T]\to E$, write
\[
 x_{s,t}:=x(t)-x(s),
 \qquad
 [x]_\eta:=\sup_{0\le s<t\le T}\frac{\|x_{s,t}\|}{|t-s|^\eta}.
\]
The space $C^\eta([0,T];E)$ consists of continuous paths with $[x]_\eta<\infty$, with norm $\|x\|_\infty+[x]_\eta$; the subscript $0$ means that $x(0)=0$.  For smooth paths in $\R^d$, integration produces the second-level coordinates
\begin{equation}\label{eq:intro-smooth-area}
 \A^{ab}_{s,t}(x):=\int_s^t x^a_{s,u}\,\dd x^b(u).
\end{equation}
Splitting this integral at an intermediate time gives
\begin{equation}\label{eq:intro-chen}
 \A^{ab}_{s,t}
 =\A^{ab}_{s,u}+\A^{ab}_{u,t}+x^a_{s,u}x^b_{u,t}.
\end{equation}
The integral depends only on the increments of the two displayed components within $[s,t]$.

Young's theorem extends \eqref{eq:intro-smooth-area} to paths with complementary H\"older exponents whose sum exceeds one \cite{Young1936}.  In particular, a single $\eta$-H\"older path determines these integrals when $2\eta>1$.  For general paths of exponent at most $1/2$, this regularity condition no longer supplies the iterated integrals.  Rough-path theory retains their algebraic identities and graded regularity as additional input rather than assuming that an ordinary integral exists \cite{Chen1957,Lyons1998}.

We recall the precise meaning of that input.  For an integer $m\ge1$, let
\[
 T^{(m)}(E):=\bigoplus_{r=0}^m E^{\otimes r},
 \qquad E^{\otimes0}:=\R,
\]
with tensor concatenation truncated above degree $m$.  Denote the projection onto degree $r$ by $\pi_r$.  Let $L_r(E)$ be the span of the degree-$r$ iterated commutators of elements of $E$, where $[u,v]=u\otimes v-v\otimes u$, and put
\[
 G^{(m)}(E):=\exp\!\left(\bigoplus_{r=1}^mL_r(E)\right).
\]
The exponential is computed in the truncated tensor algebra.  These are the group-like tensors used below.  A step-$m$ weakly geometric $\eta$-H\"older rough path above $x$ is a continuous family $\mathbf X_{s,t}\in G^{(m)}(E)$ with
\begin{equation}\label{eq:intro-rough-path}
 \begin{gathered}
 \mathbf X_{s,s}=1,
 \qquad \pi_1\mathbf X_{s,t}=x_{s,t},
 \qquad \mathbf X_{s,t}=\mathbf X_{s,u}\mathbf X_{u,t},\\
 \|\pi_r\mathbf X_{s,t}\|\le C|t-s|^{r\eta},
 \qquad 1\le r\le m.
 \end{gathered}
\end{equation}
Here one fixes norms on the finite-dimensional tensor levels, and $C<\infty$ may depend on the path.  The usual roughness depth is $m=\lfloor1/\eta\rfloor$, so $(m+1)\eta>1$.  The second-level part of multiplicativity is exactly \eqref{eq:intro-chen}.  A lift is \emph{strong geometric} at exponent $\eta$ if it is a limit, in the $\eta$-H\"older rough-path topology, of signatures of bounded-variation paths; convergence means that the tensor-coordinate differences, divided by $|t-s|^{r\eta}$ at degree $r$, tend uniformly to zero.  A signature is the tensor family of iterated integrals, with degree $r$ obtained by integrating over $s<u_1<\cdots<u_r<t$.

Extension and Fourier-normal-ordering constructions provide geometric lifts over irregular paths, with the appropriate choice of roughness exponent \cite{LyonsVictoir2007,UnterbergerFNO2010}.  They do not by themselves require that the tensor assigned to $[s,t]$ use only observations within that interval.  The distinction is the central stochastic question of this paper: can multiplicative rough geometry be selected measurably from local increments, or does a lift necessarily use additional information?

For fractional Brownian motion, the canonical Gaussian lift exists above Hurst index $1/4$ \cite{CoutinQian2002}, whereas Fourier normal ordering gives geometric lifts at every positive Hurst index \cite{UnterbergerFNO2010,UnterbergerFBM2010}.  Chevyrev and Ferrucci classified interval-local square-integrable lifts and asked whether the moment assumption is automatic, and whether any positive-measure Borel domain can support a local lift at or below one quarter \cite{ChevyrevFerrucci2026}.  We answer these questions in the wordwise-local sense defined next.  The two thresholds therefore concern different requirements:
\begin{equation}\label{eq:intro-construction-obstruction}
 \begin{aligned}
  \eta>\tfrac12&:\quad\text{Young integration from the first-level path},\\
  H>\tfrac14&:\quad\text{measurable interval-local fractional-Brownian enhancement}.
 \end{aligned}
\end{equation}
The second line is an existence statement for local lifts; the negative result below applies even without the analytic bounds in \eqref{eq:intro-rough-path}.

\subsection{Locality and the fractional-Brownian phase transition}

Let
\[
 \Omega_d:=C_0([0,T];\R^d)
\]
with its uniform topology and Borel sigma-algebra.  The law $\mu_H^{(d)}$ of $d$ independent fractional Brownian components is the centered Gaussian law with covariance
\[
 \E[B^{H,a}_sB^{H,b}_t]
 =\frac{\delta_{ab}}2\bigl(s^{2H}+t^{2H}-|t-s|^{2H}\bigr),
 \qquad 0<H<1,
\]
where $\delta_{ab}$ is one for $a=b$ and zero otherwise.  We identify the canonical process with evaluation on $\Omega_d$.

For $J\subseteq\{1,\ldots,d\}$, define the increment restriction and its sigma-field by
\begin{equation}\label{eq:intro-local-sigma-field}
 r_{s,t}^J(x):=(x^a(u)-x^a(s))_{u\in[s,t],\,a\in J},
 \qquad
 \mathcal F_{s,t}^J:=\sigma(r_{s,t}^J).
\end{equation}
If $B\subseteq\Omega_d$ is Borel, the corresponding trace sigma-field is
\[
 \mathcal F_{s,t}^J\vert_B
 :=\{B\cap A:A\in\mathcal F_{s,t}^J\}.
\]
Thus a map on $B$ is local for the coordinates $J$ if it is measurable from the observed increment restriction, even when membership in $B$ depends on the whole path.  Such a map takes the same value on any two paths in $B$ with the same restriction $r_{s,t}^J$.

In particular, a \emph{local level-two Chen coordinate on $B$} is a family of finite real-valued Borel maps $\A^{ab}_{s,t}:B\to\R$, for one fixed pair $a\ne b$, which satisfies \eqref{eq:intro-chen} for every $s\le u\le t$ and every $x\in B$, and is $\mathcal F_{s,t}^{\{a,b\}}\vert_B$-measurable.  This is the definition used in the negative theorem; continuity and moment bounds are not part of it.  The full formulation is repeated in \cref{def:positive-domain-area}.

For a tensor word $w=(i_1,\ldots,i_r)$, let $\mathbf X^w_{s,t}$ be the coefficient of $e_{i_1}\otimes\cdots\otimes e_{i_r}$ in $\pi_r\mathbf X_{s,t}$, and let $J(w)$ be the set of indices occurring in $w$.  A rough lift is \emph{wordwise-local} if $\mathbf X^w_{s,t}$ is measurable for $\mathcal F_{s,t}^{J(w)}$; on a Borel domain we use the trace sigma-field.  Throughout the paper, ``local lift'' means this wordwise condition.  It is stronger than merely being adapted to the past: the allowed information consists only of the increments on the specified interval and in the specified components.  ``Full-law'' means that the lift is defined as a random rough path under the entire law $\mu_H^{(d)}$, with identities and regularity holding almost surely.

\begin{theorem}[Fractional-Brownian locality phase transition]\label{thm:phase-transition}
Let $d\ge2$ and $0<H<1$.
\begin{enumerate}[label=(\roman*),leftmargin=2.4em]
\item If $0<H\le1/4$ and a Borel set $B\subseteq\Omega_d$ carries a local level-two Chen coordinate, then
\[
 \mu_H^{(d)}(B)=0.
\]
One finite off-diagonal coordinate suffices.  No moment bound, H\"older control, geometricity, stationarity, or scaling assumption is imposed.
\item Suppose $1/4<H\le1/2$, put $m=\lfloor1/H\rfloor\in\{2,3\}$, and choose
\[
 \frac1{m+1}<\eta<H.
\]
Let $B$ be Borel with $\mu_H^{(d)}(B)=1$, and let $\mathbf X$ be a Borel $\eta$-H\"older rough-path lift on $B$ whose rational-interval word coordinates are trace-local.  For every $s\le t$ and $1\le r\le m$,
\[
 \E_{\mu_H^{(d)}}\!\left[\|\pi_r\mathbf X_{s,t}\|^2\right]<\infty.
\]
Here a coordinate on $B$ is regarded as an almost-everywhere equivalence class.  On a common full-measure Borel subset, the lift agrees at all times with a full-law local lift of the deterministic-correction form described below.  In particular, the Chevyrev--Ferrucci classification applies without an assumed $L^2$ condition.
\item Suppose $1/2<H<1$ and $1/2<\eta<H$.  At every finite tensor step $M$, the first level $B^H$ has a unique multiplicative group-like extension satisfying
\[
 \|\pi_r\mathbf B^{H,(M)}_{s,t}\|
 \le C(x)|t-s|^{r\eta},
 \qquad 1\le r\le M.
\]
It is the truncated Young signature.  For fractional Brownian sample paths it is strong geometric at exponent $\eta$, and each of its finitely many levels has moments of every finite order.
\end{enumerate}
\end{theorem}

The correction in part~(ii) can be stated explicitly.  Let $\overline{\mathbf B}^{H}$ denote the canonical local Gaussian lift, used as the reference enhancement for $H>1/4$.  Write $\mathbf X^{(r)}:=\pi_r\mathbf X$ and $\overline{\mathbf B}^{H,(r)}:=\pi_r\overline{\mathbf B}^{H}$.  There are deterministic paths $\varphi^{(r)}:[0,T]\to L_r(\R^d)$ for $2\le r\le m$, based at zero, with increments $\varphi^{(r)}_{s,t}:=\varphi^{(r)}(t)-\varphi^{(r)}(s)$ satisfying
\[
 \|\varphi^{(r)}_{s,t}\|\le C_r|t-s|^{r\eta}.
\]
At level two,
\begin{equation}\label{eq:intro-level-two-classification}
 \mathbf X^{(2)}_{s,t}
 =\overline{\mathbf B}^{H,(2)}_{s,t}+\varphi^{(2)}_{s,t}.
\end{equation}
When $m=3$, the third level is
\begin{equation}\label{eq:intro-level-three-classification}
 \begin{aligned}
 \mathbf X^{(3)}_{s,t}
 ={}&\overline{\mathbf B}^{H,(3)}_{s,t}+\varphi^{(3)}_{s,t}\\
 &+\int_s^t\varphi^{(2)}_{s,u}\otimes\dd B^H_u
   +\int_s^t\dd B^H_u\otimes\varphi^{(2)}_{u,t}.
 \end{aligned}
\end{equation}
The two integrals are deterministic-integrand Wiener integrals: they are defined as $L^2$ Gaussian limits of linear combinations of the local increments, not as Young integrals.  The condition needed for them is available because $2\eta>1/2-H$.  The Lie-valued increments $\varphi^{(2)}$ and $\varphi^{(3)}$ form a deterministic rough path above the zero first level.  Conversely, every such deterministic correction gives a full-law local lift.  These formulas, proved in \cref{thm:auto-L2,lem:deterministic-correction-lie,cor:classification}, specify what is meant here by classification and by automatic square-integrability.

Part~(i) excludes a measurable local rule itself, not merely a preferred approximation to area.  Part~(ii) requires a full-measure domain and the displayed rough regularity; these hypotheses are not imposed in part~(i).  Combining the theorem with the Borel construction below gives
\begin{equation}\label{eq:intro-dichotomy}
 \begin{array}{c|l}
 H>1/2 & \text{unique Young enhancement under the graded bounds},\\[1mm]
 1/4<H\le1/2 & \text{local enhancements with deterministic correction freedom},\\[1mm]
 0<H\le1/4 & \text{Borel enhancements, but no positive-domain local enhancement}.
 \end{array}
\end{equation}
In the last line, enhancement exists at every rough exponent strictly below $H$.  Its unavoidable nonlocality is the distinction between lift existence and sufficiency of interval observations.

\subsection{Why the question reaches differential equations}

Let $E_{ij}$ be the elementary $3\times3$ matrix with its only nonzero entry, equal to one, at $(i,j)$, and let $I$ be the identity matrix.  Consider the right-driven equation
\begin{equation}\label{eq:heisenberg-rde}
 \dd Y_t=Y_t\bigl(E_{12}\,\dd X^1_t+E_{23}\,\dd X^2_t\bigr),
 \qquad Y_s=I.
\end{equation}
For a smooth driver, its solution is
\[
 Y_{s,t}=
 \begin{pmatrix}
 1&X^1_{s,t}&C_{s,t}\\
 0&1&X^2_{s,t}\\
 0&0&1
 \end{pmatrix},
 \qquad
 C_{s,t}=\int_s^tX^1_{s,u}\,\dd X^2_u.
\]
For a rough driver, a finite multiplicative flow of this form is specified by its central coordinate $C$: matrix multiplication gives
\[
 Y_{s,t}=Y_{s,u}Y_{u,t}
 \quad\Longleftrightarrow\quad
 C_{s,t}=C_{s,u}+C_{u,t}+X^1_{s,u}X^2_{u,t}.
\]
A \emph{local flow on $B$} means that $C_{s,t}$ is Borel and measurable for $\mathcal F_{s,t}^{\{1,2\}}\vert_B$.  It is $\eta$-H\"older if $|C_{s,t}|\le R(x)|t-s|^{2\eta}$ for an almost surely finite $R$.  Thus its central coordinate is precisely the object excluded by \cref{thm:phase-transition}(i).

\begin{theorem}[Heisenberg equation at the quarter threshold]\label{thm:heisenberg-main}
Let $d\ge2$ and let the first two components of $B^H$ drive \eqref{eq:heisenberg-rde}.
\begin{enumerate}[label=(\roman*),leftmargin=2.4em]
\item If $0<H\le1/4$, no positive-$\mu_H^{(d)}$ Borel domain supports a finite local multiplicative flow of the displayed form.
\item If $1/4<H\le1/2$ and $1/(\lfloor1/H\rfloor+1)<\eta<H$, the canonical lift gives a local $\eta$-H\"older flow.  Every such local flow on a full-measure Borel domain has, on a common full-measure Borel subset and simultaneously for all times,
\[
 C_{s,t}=\overline{\A}^{12}_{s,t}+a(t)-a(s),
\]
where $\overline{\A}^{12}$ is the $(1,2)$ word coordinate of $\overline{\mathbf B}^{H,(2)}$, and $a$ is deterministic and $2\eta$-H\"older.
\item If $1/2<H<1$ and $1/2<\eta<H$, the $\eta$-H\"older flow is unique and
\[
 C_{s,t}=\int_s^t B^{H,1}_{s,u}\,\dd B^{H,2}_u
\]
in the Young sense.
\end{enumerate}
\end{theorem}

The negative conclusion concerns a solver which uses only the specified local increments and composes over adjacent intervals.  A solution driven by a chosen nonlocal enhancement is a different object.  The Heisenberg example makes this difference visible in a single matrix entry, without a nonlinear equation or an infinite hierarchy of levels.

\subsection{A reversible differential calculus below differentiability}

The deterministic part begins before a rough enhancement is chosen.  It asks how much differential information can be retained from a function whose ordinary difference quotient does not converge.  The construction separates the wavelength of a frequency block from the displacement used to test that block.

An \emph{admissible modulus} is a continuous increasing concave function $\omega:[0,1]\to[0,\infty)$ with $\omega(0)=0$ and $\omega(r)>0$ for $r>0$; it is extended constantly for $r\ge1$.  Define
\[
 A_\omega:=C_b^\omega(\R),
 \qquad
 \|f\|_{C^\omega}:=\|f\|_\infty+
 \sup_{0<|x-y|\le1}\frac{|f(x)-f(y)|}{\omega(|x-y|)}.
\]
Here $C_b$ denotes bounded continuous real-valued functions.  The choice $\omega(r)=r^\alpha$, $0<\alpha<1$, gives the usual bounded H\"older algebra.

Fix a smooth inhomogeneous Littlewood--Paley resolution.  More explicitly, choose a low-frequency cutoff $\chi$ and an annular cutoff $\varphi$, supported away from zero, such that
\[
 \chi(\xi)+\sum_{j\ge0}\varphi(2^{-j}\xi)=1,
\]
and define the Fourier multipliers
\[
 \widehat{\Delta_{-1}f}(\xi)=\chi(\xi)\widehat f(\xi),
 \qquad
 \widehat{\Delta_jf}(\xi)=\varphi(2^{-j}\xi)\widehat f(\xi).
\]
The transforms may be read in the tempered-distribution sense.  Thus $\Delta_{-1}f$ is the low-frequency block, while $\Delta_jf$ has wavelength comparable to $r_j:=2^{-j}$.  Standard dyadic estimates give
\[
 \lambda_j:=\omega(r_j),
 \qquad
 \|\partial_x^k\Delta_jf\|_\infty
 \le C_k r_j^{-k}\lambda_j\|f\|_{C^\omega}
 \quad(k\ge0).
\]
The resolution and these estimates are developed in \cref{sec:modulus-lp}; the underlying frequency-analysis tools are classical \cite{BahouriCheminDanchin2011,Triebel1983}.

A fixed displacement $h$ probes the $j$th block at dimensionless size $h/r_j$.  To take the same small-displacement limit at every frequency, we instead use $h=\varepsilon r_j$ and define
\[
 d_{j,\varepsilon}^\omega f(x):=
 \frac{\Delta_jf(x+\varepsilon r_j/2)-\Delta_jf(x-\varepsilon r_j/2)}
      {\varepsilon\lambda_j},
 \qquad \varepsilon>0.
\]
The candidate limit and the complete coordinate are
\begin{equation}\label{eq:intro-infinitesimal}
 d_j^\omega f:=\frac{r_j}{\lambda_j}\partial_x\Delta_jf,
 \qquad d_{j,0}^\omega f:=d_j^\omega f,
\end{equation}
\[
 \cD_{\omega,\varepsilon}f
 :=\bigl(\Delta_{-1}f,(d_{j,\varepsilon}^\omega f)_{j\ge0}\bigr),
 \qquad
 \cA_{\omega,\varepsilon}:=\Ran\cD_{\omega,\varepsilon}.
\]
The low block is retained because differentiation of the annular blocks cannot recover zero frequency.

\begin{theorem}[Exact matched scale calculus]\label{thm:intro-construction}
For the preceding modulus and fixed frequency resolution, there is $\varepsilon_0>0$ such that the following statements hold for $0\le\varepsilon\le\varepsilon_0$.
\begin{enumerate}[label=(\roman*),leftmargin=2.4em]
\item For every $f\in A_\omega$,
\[
 \sup_{j\ge0}\|d_{j,\varepsilon}^\omega f-d_j^\omega f\|_\infty
 \le C\varepsilon^2\|f\|_{C^\omega}.
\]
The exponent $2$ is sharp in operator norm.  The normalization is critical in the following precise sense: for a scalar sequence $(a_j)$, the map $f\mapsto(a_j\partial_x\Delta_jf)_j$ is bounded into $\ell^\infty(\Nzero;C_b(\R))$ if and only if
\[
 \sup_j |a_j|\lambda_j/r_j<\infty.
\]
\item The map $\cD_{\omega,\varepsilon}$ is a bijection onto its realized range.  Its inverse $\cI_{\omega,\varepsilon}$ is given by
\begin{equation}\label{eq:intro-explicit-reconstruction}
 \cI_{\omega,\varepsilon}(u_{-1},z)
 =u_{-1}+\sum_{j\ge0}I_{j,\varepsilon}^\omega z_j,
\end{equation}
with uniform convergence for $(u_{-1},z)\in\cA_{\omega,\varepsilon}$.  The operator $I_{j,\varepsilon}^\omega$ has, on the analyzed annulus, Fourier symbol
\[
 \lambda_j\frac{\varepsilon}{2i\sin(\varepsilon r_j\xi/2)}
 \quad(\varepsilon>0),
 \qquad
 \frac{\lambda_j}{ir_j\xi}
 \quad(\varepsilon=0),
\]
multiplied by a rescaled smooth cutoff equal to one on that annulus.  The choice of $\varepsilon_0$ keeps the sine nonzero on the cutoff support.  In particular,
\[
 \cI_{\omega,\varepsilon}\cD_{\omega,\varepsilon}=\id_{A_\omega},
 \qquad
 \cD_{\omega,\varepsilon}\cI_{\omega,\varepsilon}
 =\id_{\cA_{\omega,\varepsilon}}.
\]
\item Put $S_Nf:=\Delta_{-1}f+\sum_{j=0}^N\Delta_jf$.  The scalar collapse
\[
 \partial_xS_Nf
 =\partial_x\Delta_{-1}f+
   \sum_{j=0}^N\frac{\lambda_j}{r_j}d_j^\omega f
\]
converges uniformly precisely when $f\in C_b^1(\R)$ and $f'$ is bounded and uniformly continuous.  Its limit is then $f'$.
\end{enumerate}
\end{theorem}

The inverse in \eqref{eq:intro-explicit-reconstruction} acts on the realized range, not on every bounded sequence of candidate blocks.  Exact reconstruction there needs no additional summability condition on $\omega$; synthesis of arbitrary towers is a separate question treated in \cref{app:free-synthesis}.  The proofs of the theorem are \cref{thm:critical-normalization,thm:matched-limit,prop:matched-sharpness,thm:analysis-reconstruction,cor:classical-collapse}.

Multiplication and smooth composition transfer to these coordinates by the explicit rules
\[
 a\star_{\omega,\varepsilon}b
 :=\cD_{\omega,\varepsilon}
       \bigl(\cI_{\omega,\varepsilon}a\,\cI_{\omega,\varepsilon}b\bigr),
 \qquad
 F_*(a):=\cD_{\omega,\varepsilon}
       \bigl(F(\cI_{\omega,\varepsilon}a)\bigr),
\]
where $a,b\in\cA_{\omega,\varepsilon}$ and $F$ is smooth on an interval containing the relevant range.  Thus the complete scale differential is an invertible coordinate for the original function.  Convergence of a classical derivative is an additional property of its weighted sum, not a condition for constructing the coordinate.

\subsection{Exact quotient rules and higher product defects}

The complete tower is reversible, but its individual operators $d_j^\omega$ are not derivations.  Their product defect satisfies
\begin{equation}\label{eq:intro-product-defect}
 \|d_j^\omega(fg)-f\,d_j^\omega g-g\,d_j^\omega f\|_\infty
 \le C\lambda_j\|f\|_{C^\omega}\|g\|_{C^\omega}.
\end{equation}
Since $\lambda_j\to0$, an exact Leibniz rule appears after discarding vanishing towers.  Define
\[
 \cB:=\ell^\infty(\Nzero;C_b(\R)),
 \qquad
 \cB_0:=\{(u_j)\in\cB:\|u_j\|_\infty\to0\},
 \qquad
 \cQ:=\cB/\cB_0.
\]
The product on $\cB$ is coordinatewise and descends to $\cQ$.  This quotient is called the \emph{scale corona}.  If $q:\cB\to\cQ$ is its quotient map, set
\[
 \delta_\omega f:=q\bigl((d_j^\omega f)_j\bigr),
 \qquad
 \iota(f):=q((f,f,\ldots)).
\]
The constant sequence $\iota(f)$ specifies the action of $A_\omega$ on the quotient.

\begin{theorem}[Corona differential]\label{thm:intro-corona}
The map $\delta_\omega:A_\omega\to\cQ$ is continuous and satisfies
\[
 \delta_\omega(fg)
 =\iota(f)\delta_\omega g+\iota(g)\delta_\omega f,
 \qquad
 \delta_\omega(F(f))=\iota(F'(f))\delta_\omega f,
\]
where $F\in C^2(I)$, $I$ is an open interval, and the closure of $f(\R)$ is a compact subset of $I$.  Its kernel is the explicitly defined little scale space
\[
 c_\Delta^\omega
 :=\{f\in A_\omega:\lambda_j^{-1}\|\Delta_jf\|_\infty\to0\},
\]
and
\[
 \|\delta_\omega f\|_\cQ
 \asymp\limsup_{j\to\infty}\lambda_j^{-1}\|\Delta_jf\|_\infty.
\]
If $M$ is a Banach symmetric $A_\omega$-module and $U:\cB\to M$ is a bounded module map with $\cB_0\subseteq\Ker U$, then there is a unique bounded module map $\overline U:\cQ\to M$ such that $U=\overline Uq$.  Consequently $f\mapsto U((d_j^\omega f)_j)$ factors through $\delta_\omega$ and satisfies the same exact product and chain rules.
\end{theorem}

Here symmetric means that the left and right module actions coincide.  The last statement gives the precise universal property: it applies to bounded module-valued operations on the raw tower that annihilate vanishing sequences.  The quotient differential retains the nonvanishing critical part, whereas the complete coordinate also retains the low block and subcritical information.  The two objects therefore serve different purposes.  The theorem follows from \cref{thm:exact-corona-calculus,thm:universal-exactification,thm:corona-kernel}; nonsmooth derivations on Lipschitz algebras provide a complementary context \cite{Weaver2000}.

Higher nonlinear information is retained by normalizing the defects before taking the quotient.  For $[n]:=\{1,\ldots,n\}$ and the empty-product convention $\prod_\varnothing f_i=1$, define
\[
 \Phi_{j,n}^\omega(f_1,\ldots,f_n)
 :=\sum_{S\subseteq[n]}(-1)^{n-|S|}
 d_j^\omega\!\left(\prod_{i\in S}f_i\right)\prod_{i\notin S}f_i.
\]
For $n=1$ this is $d_j^\omega f_1$; for $n=2$ it is the defect in \eqref{eq:intro-product-defect}.  The exact kernel contains one increment of each argument, giving
\[
 \|\Phi_{j,n}^\omega(f_1,\ldots,f_n)\|_\infty
 \le C_n\lambda_j^{n-1}\prod_{i=1}^n\|f_i\|_{C^\omega}.
\]
Hence
\[
 \kappa_n^\omega(f_1,\ldots,f_n)
 :=q\bigl((\lambda_j^{1-n}\Phi_{j,n}^\omega(f_1,\ldots,f_n))_j\bigr)
\]
is defined.  It is symmetric and is a derivation in each argument, with $\kappa_1^\omega=\delta_\omega$.  In particular these maps are Hochschild cocycles with coefficients in $\cQ$; their role here is to record the leading terms in exact product expansions and finite Taylor--Fa\`a di Bruno chain expansions.  The formulas and remainder estimates are given in \cref{sec:higher-defects}.

\subsection{From scale coordinates to quantitative rough enhancement}

For the H\"older modulus $\omega(r)=r^\alpha$, amplitude reinsertion gives
\[
 \lambda_j=2^{-\alpha j},
 \qquad
 \lambda_jd_j^\alpha f=r_j\partial_x\Delta_jf.
\]
On the circle $\T:=\R/(2\pi\Z)$, with periodic frequency blocks, the two-channel coordinate
\[
 Z_j(f):=\bigl(\Delta_jf,r_j\partial_x\Delta_jf\bigr)
         =\bigl(\Delta_jf,\lambda_jd_j^\alpha f\bigr)
\]
therefore has both value and derivative channels at the same amplitude.  The value-channel projection recovers the original path.  More general uniformly bounded linear combinations of the two channels produce finite-dimensional dyadic paths, as specified in \cref{thm:amplitude-reinsertion}.

The common input space is $\cS^\alpha(E)$, consisting of smooth periodic blocks $x=(x_j)_{j\ge-1}$ with fixed low-frequency support for $x_{-1}$, annular support at frequency $2^j$ for $x_j$, and finite norm
\[
 \|x\|_{\cS^\alpha}
 :=\|x_{-1}\|_{C^1}
 +\sup_{j\ge0}\left(2^{\alpha j}\|x_j\|_\infty
                   +2^{(\alpha-1)j}\|\partial_xx_j\|_\infty\right).
\]
Its synthesized path and frequency cutoffs are
\[
 X:=\sum_{j\ge-1}x_j,
 \qquad X_N:=\sum_{-1\le j\le N}x_j.
\]
For $0<\gamma<\alpha<1$, these satisfy
\[
 \|X-X_N\|_{C^\gamma}\le C2^{-(\alpha-\gamma)N}\|x\|_{\cS^\alpha}.
\]
The term ``ultraviolet cutoff'' below refers exactly to retaining the blocks with $j\le N$.

We use one fixed \emph{Fourier-normal-ordering} (FNO) scheme.  This is Unterberger's construction of iterated-integral coordinates by ordering frequency interactions, regularizing the integration-tree formulas, and preserving Chen and shuffle identities \cite{UnterbergerFNO2010}.  Its existence and lower-exponent geometric approximation conclusions are imported in \cref{thm:unterberger}.  The estimates proved here concern this fixed scheme, not a scheme-independent canonical enhancement.

To state the estimates, for two step-$m$ rough paths on the same interval define
\begin{align*}
 d_{\beta,m}^{\mathrm{lay}}(\mathbf X,\mathbf Y)
 &:=\max_{1\le r\le m}\sup_{s<t}
   \frac{\|\pi_r(\mathbf X_{s,t}-\mathbf Y_{s,t})\|}{|t-s|^{r\beta}},\\
 d_{\beta,m}^{\mathrm{hom}}(\mathbf X,\mathbf Y)
 &:=\max_{1\le r\le m}\sup_{s<t}
   \frac{\|\pi_r(\mathbf X_{s,t}-\mathbf Y_{s,t})\|^{1/r}}{|t-s|^\beta}.
\end{align*}
These are the layerwise and homogeneous coordinate metrics used throughout.  The $r$th root in the second metric is the source of the different stability exponents.

For a periodic path $Y$, fix a smooth compactly supported cutoff $\chi_0$ equal to one near $[0,2\pi]$, let $\widetilde Y$ be the periodic extension to $\R$, and set
\[
 LY:=\chi_0(\widetilde Y-Y(0)).
\]
Thus $LY$ is based at zero and has the same increments as $Y$ on the cut interval.  Write $\cR_\gamma$ for the selected FNO map on compactly supported $\gamma$-H\"older paths, and set
\[
 \mathbf X^N:=\cR_\gamma(LX_N)\vert_{[0,2\pi]},
 \qquad
 \mathbf X:=\cR_\gamma(LX)\vert_{[0,2\pi]}.
\]

\begin{theorem}[Quantitative scale--FNO enhancement]\label{thm:intro-rough-enhancement}
Let $E$ be finite-dimensional and choose
\[
 \frac1{m+1}<\beta<\gamma<\alpha\le\frac1m,
 \qquad 0<\alpha<1.
\]
For $x\in\cS^\alpha(E)$, the preceding construction yields a step-$m$ strong geometric $\beta$-H\"older lift above $X$.  On every ball $\|x\|_{\cS^\alpha}\le R$,
\[
 \begin{aligned}
 d_{\beta,m}^{\mathrm{lay}}(\mathbf X^N,\mathbf X)
 &\le C_R2^{-(\alpha-\gamma)N},\\
 d_{\beta,m}^{\mathrm{hom}}(\mathbf X^N,\mathbf X)
 &\le C_R2^{-(\alpha-\gamma)N/m}.
 \end{aligned}
\]
For two inputs in that ball, put $\delta:=\|x-y\|_{\cS^\alpha}\le1$.  Then
\[
 d_{\beta,m}^{\mathrm{lay}}(\mathbf X[x],\mathbf X[y])\le C_R\delta,
 \qquad
 d_{\beta,m}^{\mathrm{hom}}(\mathbf X[x],\mathbf X[y])\le C_R\delta^{1/m}.
\]
The constants may depend on the exponents, the fixed localization and FNO scheme, and the radius, but not on $N$.
\end{theorem}

The proof assigns separate inputs to the integration vertices in the regularized FNO formulas, proves the mixed multilinear estimate, and applies it to the cutoff tail; see \cref{prop:mixed-fno,thm:rough-enhancement}.  The object $\mathbf X^N$ is the selected FNO lift of the smooth cutoff $X_N$.  It is not identified with the unmodified signature of $X_N$.  Strong geometricity uses the bounded-variation approximation supplied by the FNO theorem at a strictly lower exponent.  This distinction is essential when ordinary smooth approximations do not select the desired area.

For a based path $x$ on $[0,T]$, a fixed loop extension records $x$ on the first half-circle and returns along the same path on the second:
\[
 (\mathscr L_Tx)(\theta):=
 \begin{cases}
 x(T\theta/\pi),&0\le\theta\le\pi,\\
 x(T(2\pi-\theta)/\pi),&\pi\le\theta\le2\pi.
 \end{cases}
\]
Apply periodic scale analysis, the same localization and FNO scheme, and then restrict to the first half-circle with the inverse time change.  Let $\mathfrak R_T^{\alpha,\beta}$ denote the resulting map, and let $\mathscr{RP}_{\beta,m}([0,T];E)$ be the space of step-$m$ strong geometric $\beta$-H\"older rough paths with the preceding rough-path topology.

\begin{theorem}[Borel scale--FNO enhancement on a finite interval]\label{thm:intro-borel-lift}
Under the exponent assumptions of \cref{thm:intro-rough-enhancement}, the fixed construction gives
\[
 \mathfrak R_T^{\alpha,\beta}:
 C_0^\alpha([0,T];E)\longrightarrow
 \mathscr{RP}_{\beta,m}([0,T];E),
 \qquad
 \pi_1\mathfrak R_T^{\alpha,\beta}(x)_{s,t}=x_{s,t}.
\]
It has the same cutoff rates and local stability estimates, with constants also depending on $T$.  It admits a Borel realization on the Borel set
\[
 \Omega_{\alpha,T}(E):=\{x\in C_0([0,T];E):[x]_\alpha<\infty\}
\]
for the uniform-path Borel sigma-algebra.

Consequently, for every $0<H<1$ and $0<\beta<H$, a full-$\mu_H^{(d)}$ Borel set carries a Borel strong geometric $\beta$-H\"older lift of $B^H$, through step $\lfloor1/\beta\rfloor$.
\end{theorem}

The Borel statement follows from the smooth frequency cutoffs and their pointwise convergence in the rough-path metric, not just from continuity in the stronger H\"older norm.  At reciprocal roughness exponents the construction first uses an intermediate exponent and then the controlled finite-step extension; see \cref{sec:interval-lift,thm:controlled-extension}.  The loop extension and the frequency regularization may use information outside a requested subinterval.  Below one quarter, \cref{thm:phase-transition}(i) proves that no measurable construction can remove this dependence while retaining wordwise locality on a positive-probability Borel domain.

\subsection{The mechanism behind the exponent one quarter}

The negative theorem starts from a finite Borel coordinate, not an $L^2$ random variable.  Its proof extracts a uniform Hilbert-space estimate from that weak assumption and then contradicts it with explicit high-frequency shifts.

Let $\Hh_H$ be the Cameron--Martin Hilbert space of the scalar fractional-Brownian law: its elements are the deterministic shifts under which the Gaussian measure class is preserved.  On suitable smooth shifts the relevant deterministic bilinear form is
\[
 b(h,k):=\int_0^T h(u)\,\dd k(u).
\]
For finite-dimensional subspaces with orthonormal bases $(e_p)$ and $(f_q)$, the squared Hilbert--Schmidt norm of the restricted form is
\[
 \|b\|_{\HS}^2:=\sum_{p,q}|b(e_p,f_q)|^2.
\]
The proof shows that a positive-domain local Chen coordinate would force these finite-dimensional norms to remain uniformly bounded.

There are three steps to this implication.  First, finite-dimensional Gaussian resampling retains the common residual and independently replaces two coordinate blocks.  If the Borel domain has mass $\beta_0>0$, the four resampled corners needed for a mixed difference lie in the domain with probability at least $\beta_0^4$.  Second, locality and Chen's identity identify the mixed rectangle, for shifts with separated derivative supports, with the Young bilinear form.  Third, Borel finiteness supplies a bounded-value event of positive probability, and Gaussian polynomial anti-concentration converts that event into a dimension-free Hilbert--Schmidt bound \cite{CarberyWright2001}.  No moment estimate for the proposed coordinate is assumed at the start.

Oscillatory Cameron--Martin packets contradict the resulting bound.  On the packet subspaces of size $K$, the same forced form satisfies
\[
 \|b\|_{\HS}^2\gtrsim\sum_{k=1}^K k^{-4H}.
\]
This grows as $K^{1-4H}$ when $H<1/4$ and as $\log K$ when $H=1/4$.  The obstruction therefore includes the critical endpoint and does not depend on a chosen lift or regularization scheme.

Above one quarter, the complementary argument starts with a local lift and subtracts the canonical one.  The second-level difference is local, additive, and $2\eta$-H\"older.  The local additive rigidity theorem proves that such a term is deterministic when its order exceeds $1/2$; here $2\eta>1/2$.  If a third level is required, subtracting the two Wiener terms in \eqref{eq:intro-level-three-classification} leaves a local additive remainder of order $3\eta>1/2$, which is again deterministic.  The tensor identities force the corrections to lie in the corresponding free-Lie spaces.  This proves the displayed classification formulas and their integrability.  Above $H=1/2$, the graded exponent of an additive difference at every level $r\ge2$ is greater than one, so subdivision forces it to vanish and Young uniqueness follows.

\subsection{Relation to previous work and organization}

The paper separates three objects that should not be identified: a complete scale coordinate, its asymptotic quotient differential, and a selected rough enhancement.  The complete coordinate reconstructs the function; the corona quotient records its critical high-frequency differential class; a rough enhancement supplies multiplicative iterated-integral data.  On the stochastic side, existence of such data is distinct from the requirement that each word coordinate be measurable from the corresponding interval observations.

The analytic background consists of Young integration, classical Littlewood--Paley estimates, rough-path extension, and Unterberger's FNO existence theorem \cite{Young1936,BahouriCheminDanchin2011,Triebel1983,LyonsVictoir2007,FrizVictoir2010,UnterbergerFNO2010}.  The prior local classification under a square-integrability assumption is due to Chevyrev and Ferrucci \cite{ChevyrevFerrucci2026}.  The contributions developed here are the matched scale calculus and its exact reconstruction and quotient rules; the mixed-input, quantitative scale-to-FNO interface; positive-domain nullity without moments below and at one quarter; automatic integrability and deterministic correction formulas above one quarter; and the resulting Heisenberg flow theorem.

Part~I proves the deterministic statements in dependency order.  Sections~\ref{sec:modulus-lp}--\ref{sec:matched-reconstruction} establish the modulus estimates, sharp normalization, matched limit, and reconstruction.  Sections~\ref{sec:corona}--\ref{sec:higher-defects} treat the corona differential and higher defects.  Sections~\ref{sec:periodic-interface}--\ref{sec:fno} connect the complete coordinates to quantitative rough enhancement.  Sections~\ref{sec:typed-reconstruction}--\ref{sec:extension} develop the smooth-cutoff logarithmic and differential reconstruction, scheme dependence, and controlled extension to higher tensor order.

Part~II begins with the finite-interval Borel construction in \cref{sec:interval-lift}.  It then proves positive-domain nullity by Gaussian rectangles and Cameron--Martin packets, establishes local additive rigidity and the passage from relative domains to local full-law versions, derives automatic integrability and the classification formulas, and concludes with the Young regime and the Heisenberg equation.  The appendices discuss free synthesis outside the realized range, changes of frequency coordinates, and low-degree logarithmic formulas.

\part{Exact scale coordinates and deterministic rough enhancement}

\section{The modulus algebra and dyadic estimates}\label{sec:modulus-lp}

The normalization in \eqref{eq:intro-infinitesimal} is meaningful only after the critical size of a block at wavelength $r_j$ has been fixed.  We therefore begin with a modulus $\omega$ and derive every scale estimate from its increment geometry.  The only frequency-analysis input is a fixed smooth inhomogeneous Littlewood--Paley pair.

The modulus must compare amplitudes at different wavelengths without introducing constants depending on $j$.  Concavity supplies exactly this comparison.

\begin{definition}[Admissible modulus]\label{def:admissible-modulus}
An \emph{admissible modulus} is a continuous function
\[
 \omega:[0,1]\longrightarrow[0,\infty)
\]
such that
\[
 \omega(0)=0,
 \qquad
 \omega(r)>0\quad(0<r\leq1),
\]
\[
 \omega\ \text{is increasing},
 \qquad
 \omega\ \text{is concave}.
\]
We extend $\omega$ to $[0,\infty)$ by $\omega(r):=\omega(1)$ for $r\geq1$.  For
\[
 r_j:=2^{-j},
 \qquad
 \lambda_j:=\omega(r_j),
\]
the number $\lambda_j$ is the amplitude declared critical at frequency $2^j$.  In particular,
\[
 \lambda_j\longrightarrow0
 \qquad (j\to\infty).
\]
\end{definition}

The next inequality is the scale-comparison mechanism used in every kernel estimate below.

\begin{lemma}[Concave scale inequality]\label{lem:concave-scale}
For every admissible modulus $\omega$, every $0<r\leq1$, and every $z\in\R$,
\begin{equation}\label{eq:concave-scale}
 \omega(r\abs z)\leq (1+\abs z)\omega(r).
\end{equation}
\end{lemma}

\begin{proof}
If $\abs z\leq1$, monotonicity gives
\[
 \omega(r\abs z)\leq\omega(r).
\]
If $1<\abs z\leq r^{-1}$, concavity and $\omega(0)=0$ imply that $t\mapsto\omega(t)/t$ is decreasing.  Hence
\[
 \frac{\omega(r\abs z)}{r\abs z}
 \leq
 \frac{\omega(r)}r,
\]
and therefore
\[
 \omega(r\abs z)\leq\abs z\,\omega(r).
\]
Finally, if $\abs z>r^{-1}$, then the constant extension gives
\[
 \omega(r\abs z)=\omega(1).
\]
Concavity also gives $\omega(r)\geq r\omega(1)$, while $r\abs z>1$; consequently
\[
 \omega(1)\leq\abs z\,\omega(r).
\]
The three cases imply \eqref{eq:concave-scale}.
\end{proof}

The ambient algebra encodes exactly the increment control supplied by $\omega$.

\begin{definition}[Modulus algebra]\label{def:modulus-algebra}
For an admissible modulus $\omega$, define
\[
 A_\omega:=C_b^\omega(\R)
\]
to be the real vector space of bounded continuous functions $f:\R\to\R$ for which
\[
 [f]_\omega
 :=
 \sup_{0<\abs{x-y}\leq1}
 \frac{\abs{f(x)-f(y)}}{\omega(\abs{x-y})}
 <\infty.
\]
We use the norm
\[
 \norm f_{C^\omega}:=\norm f_\infty+[f]_\omega.
\]
The space $A_\omega$ is the domain of every commutative construction in Sections \ref{sec:modulus-lp}--\ref{sec:higher-defects}.
\end{definition}

Pointwise multiplication must remain inside the domain because the corona differential will be a derivation on this algebra.

\begin{proposition}[Banach-algebra structure]\label{prop:modulus-banach-algebra}
The pointwise product makes $A_\omega$ a commutative unital Banach algebra, and
\begin{equation}\label{eq:modulus-product}
 \norm{fg}_{C^\omega}
 \leq
 \norm f_{C^\omega}\norm g_{C^\omega}
 \qquad(f,g\in A_\omega).
\end{equation}
If $I\subset\R$ is open, $F\in C^1(I)$, and $f(\R)\Subset I$, then $F\circ f\in A_\omega$.
\end{proposition}

\begin{proof}
For $0<\abs{x-y}\leq1$,
\[
 \abs{f(x)g(x)-f(y)g(y)}
 \leq
 \norm f_\infty\abs{g(x)-g(y)}
 +
 \norm g_\infty\abs{f(x)-f(y)}.
\]
After division by $\omega(\abs{x-y})$,
\[
 [fg]_\omega
 \leq
 \norm f_\infty[g]_\omega
 +
 \norm g_\infty[f]_\omega.
\]
Together with $\norm{fg}_\infty\leq\norm f_\infty\norm g_\infty$, this yields \eqref{eq:modulus-product}.

Let $(f_n)$ be Cauchy in $A_\omega$.  It converges uniformly to a bounded continuous function $f$.  For fixed $n$ and $0<\abs{x-y}\leq1$,
\[
 \frac{\abs{(f_n-f)(x)-(f_n-f)(y)}}{\omega(\abs{x-y})}
 \leq
 \liminf_{m\to\infty}[f_n-f_m]_\omega.
\]
Taking the supremum in $x,y$ and then $n\to\infty$ proves $f_n\to f$ in $A_\omega$.

For the functional calculus, put $K_f:=\co(f(\R))\Subset I$.  The mean-value theorem gives
\[
 \abs{F(f(x))-F(f(y))}
 \leq
 \norm{F'}_{L^\infty(K_f)}\abs{f(x)-f(y)},
\]
so
\[
 [F\circ f]_\omega
 \leq
 \norm{F'}_{L^\infty(K_f)}[f]_\omega.
\]
\end{proof}

Kernel estimates below involve increments larger than one because the convolution kernels are not compactly supported.  The constant extension of $\omega$ gives the required global form.

\begin{lemma}[Global increment bound]\label{lem:global-increment}
For every $f\in A_\omega$ and every $x,y\in\R$,
\begin{equation}\label{eq:global-increment}
 \abs{f(x-y)-f(x)}
 \leq
 C_\omega\norm f_{C^\omega}\omega(\abs y),
 \qquad
 C_\omega:=\max\left\{1,\frac2{\omega(1)}\right\}.
\end{equation}
\end{lemma}

\begin{proof}
For $\abs y\leq1$, \eqref{eq:global-increment} follows from the definition of $[f]_\omega$.  If $\abs y>1$, then $\omega(\abs y)=\omega(1)$ and
\[
 \abs{f(x-y)-f(x)}
 \leq2\norm f_\infty
 \leq
 \frac2{\omega(1)}\norm f_{C^\omega}\omega(\abs y).
\]
\end{proof}

We now fix the frequency coordinates.  Their smoothness and annular support are classical; the enlarged cutoff is retained because exact inversion in Section \ref{sec:matched-reconstruction} requires it \cite{BahouriCheminDanchin2011,Triebel1983}.

\begin{definition}[Littlewood--Paley analysis pair]\label{def:lp-pair}
Use the Fourier convention
\[
 \widehat f(\xi)=\int_\R e^{-ix\xi}f(x)\,\dd x,
 \qquad
 D:=-i\partial_x.
\]
Choose an even function $\chi\in C_c^\infty(\R)$ such that
\[
 0\leq\chi\leq1,
 \qquad
 \chi=1\ \text{on }\{\abs\xi\leq3/4\},
 \qquad
 \supp\chi\subset\{\abs\xi\leq4/3\}.
\]
Set
\[
 \varphi(\xi):=\chi(\xi/2)-\chi(\xi),
\]
and choose an even $\widetilde\varphi\in C_c^\infty(\R\setminus\{0\})$ satisfying
\[
 \widetilde\varphi=1
 \quad\text{on }\supp\varphi.
\]
Define
\[
 \Delta_{-1}:=\chi(D),
 \qquad
 \Delta_j:=\varphi(r_jD),
 \qquad
 \widetilde\Delta_j:=\widetilde\varphi(r_jD)
 \quad(j\geq0).
\]
Then
\[
 \chi(\xi)+\sum_{j\geq0}\varphi(2^{-j}\xi)=1.
\]
If $K:=\mathcal F^{-1}\varphi$, then
\[
 \Delta_jf=K_j*f,
 \qquad
 K_j(x):=r_j^{-1}K(x/r_j),
 \qquad
 \int_\R K(x)\,\dd x=0.
\]
\end{definition}

The cancellation $\int K=0$ converts the modulus increment of $f$ into a sharp amplitude bound for every block.

\begin{lemma}[Scale-matched increment estimate]\label{lem:scale-matched-increment}
For every $f\in A_\omega$, $j\geq0$, and $x,z\in\R$,
\begin{equation}\label{eq:scale-matched-increment}
 \abs{f(x-r_jz)-f(x)}
 \leq
 2(1+\abs z)\lambda_j[f]_\omega.
\end{equation}
The numerical constant is independent of the admissible modulus.
\end{lemma}

\begin{proof}
If $r_j\abs z\leq1$, Definition \ref{def:modulus-algebra} and \cref{lem:concave-scale} give
\[
 \abs{f(x-r_jz)-f(x)}
 \leq
 [f]_\omega\omega(r_j\abs z)
 \leq
 (1+\abs z)\lambda_j[f]_\omega.
\]
Assume $r_j\abs z>1$ and put $N:=\lceil r_j\abs z\rceil$.  With
\[
 x_k:=x-\frac{k}{N}r_jz
 \qquad(0\leq k\leq N),
\]
each increment has length at most one.  Hence
\[
 \abs{f(x-r_jz)-f(x)}
 \leq
 N[f]_\omega\omega\!\left(\frac{r_j\abs z}{N}\right)
 \leq
 N[f]_\omega\omega(1).
\]
Since $N\leq2r_j\abs z$ and concavity gives $\lambda_j=\omega(r_j)\geq r_j\omega(1)$,
\[
 \abs{f(x-r_jz)-f(x)}
 \leq
 2\abs z\lambda_j[f]_\omega.
\]
Combining the two cases proves \eqref{eq:scale-matched-increment}.
\end{proof}

The next estimate is the analytic input for the matched limit, the corona defects, and the periodic interface.

\begin{proposition}[Uniform dyadic bounds]\label{prop:dyadic-bounds}
For every integer $k\geq0$ there is $C_k>0$, depending only on $k$ and the fixed Littlewood--Paley pair, such that
\begin{equation}\label{eq:dyadic-bounds}
 \norm{\partial_x^k\Delta_jf}_\infty
 \leq
 C_kr_j^{-k}\lambda_j\norm f_{C^\omega}
 \qquad(f\in A_\omega,\ j\geq0).
\end{equation}
Moreover,
\begin{equation}\label{eq:lp-uniform-reconstruction}
 \Delta_{-1}f+\sum_{j=0}^N\Delta_jf
 \longrightarrow f
 \quad\text{uniformly on }\R.
\end{equation}
\end{proposition}

\begin{proof}
The cancellation of $K$ and the change of variables $y=r_jz$ give
\[
 \Delta_jf(x)
 =
 \int_\R K(z)\bigl(f(x-r_jz)-f(x)\bigr)\,\dd z.
\]
By \cref{lem:scale-matched-increment},
\[
 \norm{\Delta_jf}_\infty
 \leq
 2\lambda_j[f]_\omega
 \int_\R\abs{K(z)}(1+\abs z)\,\dd z
 \leq
 C\lambda_j\norm f_{C^\omega}.
\]
The Fourier support of $\Delta_jf$ lies in a fixed annulus of radius $r_j^{-1}$.  Bernstein's multiplier estimate therefore yields
\[
 \norm{\partial_x^k\Delta_jf}_\infty
 \leq
 C_kr_j^{-k}\norm{\Delta_jf}_\infty
 \leq
 C_kr_j^{-k}\lambda_j\norm f_{C^\omega}.
\]
For \eqref{eq:lp-uniform-reconstruction}, the partial sum has multiplier $\chi(2^{-N-1}D)$ and is a smooth approximate identity at scale $2^{-N}$.  The global increment estimate gives
\[
 \norm{f-\chi(2^{-N-1}D)f}_\infty
 \leq
 C\omega(2^{-N})\norm f_{C^\omega}
 \longrightarrow0.
\]
\end{proof}

The amplitude factor in \eqref{eq:intro-infinitesimal} is not an arbitrary choice.  The following theorem characterizes every scalar normalization that is uniformly bounded on the full modulus algebra.

\begin{theorem}[Operator-norm critical normalization]\label{thm:critical-normalization}
Let $a=(a_j)_{j\geq0}$ be a real scalar sequence and define
\[
 D_af:=(a_j\partial_x\Delta_jf)_{j\geq0}.
\]
Then
\[
 D_a:A_\omega\longrightarrow\ell^\infty(\Nzero;C_b(\R))
\]
is bounded if and only if
\begin{equation}\label{eq:critical-normalization-condition}
 M_a:=\sup_{j\geq0}\frac{\abs{a_j}\lambda_j}{r_j}<\infty.
\end{equation}
There are constants $0<c\leq C<\infty$, depending only on $\omega$ and the fixed Littlewood--Paley pair, such that
\begin{equation}\label{eq:critical-normalization-norm}
 cM_a
 \leq
 \norm{D_a}_{A_\omega\to\ell^\infty(C_b)}
 \leq
 CM_a.
\end{equation}
Consequently,
\[
 a_j^{\mathrm{crit}}:=\frac{r_j}{\lambda_j}
\]
is critical up to multiplication by a bounded scalar sequence.
\end{theorem}

\begin{proof}
If \eqref{eq:critical-normalization-condition} holds, \cref{prop:dyadic-bounds} gives
\[
 \norm{a_j\partial_x\Delta_jf}_\infty
 \leq
 C\frac{\abs{a_j}\lambda_j}{r_j}\norm f_{C^\omega}
 \leq
 CM_a\norm f_{C^\omega}.
\]
Taking the supremum in $j$ proves the upper bound in \eqref{eq:critical-normalization-norm}.

For the reverse bound, choose $c_*>0$ with $\varphi(c_*)\neq0$ and define
\[
 f_j(x):=\lambda_j\cos(c_*x/r_j).
\]
We first show that
\begin{equation}\label{eq:pure-frequency-uniform-modulus}
 \sup_{j\geq0}\norm{f_j}_{C^\omega}<\infty.
\end{equation}
For $0<\abs h\leq1$,
\[
 \abs{f_j(x+h)-f_j(x)}
 \leq
 \min\left\{2\lambda_j,\frac{c_*\lambda_j}{r_j}\abs h\right\}.
\]
If $\abs h\leq r_j$, concavity gives
\[
 \omega(\abs h)
 \geq
 \frac{\abs h}{r_j}\omega(r_j)
 =
 \frac{\abs h}{r_j}\lambda_j.
\]
If $\abs h\geq r_j$, monotonicity gives $\omega(\abs h)\geq\lambda_j$.  These two inequalities prove \eqref{eq:pure-frequency-uniform-modulus}.

Because $\varphi$ is even,
\[
 \Delta_jf_j=\varphi(c_*)f_j.
\]
Therefore
\[
 \norm{a_j\partial_x\Delta_jf_j}_\infty
 =
 \abs{a_j}\abs{\varphi(c_*)}c_*\frac{\lambda_j}{r_j}.
\]
Testing $D_a$ on $f_j$ and using \eqref{eq:pure-frequency-uniform-modulus} gives
\[
 \frac{\abs{a_j}\lambda_j}{r_j}
 \leq
 C\norm{D_a}_{A_\omega\to\ell^\infty(C_b)}.
\]
Taking the supremum over $j$ proves the lower bound and the necessity of \eqref{eq:critical-normalization-condition}.
\end{proof}

A normalization asymptotically smaller than $r_j/\lambda_j$ loses every fine-scale corona class.  This observation will identify the scale used in \cref{sec:corona} as genuinely critical.

\begin{corollary}[Subcritical normalizations vanish]\label{cor:subcritical-vanish}
If
\[
 \frac{\abs{a_j}\lambda_j}{r_j}\longrightarrow0,
\]
then for every $f\in A_\omega$,
\[
 \norm{a_j\partial_x\Delta_jf}_\infty\longrightarrow0.
\]
\end{corollary}

\begin{proof}
By \cref{prop:dyadic-bounds},
\[
 \norm{a_j\partial_x\Delta_jf}_\infty
 \leq
 C\frac{\abs{a_j}\lambda_j}{r_j}\norm f_{C^\omega}
 \longrightarrow0.
\]
\end{proof}
 \section{Matched increments and exact reconstruction}\label{sec:matched-reconstruction}

The dyadic estimates determine the correct amplitude, but a differential coordinate also requires a single collision parameter that has the same dimensionless meaning at every frequency.  We therefore test the $j$th block at the physical separation $\eps r_j$.

\begin{definition}[Infinitesimal and matched scale derivatives]\label{def:matched-derivatives}
For $f\in A_\omega$ and $j\geq0$, define
\[
 d_j^\omega f
 :=
 \frac{r_j}{\lambda_j}\partial_x\Delta_jf.
\]
For $\eps>0$, define
\begin{equation}\label{eq:matched-difference}
 d_{j,\eps}^\omega f(x)
 :=
 \frac{\Delta_jf(x+\eps r_j/2)-\Delta_jf(x-\eps r_j/2)}{\eps\lambda_j}.
\end{equation}
At $\eps=0$, set $d_{j,0}^\omega:=d_j^\omega$.  The sequence $(d_j^\omega f)_{j\geq0}$ is the infinitesimal scale derivative tower.
\end{definition}

The same algebraic map converts average--derivative coordinates into the two samples at every scale.  This scale independence is the transport mechanism behind \eqref{eq:matched-difference}.

\begin{definition}[Two-point transport]\label{def:two-point-transport}
For $\eps\neq0$, define
\[
 H_\eps:\R^2\longrightarrow\R^2,
 \qquad
 H_\eps(m,p)
 :=
 \left(m-\frac\eps2p,m+\frac\eps2p\right).
\]
Its inverse is
\[
 H_\eps^{-1}(u_-,u_+)
 =
 \left(\frac{u_-+u_+}{2},\frac{u_+-u_-}{\eps}\right).
\]
For $u_j:=\lambda_j^{-1}\Delta_jf$, the second component of
\[
 H_\eps^{-1}\bigl(u_j(x-\eps r_j/2),u_j(x+\eps r_j/2)\bigr)
\]
is $d_{j,\eps}^\omega f(x)$.
\end{definition}

The centered sampling symmetry removes the first-order collision error.  The dyadic third-derivative estimate then makes the resulting second-order remainder uniform in $j$.

\begin{theorem}[Frequency-uniform matched limit]\label{thm:matched-limit}
There are constants $\eps_0>0$ and $C>0$, depending only on the fixed Littlewood--Paley pair, such that
\begin{equation}\label{eq:matched-limit}
 \sup_{j\geq0}
 \norm{d_{j,\eps}^\omega f-d_j^\omega f}_\infty
 \leq
 C\eps^2\norm f_{C^\omega}
\end{equation}
for every $f\in A_\omega$ and every $0\leq\eps\leq\eps_0$.
\end{theorem}

\begin{proof}
For $g\in C^3(\R)$ and $h\neq0$, the centered Taylor remainder gives
\begin{equation}\label{eq:centered-taylor}
 \norm{\frac{g(\cdot+h/2)-g(\cdot-h/2)}h-g'}_\infty
 \leq
 \frac{h^2}{24}\norm{g^{(3)}}_\infty.
\end{equation}
Fix $j\geq0$ and set
\[
 g:=\Delta_jf,
 \qquad
 h:=\eps r_j.
\]
By \cref{def:matched-derivatives,eq:centered-taylor},
\[
 \norm{d_{j,\eps}^\omega f-d_j^\omega f}_\infty
 \leq
 \frac{r_j}{\lambda_j}\frac{\eps^2r_j^2}{24}
 \norm{\partial_x^3\Delta_jf}_\infty.
\]
The estimate \eqref{eq:dyadic-bounds} with $k=3$ gives
\[
 \norm{\partial_x^3\Delta_jf}_\infty
 \leq
 C r_j^{-3}\lambda_j\norm f_{C^\omega}.
\]
Substitution yields
\[
 \norm{d_{j,\eps}^\omega f-d_j^\omega f}_\infty
 \leq
 C\eps^2\norm f_{C^\omega}.
\]
The right-hand side is independent of $j$, and taking the supremum proves \eqref{eq:matched-limit}.
\end{proof}

The exponent in \eqref{eq:matched-limit} cannot be improved as an operator-norm rate.  The same pure-frequency family used in \cref{thm:critical-normalization} detects the leading cubic term of the sine multiplier.

\begin{proposition}[Sharpness of the collision rate]\label{prop:matched-sharpness}
Choose $c_*>0$ with $\varphi(c_*)\neq0$ and set
\[
 f_j(x):=\lambda_j\cos(c_*x/r_j).
\]
Then $\sup_j\norm{f_j}_{C^\omega}<\infty$ and
\begin{equation}\label{eq:sharp-collision}
 \norm{d_{j,\eps}^\omega f_j-d_j^\omega f_j}_\infty
 =
 \abs{\varphi(c_*)}
 \left|\frac{2\sin(c_*\eps/2)}\eps-c_*\right|
 =
 \frac{\abs{\varphi(c_*)}c_*^3}{24}\eps^2+O(\eps^4),
\end{equation}
uniformly in $j$ as $\eps\downarrow0$.
\end{proposition}

\begin{proof}
The uniform $A_\omega$ bound is \eqref{eq:pure-frequency-uniform-modulus}.  Since
\[
 \Delta_jf_j(x)=\varphi(c_*)\lambda_j\cos(c_*x/r_j),
\]
we obtain
\[
 d_j^\omega f_j(x)
 =
 -c_*\varphi(c_*)\sin(c_*x/r_j)
\]
and
\[
 d_{j,\eps}^\omega f_j(x)
 =
 -\varphi(c_*)\frac{2\sin(c_*\eps/2)}\eps\sin(c_*x/r_j).
\]
Taking the uniform norm and expanding $\sin(c_*\eps/2)$ proves \eqref{eq:sharp-collision}.
\end{proof}

The limit alone would produce only a bounded scale observable.  To make it an exact coordinate, the difference multiplier must be inverted on each nonzero frequency band.  The absence of zero frequency from $\supp\varphi$ makes this possible.

For $\eps\geq0$, define
\begin{equation}\label{eq:q-epsilon}
 q_\eps(\zeta)
 :=
 \begin{cases}
 \displaystyle \frac{\eps}{2i\sin(\eps\zeta/2)},&\eps>0,\\[1.1ex]
 \displaystyle \frac1{i\zeta},&\eps=0.
 \end{cases}
\end{equation}
Choose $\eps_0>0$ so small that
\begin{equation}\label{eq:sine-nonzero}
 \sin(\eps\zeta/2)\neq0
 \qquad
 (0<\eps\leq\eps_0,\ \zeta\in\supp\widetilde\varphi).
\end{equation}

The inverse multiplier must be bounded uniformly in both the band index and the collision parameter.

\begin{definition}[Frequency-band integration operator]\label{def:band-integrator}
For $j\geq0$ and $0\leq\eps\leq\eps_0$, define the multiplier $I_{j,\eps}^\omega$ by
\begin{equation}\label{eq:band-integrator}
 \widehat{I_{j,\eps}^\omega z}(\xi)
 :=
 \lambda_j\widetilde\varphi(r_j\xi)q_\eps(r_j\xi)\widehat z(\xi).
\end{equation}
The operator $I_{j,\eps}^\omega$ is the constructed inverse of the matched derivative on the $j$th annulus.
\end{definition}

\begin{lemma}[Uniform inverse-multiplier estimates]\label{lem:inverse-multiplier}
There is $C>0$, independent of $j$ and $\eps\in[0,\eps_0]$, such that
\begin{equation}\label{eq:inverse-multiplier-bound}
 \norm{I_{j,\eps}^\omega z}_\infty
 \leq
 C\lambda_j\norm z_\infty.
\end{equation}
Moreover,
\begin{equation}\label{eq:inverse-collision-stability}
 \norm{(I_{j,\eps}^\omega-I_{j,0}^\omega)z}_\infty
 \leq
 C\eps^2\lambda_j\norm z_\infty.
\end{equation}
\end{lemma}

\begin{proof}
By \eqref{eq:sine-nonzero}, the family
\[
 m_\eps(\zeta):=\widetilde\varphi(\zeta)q_\eps(\zeta)
 \qquad(0\leq\eps\leq\eps_0)
\]
is smooth, compactly supported away from zero, and bounded in every fixed $C^k$ norm.  If $L_\eps:=\mathcal F^{-1}m_\eps$, compactly supported multiplier estimates give
\[
 \sup_{0\leq\eps\leq\eps_0}\norm{L_\eps}_{L^1}<\infty.
\]
The kernel of $I_{j,\eps}^\omega$ is
\[
 \lambda_jr_j^{-1}L_\eps(\cdot/r_j),
\]
whose $L^1$ norm is $\lambda_j\norm{L_\eps}_{L^1}$.  Young's inequality gives \eqref{eq:inverse-multiplier-bound}.

On $\supp\widetilde\varphi$,
\[
 q_\eps(\zeta)
 =
 \frac1{i\zeta}\frac{\eps\zeta/2}{\sin(\eps\zeta/2)}
 =
 q_0(\zeta)+O(\eps^2),
\]
and the same $O(\eps^2)$ bound holds for every fixed $\zeta$-derivative.  Hence the inverse Fourier transform of
\[
 \widetilde\varphi(\zeta)\bigl(q_\eps(\zeta)-q_0(\zeta)\bigr)
\]
has $L^1$ norm $O(\eps^2)$.  Scaling the kernel and multiplying by $\lambda_j$ proves \eqref{eq:inverse-collision-stability}.
\end{proof}

The multiplier product is exactly one on the analyzed annulus; no limiting argument is involved in the next result.

\begin{theorem}[Exact inversion on one band]\label{thm:band-inversion}
Let $0\leq\eps\leq\eps_0$, and let $g_j\in C_b(\R)$ satisfy
\[
 \supp\widehat g_j
 \subset
 \{\xi:r_j\xi\in\supp\varphi\}.
\]
Define
\[
 D_{j,\eps}^\omega g_j
 :=
 \begin{cases}
 \displaystyle
 \frac{g_j(\cdot+\eps r_j/2)-g_j(\cdot-\eps r_j/2)}{\eps\lambda_j},&\eps>0,\\[1.2ex]
 \displaystyle
 \frac{r_j}{\lambda_j}\partial_xg_j,&\eps=0.
 \end{cases}
\]
Then
\begin{equation}\label{eq:exact-band-inversion}
 I_{j,\eps}^\omega D_{j,\eps}^\omega g_j=g_j.
\end{equation}
\end{theorem}

\begin{proof}
For $\eps>0$, the multiplier of $D_{j,\eps}^\omega$ is
\[
 \frac{2i\sin(\eps r_j\xi/2)}{\eps\lambda_j}.
\]
Multiplication with the symbol in \eqref{eq:band-integrator} gives
\[
 \lambda_j\widetilde\varphi(r_j\xi)
 \frac{\eps}{2i\sin(\eps r_j\xi/2)}
 \frac{2i\sin(\eps r_j\xi/2)}{\eps\lambda_j}
 =
 \widetilde\varphi(r_j\xi).
\]
For $\eps=0$, the corresponding product is
\[
 \lambda_j\widetilde\varphi(r_j\xi)
 \frac1{ir_j\xi}
 \frac{ir_j\xi}{\lambda_j}
 =
 \widetilde\varphi(r_j\xi).
\]
Because $\widetilde\varphi=1$ on $\supp\varphi$ and $\widehat g_j$ is supported there, both cases yield \eqref{eq:exact-band-inversion}.
\end{proof}

Differentiation cannot recover the zero frequency.  The low block must therefore be retained as an independent coordinate.

\begin{definition}[Complete analysis and admissible range]\label{def:complete-analysis}
For $0\leq\eps\leq\eps_0$, define
\[
 \cD_{\omega,\eps}f
 :=
 \bigl(\Delta_{-1}f,(d_{j,\eps}^\omega f)_{j\geq0}\bigr),
\]
and let
\[
 \cA_{\omega,\eps}:=\Ran\cD_{\omega,\eps}.
\]
At this point $\cA_{\omega,\eps}$ is only a labelled realized range.  The inverse map is introduced in \cref{thm:analysis-reconstruction} after injectivity of $\cD_{\omega,\eps}$ has been proved.
\end{definition}

The realized-range formulation separates exact reconstruction from the stronger problem of synthesizing an arbitrary bounded tower.  The latter requires additional summability and is treated in \cref{app:free-synthesis}.

\begin{theorem}[Analysis--reconstruction bijection]\label{thm:analysis-reconstruction}
For every $0\leq\eps\leq\eps_0$, the analysis map
\[
 \cD_{\omega,\eps}:A_\omega\longrightarrow\cA_{\omega,\eps}
\]
is injective.  Hence the rule
\begin{equation}\label{eq:inverse-on-range}
 \cI_{\omega,\eps}(\cD_{\omega,\eps}f):=f
 \qquad(f\in A_\omega)
\end{equation}
defines a unique inverse on the realized range.  If
\[
 (u_{-1},z)=\cD_{\omega,\eps}f,
\]
then this inverse is represented by the explicit Littlewood--Paley series
\begin{equation}\label{eq:realized-reconstruction}
 \cI_{\omega,\eps}(u_{-1},z)
 =
 u_{-1}+\sum_{j=0}^\infty I_{j,\eps}^\omega z_j,
\end{equation}
where the series converges uniformly to $f$.  Consequently,
\begin{equation}\label{eq:analysis-reconstruction-identities}
 \cI_{\omega,\eps}\cD_{\omega,\eps}=\id_{A_\omega},
 \qquad
 \cD_{\omega,\eps}\cI_{\omega,\eps}=\id_{\cA_{\omega,\eps}}.
\end{equation}
No Dini hypothesis is used.
\end{theorem}

\begin{proof}
Suppose
\[
 \cD_{\omega,\eps}f=\cD_{\omega,\eps}g.
\]
The low-frequency coordinates satisfy
\[
 \Delta_{-1}(f-g)=0,
\]
and the high-frequency coordinates satisfy
\[
 d_{j,\eps}^\omega(f-g)=0
 \qquad(j\geq0).
\]
Applying the band inverse from \cref{thm:band-inversion} gives
\[
 \Delta_j(f-g)
 =I_{j,\eps}^\omega d_{j,\eps}^\omega(f-g)
 =0
 \qquad(j\geq0).
\]
The uniform Littlewood--Paley reconstruction \eqref{eq:lp-uniform-reconstruction} therefore yields
\[
 f-g
 =\Delta_{-1}(f-g)+\sum_{j\geq0}\Delta_j(f-g)
 =0.
\]
Thus $\cD_{\omega,\eps}$ is injective, and \eqref{eq:inverse-on-range} is well defined.

Now let $(u_{-1},z)=\cD_{\omega,\eps}f$.  Then
\[
 u_{-1}=\Delta_{-1}f,
 \qquad
 z_j=d_{j,\eps}^\omega f.
\]
By \cref{thm:band-inversion},
\[
 I_{j,\eps}^\omega z_j=\Delta_jf.
\]
Hence
\[
 u_{-1}+\sum_{j=0}^\infty I_{j,\eps}^\omega z_j
 =\Delta_{-1}f+\sum_{j=0}^\infty\Delta_jf
 =f
\]
uniformly.  This proves \eqref{eq:realized-reconstruction}.  The two identities in \eqref{eq:analysis-reconstruction-identities} now follow directly from the definition of the inverse on the range.
\end{proof}

Exact reconstruction transports the pointwise algebra without introducing another analytic estimate.  We state only the consequence needed to interpret the tower as a coordinate system.

\begin{corollary}[Transported algebra on the realized range]\label{cor:transported-algebra}
For $a,b\in\cA_{\omega,\eps}$, define
\[
 a\star_{\omega,\eps}b
 :=
 \cD_{\omega,\eps}
 \bigl(\cI_{\omega,\eps}a\,\cI_{\omega,\eps}b\bigr),
\]
and define
\[
 \norm a_{\omega,\eps}^{\mathrm{tr}}
 :=
 \norm{\cI_{\omega,\eps}a}_{C^\omega}.
\]
Then $(\cA_{\omega,\eps},\star_{\omega,\eps})$ is a commutative unital Banach algebra and
\[
 \cD_{\omega,\eps}:A_\omega\longrightarrow\cA_{\omega,\eps}
\]
is an isometric algebra isomorphism for the transported norm.
\end{corollary}

\begin{proof}
Every algebraic identity is obtained by conjugating pointwise multiplication through the bijection in \cref{thm:analysis-reconstruction}.  For example,
\[
 (a\star_{\omega,\eps}b)\star_{\omega,\eps}c
 =
 \cD_{\omega,\eps}
 \bigl(\cI_{\omega,\eps}a\,\cI_{\omega,\eps}b\,\cI_{\omega,\eps}c\bigr),
\]
and the same expression equals $a\star_{\omega,\eps}(b\star_{\omega,\eps}c)$.  Completeness and submultiplicativity follow from the corresponding properties of $A_\omega$ under the transported norm.
\end{proof}

The ordinary derivative is not part of the scale coordinate.  It appears only when the weighted tower can be summed in the classical function space.

\begin{corollary}[Classical derivative as scalar collapse]\label{cor:classical-collapse}
Let
\[
 S_Nf:=\Delta_{-1}f+\sum_{j=0}^N\Delta_jf.
\]
Then
\begin{equation}\label{eq:weighted-collapse}
 \partial_xS_Nf
 =
 \partial_x\Delta_{-1}f
 +
 \sum_{j=0}^N\frac{\lambda_j}{r_j}d_j^\omega f.
\end{equation}
The following are equivalent:
\begin{enumerate}[label=(\roman*),leftmargin=2.4em]
\item $\partial_xS_Nf$ converges uniformly on $\R$;
\item $f\in C_b^1(\R)$ and $f'$ is bounded and uniformly continuous.
\end{enumerate}
In this case, the limit in \eqref{eq:weighted-collapse} is $f'$.
\end{corollary}

\begin{proof}
Identity \eqref{eq:weighted-collapse} follows from
\[
 \partial_x\Delta_jf=\frac{\lambda_j}{r_j}d_j^\omega f.
\]
If $S_Nf\to f$ and $\partial_xS_Nf\to g$ uniformly, the uniform derivative theorem gives
\[
 f\in C^1(\R),
 \qquad
 f'=g.
\]
The uniform limit $g$ is bounded and uniformly continuous.  Conversely, if $f'\in C_{bu}(\R)$, the low-pass multipliers commute with $\partial_x$, and
\[
 \partial_xS_Nf=S_N(f')\longrightarrow f'
\]
uniformly by the approximate-identity property.
\end{proof}
 \section{Corona exactification of the scale derivative}\label{sec:corona}

The complete tower from \cref{sec:matched-reconstruction} retains the function, but each finite-frequency operator $d_j^\omega$ has a nonzero Leibniz defect.  The defect estimate below shows that this failure is smaller than the critical amplitude and therefore disappears at infinite scale.  The appropriate target is the quotient of bounded towers by towers converging to zero.

\begin{definition}[Scale corona algebra]\label{def:scale-corona}
Define
\[
 \cB:=\ell^\infty(\Nzero;C_b(\R))
\]
with coordinatewise multiplication, and define the closed ideal
\[
 \cB_0:=c_0(\Nzero;C_b(\R))
 =
 \left\{(u_j)_{j\geq0}:\norm{u_j}_\infty\to0\right\}.
\]
The \emph{scale corona algebra} is
\[
 \cQ:=\cB/\cB_0.
\]
For a bounded tower $u=(u_j)$,
\begin{equation}\label{eq:corona-norm}
 \norm{[u]}_\cQ
 =
 \limsup_{j\to\infty}\norm{u_j}_\infty.
\end{equation}
The algebra $A_\omega$ acts on $\cB$ and $\cQ$ through coordinatewise multiplication by the constant sequence
\[
 \iota(f):=[(f,f,\ldots)].
\]
\end{definition}

The dyadic bound \eqref{eq:dyadic-bounds} makes the raw derivative tower a bounded map into $\cB$.

\begin{definition}[Raw tower and corona derivative]\label{def:corona-derivative}
Define
\[
 d^\omega:A_\omega\longrightarrow\cB,
 \qquad
 d^\omega f:=(d_j^\omega f)_{j\geq0},
\]
and let $q:\cB\to\cQ$ be the quotient map.  The \emph{corona derivative} is
\[
 \delta_\omega:=q\circ d^\omega,
 \qquad
 \delta_\omega f=[(d_j^\omega f)_{j\geq0}].
\]
Its role is to retain precisely the nonvanishing critical class of the scale derivative.
\end{definition}

The first obstruction to exactness is the product rule.  Its exact kernel contains two increments and therefore gains one factor of $\lambda_j$.

\begin{lemma}[Leibniz-defect kernel]\label{lem:product-defect}
For $f,g\in A_\omega$, define
\[
 B_j^\omega(f,g)
 :=
 d_j^\omega(fg)-f\,d_j^\omega g-g\,d_j^\omega f.
\]
Then
\begin{equation}\label{eq:product-defect-kernel}
 B_j^\omega(f,g)(x)
 =
 \frac{r_j}{\lambda_j}
 \int_\R K_j'(y)
 \bigl(f(x-y)-f(x)\bigr)
 \bigl(g(x-y)-g(x)\bigr)
 \,\dd y,
\end{equation}
and
\begin{equation}\label{eq:product-defect-bound}
 \norm{B_j^\omega(f,g)}_\infty
 \leq
 C\lambda_j\norm f_{C^\omega}\norm g_{C^\omega}.
\end{equation}
\end{lemma}

\begin{proof}
Since $\int_\R K_j'=0$,
\[
 d_j^\omega h(x)
 =
 \frac{r_j}{\lambda_j}
 \int_\R K_j'(y)\bigl(h(x-y)-h(x)\bigr)\,\dd y.
\]
For $h=fg$, the increment identity
\begin{align*}
 f(x-y)g(x-y)-f(x)g(x)
 ={}&f(x)\bigl(g(x-y)-g(x)\bigr)\\
 &+g(x)\bigl(f(x-y)-f(x)\bigr)\\
 &+\bigl(f(x-y)-f(x)\bigr)
   \bigl(g(x-y)-g(x)\bigr)
\end{align*}
shows that subtracting the two Leibniz terms leaves \eqref{eq:product-defect-kernel}.

By \cref{lem:global-increment},
\[
 \abs{f(x-y)-f(x)}
 \leq
 C_\omega\norm f_{C^\omega}\omega(\abs y),
\]
and the same estimate holds for $g$.  Hence
\begin{align*}
 \norm{B_j^\omega(f,g)}_\infty
 &\leq
 C\frac{r_j}{\lambda_j}\norm f_{C^\omega}\norm g_{C^\omega}
 \int_\R\abs{K_j'(y)}\omega(\abs y)^2\,\dd y\\
 &=
 \frac{C}{\lambda_j}\norm f_{C^\omega}\norm g_{C^\omega}
 \int_\R\abs{K'(z)}\omega(r_j\abs z)^2\,\dd z.
\end{align*}
By \cref{lem:concave-scale},
\[
 \omega(r_j\abs z)^2
 \leq
 (1+\abs z)^2\lambda_j^2.
\]
Because $K'$ is Schwartz,
\[
 \int_\R\abs{K'(z)}(1+\abs z)^2\,\dd z<\infty.
\]
Substitution gives \eqref{eq:product-defect-bound}.
\end{proof}

The smooth chain rule has the same structure: the first-order Taylor term is exact and the quadratic remainder contributes two increments.

\begin{lemma}[Chain-rule defect]\label{lem:chain-defect}
Let $I\subset\R$ be open, let $F\in C^2(I)$, and let $f\in A_\omega$ satisfy $f(\R)\Subset I$.  Put
\[
 K_f:=\co(f(\R))\Subset I
\]
and define
\[
 E_{j,F}^\omega(f)
 :=
 d_j^\omega(F(f))-F'(f)d_j^\omega f.
\]
Then
\begin{align}
 E_{j,F}^\omega(f)(x)
 ={}&
 \frac{r_j}{\lambda_j}
 \int_\R K_j'(y)
 \Bigl[
 F(f(x-y))-F(f(x))
 \nonumber\\[-2mm]
 &\hspace{42mm}
 -F'(f(x))\bigl(f(x-y)-f(x)\bigr)
 \Bigr]\,\dd y,
 \label{eq:chain-defect-kernel}
\end{align}
and
\begin{equation}\label{eq:chain-defect-bound}
 \norm{E_{j,F}^\omega(f)}_\infty
 \leq
 C\lambda_j
 \norm{F''}_{L^\infty(K_f)}
 \norm f_{C^\omega}^2.
\end{equation}
\end{lemma}

\begin{proof}
The increment representation for $d_j^\omega$ gives \eqref{eq:chain-defect-kernel}.  Taylor's theorem yields
\[
 \abs{F(a)-F(b)-F'(b)(a-b)}
 \leq
 \frac12\norm{F''}_{L^\infty(K_f)}\abs{a-b}^2
 \qquad(a,b\in K_f).
\]
Set $a=f(x-y)$ and $b=f(x)$.  The remaining kernel integral is the one estimated in the proof of \cref{lem:product-defect}; this gives \eqref{eq:chain-defect-bound}.
\end{proof}

The factors $\lambda_j$ in \eqref{eq:product-defect-bound} and \eqref{eq:chain-defect-bound} tend to zero.  Thus the quotient in \cref{def:scale-corona} turns the asymptotically exact finite-frequency rules into exact identities.

\begin{theorem}[Exact corona calculus]\label{thm:exact-corona-calculus}
The map
\[
 \delta_\omega:A_\omega\longrightarrow\cQ
\]
is continuous and satisfies
\begin{equation}\label{eq:corona-leibniz}
 \delta_\omega(fg)
 =
 \iota(f)\delta_\omega g+
 \iota(g)\delta_\omega f
 \qquad(f,g\in A_\omega).
\end{equation}
If $I\subset\R$ is open, $F\in C^2(I)$, and $f(\R)\Subset I$, then
\begin{equation}\label{eq:corona-chain}
 \delta_\omega(F(f))
 =
 \iota(F'(f))\delta_\omega f.
\end{equation}
\end{theorem}

\begin{proof}
By \cref{prop:dyadic-bounds},
\[
 \norm{d_j^\omega f}_\infty
 =
 \frac{r_j}{\lambda_j}\norm{\partial_x\Delta_jf}_\infty
 \leq
 C\norm f_{C^\omega}.
\]
Equation \eqref{eq:corona-norm} therefore gives
\[
 \norm{\delta_\omega f}_\cQ
 \leq
 C\norm f_{C^\omega}.
\]
By \eqref{eq:product-defect-bound} and $\lambda_j\to0$,
\[
 (B_j^\omega(f,g))_{j\geq0}\in\cB_0.
\]
Applying the quotient map to the defining identity for $B_j^\omega(f,g)$ gives \eqref{eq:corona-leibniz}.  Likewise, \eqref{eq:chain-defect-bound} implies
\[
 (E_{j,F}^\omega(f))_{j\geq0}\in\cB_0,
\]
and its quotient class is exactly the difference between the two sides of \eqref{eq:corona-chain}.
\end{proof}

Exactness arises from annihilating all vanishing towers, not from a special choice of a downstream target.  The next theorem states the precise universal class: bounded symmetric $A_\omega$-module maps defined on the raw tower space and vanishing on $\cB_0$.

\begin{theorem}[Universal bounded exactification]\label{thm:universal-exactification}
Let $M$ be a Banach symmetric $A_\omega$-module, and let
\[
 T:\cB\longrightarrow M
\]
be a bounded $A_\omega$-module map satisfying
\[
 \cB_0\subset\Ker T.
\]
Then there is a unique bounded $A_\omega$-module map
\[
 \overline T:\cQ\longrightarrow M
\]
such that
\begin{equation}\label{eq:universal-factorization}
 T=\overline T\circ q.
\end{equation}
The induced map
\[
 D_T:=T\circ d^\omega
 =
 \overline T\circ\delta_\omega:A_\omega\longrightarrow M
\]
is a continuous derivation and satisfies the $C^2$ chain rule.  Conversely, every map obtained from $d^\omega$ by such a bounded module map factors uniquely through $\delta_\omega$.
\end{theorem}

\begin{proof}
The subspace $\cB_0$ is a closed $A_\omega$-submodule of $\cB$.  Since $\cB_0\subset\Ker T$, the Banach-space quotient property gives a unique bounded linear map $\overline T$ satisfying \eqref{eq:universal-factorization}.  For $f\in A_\omega$ and $u\in\cB$,
\[
 \overline T\bigl(\iota(f)q(u)\bigr)
 =
 \overline T(q(fu))
 =
 T(fu)
 =
 fT(u)
 =
 f\overline T(q(u)),
\]
so $\overline T$ is an $A_\omega$-module map.

By \cref{lem:product-defect},
\[
 d^\omega(fg)-f\,d^\omega g-g\,d^\omega f
 =
 (B_j^\omega(f,g))_{j\geq0}
 \in\cB_0.
\]
Applying $T$ gives
\[
 D_T(fg)=fD_Tg+gD_Tf.
\]
Similarly, \cref{lem:chain-defect} gives
\[
 d^\omega(F(f))-F'(f)d^\omega f
 \in\cB_0,
\]
and application of $T$ gives
\[
 D_T(F(f))=F'(f)D_Tf.
\]
Continuity follows from the boundedness of $d^\omega$ and $T$.  The converse and uniqueness are exactly the quotient factorization already proved.
\end{proof}

The corona quotient forgets functions whose block amplitudes are asymptotically smaller than the selected modulus.  Exact band inversion supplies the reverse inequality needed to identify this kernel.

\begin{definition}[Little scale space]\label{def:little-scale-space}
Define
\[
 c_\Delta^\omega
 :=
 \left\{
 f\in A_\omega:
 \lambda_j^{-1}\norm{\Delta_jf}_\infty\longrightarrow0
 \right\}.
\]
For $\omega(r)=r^\alpha$, $0<\alpha<1$, this is the standard high-frequency description of the little H\"older--Zygmund space.
\end{definition}

\begin{theorem}[Kernel and norm of the corona derivative]\label{thm:corona-kernel}
For every $f\in A_\omega$ and every $j\geq0$,
\begin{equation}\label{eq:block-derivative-equivalence}
 C^{-1}\lambda_j^{-1}\norm{\Delta_jf}_\infty
 \leq
 \norm{d_j^\omega f}_\infty
 \leq
 C\lambda_j^{-1}\norm{\Delta_jf}_\infty.
\end{equation}
Consequently,
\begin{equation}\label{eq:corona-kernel}
 \Ker\delta_\omega=c_\Delta^\omega,
\end{equation}
and
\begin{equation}\label{eq:corona-norm-equivalence}
 \norm{\delta_\omega f}_\cQ
 \asymp
 \limsup_{j\to\infty}
 \lambda_j^{-1}\norm{\Delta_jf}_\infty.
\end{equation}
\end{theorem}

\begin{proof}
Bernstein's inequality gives
\[
 \norm{d_j^\omega f}_\infty
 =
 \frac{r_j}{\lambda_j}\norm{\partial_x\Delta_jf}_\infty
 \leq
 C\lambda_j^{-1}\norm{\Delta_jf}_\infty.
\]
Conversely, \cref{thm:band-inversion} at $\eps=0$ gives
\[
 \Delta_jf=I_{j,0}^\omega d_j^\omega f.
\]
Using \eqref{eq:inverse-multiplier-bound},
\[
 \norm{\Delta_jf}_\infty
 \leq
 C\lambda_j\norm{d_j^\omega f}_\infty.
\]
These are the two inequalities in \eqref{eq:block-derivative-equivalence}.  Taking limits proves \eqref{eq:corona-kernel}; taking limsups and using \eqref{eq:corona-norm} proves \eqref{eq:corona-norm-equivalence}.
\end{proof}

The distinction between the complete tower, the corona class, and the scalar derivative is visible on an explicit critical H\"older function.

\begin{example}[Nonzero corona class without scalar collapse]\label{ex:lacunary-corona}
Fix $0<\alpha<1$, put $\omega(r)=r^\alpha$, and choose
\[
 \frac43<c_*<\frac32.
\]
For the cutoffs in \cref{def:lp-pair},
\[
 \varphi(c_*)=1,
 \qquad
 \varphi(2^kc_*)=0
 \quad(k\in\Z\setminus\{0\}).
\]
Define
\[
 f_{\alpha,c_*}(x)
 :=
 \sum_{j=0}^\infty2^{-\alpha j}\cos(c_*2^jx).
\]
The standard lacunary estimate gives $f_{\alpha,c_*}\in C_b^\alpha(\R)$, while the plateau property gives
\[
 \Delta_jf_{\alpha,c_*}
 =
 2^{-\alpha j}\cos(c_*2^jx).
\]
Hence
\[
 d_j^\alpha f_{\alpha,c_*}
 =
 -c_*\sin(c_*2^jx),
 \qquad
 \norm{d_j^\alpha f_{\alpha,c_*}}_\infty=c_*.
\]
Therefore
\[
 \delta_\alpha f_{\alpha,c_*}\neq0.
\]
On the other hand,
\[
 \norm{\partial_xS_Nf_{\alpha,c_*}
       -\partial_xS_{N-1}f_{\alpha,c_*}}_\infty
 =
 \norm{\partial_x\Delta_Nf_{\alpha,c_*}}_\infty
 =
 c_*2^{(1-\alpha)N}.
\]
Thus $(\partial_xS_Nf_{\alpha,c_*})_N$ is not Cauchy.  The critical scale differential is nonzero although the classical derivative collapse fails.
\end{example}
 \section{Higher multiplicative defects}\label{sec:higher-defects}

The first product defect in \cref{lem:product-defect} contains two increments.  Repeated inclusion--exclusion produces one increment for each argument, so the order-$n$ defect carries the scale factor $\lambda_j^{n-1}$.  Dividing by that factor reveals the full leading nonlinear symbol at critical scale.

\begin{definition}[Higher multiplicative defect]\label{def:higher-defect}
Let $D:A_\omega\to C_b(\R)$ be linear and satisfy $D(1)=0$.  For $n\geq1$, define
\begin{equation}\label{eq:higher-defect-definition}
 \Phi_D^{(n)}(f_1,\ldots,f_n)
 :=
 \sum_{S\subseteq[n]}
 (-1)^{n-\abs S}
 D\!\left(\prod_{i\in S}f_i\right)
 \prod_{i\notin S}f_i,
 \qquad
 [n]:=\{1,\ldots,n\}.
\end{equation}
The empty product is $1$.  For $D=d_j^\omega$, write
\[
 \Phi_{j,n}^\omega:=\Phi_{d_j^\omega}^{(n)}.
\]
\end{definition}

The inclusion--exclusion sum factorizes exactly at the kernel level.

\begin{proposition}[Exact higher-defect kernel]\label{prop:higher-defect-kernel}
For $f_1,\ldots,f_n\in A_\omega$,
\begin{equation}\label{eq:higher-defect-kernel}
 \Phi_{j,n}^\omega(f_1,\ldots,f_n)(x)
 =
 \frac{r_j}{\lambda_j}
 \int_\R K_j'(y)
 \prod_{a=1}^n\bigl(f_a(x-y)-f_a(x)\bigr)
 \,\dd y.
\end{equation}
Consequently,
\begin{equation}\label{eq:higher-defect-estimate}
 \norm{\Phi_{j,n}^\omega(f_1,\ldots,f_n)}_\infty
 \leq
 C_n\lambda_j^{n-1}
 \prod_{a=1}^n\norm{f_a}_{C^\omega}.
\end{equation}
\end{proposition}

\begin{proof}
For fixed $x,y$, set
\[
 a_i:=f_i(x-y),
 \qquad
 b_i:=f_i(x).
\]
The elementary inclusion--exclusion identity is
\begin{equation}\label{eq:inclusion-exclusion-factorization}
 \sum_{S\subseteq[n]}
 (-1)^{n-\abs S}
 \left(\prod_{i\in S}a_i\right)
 \left(\prod_{i\notin S}b_i\right)
 =
 \prod_{i=1}^n(a_i-b_i).
\end{equation}
Insert the increment representation
\[
 d_j^\omega h(x)
 =
 \frac{r_j}{\lambda_j}
 \int_\R K_j'(y)\bigl(h(x-y)-h(x)\bigr)\,\dd y
\]
into \eqref{eq:higher-defect-definition}.  The terms involving $h(x)$ cancel because $d_j^\omega(1)=0$, and \eqref{eq:inclusion-exclusion-factorization} gives \eqref{eq:higher-defect-kernel}.

By \cref{lem:global-increment},
\[
 \prod_{a=1}^n\abs{f_a(x-y)-f_a(x)}
 \leq
 C\omega(\abs y)^n
 \prod_{a=1}^n\norm{f_a}_{C^\omega}.
\]
After the change of variables $y=r_jz$,
\begin{align*}
 \norm{\Phi_{j,n}^\omega(f_1,\ldots,f_n)}_\infty
 &\leq
 \frac{C}{\lambda_j}
 \prod_{a=1}^n\norm{f_a}_{C^\omega}
 \int_\R\abs{K'(z)}\omega(r_j\abs z)^n\,\dd z\\
 &\leq
 C_n\lambda_j^{n-1}
 \prod_{a=1}^n\norm{f_a}_{C^\omega},
\end{align*}
where the final step uses \cref{lem:concave-scale} and the Schwartz decay of $K'$.
\end{proof}

The estimate \eqref{eq:higher-defect-estimate} supplies the natural normalization at which the order-$n$ defect remains bounded.

\begin{definition}[Normalized higher symbol]\label{def:normalized-higher-symbol}
For $n\geq1$, define
\[
 K_{j,n}^\omega
 :=
 \lambda_j^{1-n}\Phi_{j,n}^\omega.
\]
By \eqref{eq:higher-defect-estimate}, the tower
\[
 \bigl(K_{j,n}^\omega(f_1,\ldots,f_n)\bigr)_{j\geq0}
\]
belongs to $\cB$.  Define
\[
 \kappa_n^\omega(f_1,\ldots,f_n)
 :=
 \left[
 \bigl(K_{j,n}^\omega(f_1,\ldots,f_n)\bigr)_{j\geq0}
 \right]
 \in\cQ.
\]
For $n=1$, one has $\kappa_1^\omega=\delta_\omega$.
\end{definition}

The next recursion shows that the remaining failure of the Leibniz rule for $K_{j,n}^\omega$ contains one additional increment and hence one additional factor $\lambda_j$.

\begin{lemma}[Defect recursion]\label{lem:defect-recursion}
For $n\geq1$ and $f_1,\ldots,f_{n-1},g,h\in A_\omega$,
\begin{align}
 &\Phi_{j,n}^\omega(f_1,\ldots,f_{n-1},gh)
 -g\Phi_{j,n}^\omega(f_1,\ldots,f_{n-1},h)
 \nonumber\\
 &\hspace{28mm}
 -h\Phi_{j,n}^\omega(f_1,\ldots,f_{n-1},g)
 \nonumber\\
 &\hspace{52mm}
 =
 \Phi_{j,n+1}^\omega(f_1,\ldots,f_{n-1},g,h).
 \label{eq:defect-recursion}
\end{align}
Equivalently,
\begin{align}
 &K_{j,n}^\omega(f_1,\ldots,f_{n-1},gh)
 -gK_{j,n}^\omega(f_1,\ldots,f_{n-1},h)
 \nonumber\\
 &\hspace{28mm}
 -hK_{j,n}^\omega(f_1,\ldots,f_{n-1},g)
 \nonumber\\
 &\hspace{52mm}
 =
 \lambda_jK_{j,n+1}^\omega(f_1,\ldots,f_{n-1},g,h).
 \label{eq:normalized-defect-recursion}
\end{align}
\end{lemma}

\begin{proof}
Write
\[
 \nabla_yf(x):=f(x-y)-f(x).
\]
By \eqref{eq:higher-defect-kernel}, the left-hand side of \eqref{eq:defect-recursion} is
\[
 \frac{r_j}{\lambda_j}
 \int_\R K_j'(y)
 \prod_{a=1}^{n-1}\nabla_yf_a(x)
 \Bigl[
 \nabla_y(gh)(x)-g(x)\nabla_yh(x)-h(x)\nabla_yg(x)
 \Bigr]\,\dd y.
\]
The bracket equals
\[
 \nabla_yg(x)\nabla_yh(x).
\]
Substitution gives precisely the kernel for $\Phi_{j,n+1}^\omega$.  Multiplying by $\lambda_j^{1-n}$ gives \eqref{eq:normalized-defect-recursion}.
\end{proof}

After passage to $\cQ$, the right-hand side of \eqref{eq:normalized-defect-recursion} vanishes.  Thus every normalized symbol becomes a derivation in each argument.

\begin{theorem}[Higher corona symbols are multiderivations]\label{thm:higher-multiderivations}
For every $n\geq1$, the map
\[
 \kappa_n^\omega:A_\omega^n\longrightarrow\cQ
\]
is continuous, symmetric, and a derivation in each argument.  In particular,
\begin{align*}
 \kappa_n^\omega(f_1,\ldots,f_{n-1},gh)
 ={}&
 \iota(g)\kappa_n^\omega(f_1,\ldots,f_{n-1},h)\\
 &+
 \iota(h)\kappa_n^\omega(f_1,\ldots,f_{n-1},g).
\end{align*}
Consequently, $\kappa_n^\omega$ is a Hochschild $n$-cocycle \cite{Hochschild1945} with values in the symmetric $A_\omega$-module $\cQ$.
\end{theorem}

\begin{proof}
Continuity follows from \eqref{eq:higher-defect-estimate}, and symmetry follows from \eqref{eq:higher-defect-kernel}.  By \eqref{eq:normalized-defect-recursion}, the failure of the Leibniz rule in the last argument is represented by
\[
 \bigl(
 \lambda_jK_{j,n+1}^\omega(f_1,\ldots,f_{n-1},g,h)
 \bigr)_{j\geq0}.
\]
The tower $K_{j,n+1}^\omega(f_1,\ldots,g,h)$ is bounded by \eqref{eq:higher-defect-estimate}, while $\lambda_j\to0$.  Hence the displayed sequence lies in $\cB_0$ and has zero class in $\cQ$.  Symmetry gives the derivation property in every argument.

For a commutative algebra with symmetric coefficients, the Hochschild boundary of a multiderivation is the sum of its Leibniz defects in the individual arguments.  All these defects vanish, so
\[
 b\kappa_n^\omega=0.
\]
\end{proof}

The quotient symbols are the leading terms of exact finite-frequency identities.  We retain the two formulas needed to reconstruct products and smooth compositions without introducing the later sampling-jet and signature branches of the original development.

\begin{theorem}[Exact product and finite chain expansions]\label{thm:finite-expansions}
Let $N\geq1$ and $f_1,\ldots,f_N\in A_\omega$.  Then, for every $j\geq0$,
\begin{equation}\label{eq:exact-product-expansion}
 d_j^\omega\!\left(\prod_{i=1}^Nf_i\right)
 =
 \sum_{\varnothing\neq S\subseteq[N]}
 \lambda_j^{\abs S-1}
 K_{j,\abs S}^\omega((f_i)_{i\in S})
 \prod_{i\notin S}f_i.
\end{equation}

Let $M\geq2$, let $I\subset\R$ be open, let $F\in C^M(I)$, and let $f\in A_\omega$ satisfy $f(\R)\Subset I$.  With $K_f:=\co(f(\R))$, one has
\begin{equation}\label{eq:finite-chain-expansion}
 d_j^\omega(F(f))
 =
 \sum_{n=1}^{M-1}
 \frac{\lambda_j^{n-1}}{n!}
 F^{(n)}(f)K_{j,n}^\omega(f,\ldots,f)
 +R_{j,M}(F,f),
\end{equation}
where
\begin{equation}\label{eq:finite-chain-remainder}
 \norm{R_{j,M}(F,f)}_\infty
 \leq
 C_M\lambda_j^{M-1}
 \norm{F^{(M)}}_{L^\infty(K_f)}
 \norm f_{C^\omega}^{M}.
\end{equation}
\end{theorem}

\begin{proof}
For fixed $x,y$,
\begin{equation}\label{eq:finite-product-increment}
 \prod_{i=1}^Nf_i(x-y)-\prod_{i=1}^Nf_i(x)
 =
 \sum_{\varnothing\neq S\subseteq[N]}
 \left(\prod_{i\in S}\nabla_yf_i(x)\right)
 \left(\prod_{i\notin S}f_i(x)\right).
\end{equation}
Multiply \eqref{eq:finite-product-increment} by $(r_j/\lambda_j)K_j'(y)$ and integrate.  By \eqref{eq:higher-defect-kernel}, the term indexed by $S$ becomes
\[
 \Phi_{j,\abs S}^\omega((f_i)_{i\in S})
 \prod_{i\notin S}f_i
 =
 \lambda_j^{\abs S-1}
 K_{j,\abs S}^\omega((f_i)_{i\in S})
 \prod_{i\notin S}f_i.
\]
Summation gives \eqref{eq:exact-product-expansion}.

For the chain expansion, Taylor's theorem gives, for $a,b\in K_f$,
\[
 F(a)-F(b)
 =
 \sum_{n=1}^{M-1}
 \frac{F^{(n)}(b)}{n!}(a-b)^n
 +\mathscr R_M(a,b),
\]
with
\[
 \abs{\mathscr R_M(a,b)}
 \leq
 \frac{\norm{F^{(M)}}_{L^\infty(K_f)}}{M!}\abs{a-b}^M.
\]
Set
\[
 a:=f(x-y),
 \qquad
 b:=f(x).
\]
Multiplication by $(r_j/\lambda_j)K_j'(y)$ and integration identify the $n$th term with
\[
 \frac{\lambda_j^{n-1}}{n!}
 F^{(n)}(f(x))K_{j,n}^\omega(f,\ldots,f)(x).
\]
For the remainder, \cref{lem:global-increment,lem:concave-scale} give
\begin{align*}
 \norm{R_{j,M}(F,f)}_\infty
 &\leq
 C_M\frac{r_j}{\lambda_j}
 \norm{F^{(M)}}_{L^\infty(K_f)}
 \norm f_{C^\omega}^M
 \int_\R\abs{K_j'(y)}\omega(\abs y)^M\,\dd y\\
 &\leq
 C_M\lambda_j^{M-1}
 \norm{F^{(M)}}_{L^\infty(K_f)}
 \norm f_{C^\omega}^M.
\end{align*}
This is \eqref{eq:finite-chain-remainder}.
\end{proof}
 \section{Periodic exact coordinates and amplitude reinsertion}\label{sec:periodic-interface}

The commutative construction has so far been formulated on $\R$, whereas Fourier normal ordering will be applied to periodic paths.  The bridge has two parts.  First, the compactly supported multipliers of Sections \ref{sec:modulus-lp}--\ref{sec:matched-reconstruction} periodize without changing their uniform $L^\infty$ bounds.  Second, in the H\"older gauge $\lambda_j=r_j^\alpha$, multiplying the normalized derivative coordinate by $\lambda_j$ recovers the unweighted quadrature channel $r_j\partial_x\Delta_jf$.

Let
\[
 \T:=\R/(2\pi\Z)
\]
and let $d_\T$ be its geodesic distance.  For $0<\alpha<1$, define
\[
 C^\alpha(\T)
 :=
 \left\{f\in C(\T):
 \norm f_{C^\alpha(\T)}
 :=
 \norm f_\infty+
 \sup_{x\neq y}
 \frac{\abs{f(x)-f(y)}}{d_\T(x,y)^\alpha}
 <\infty
 \right\}.
\]
We use the periodic Littlewood--Paley multipliers obtained by restricting the symbols in \cref{def:lp-pair} to integer frequencies:
\[
 \widehat{\Delta_{-1}^{\T}f}(k)=\chi(k)\widehat f(k),
 \qquad
 \widehat{\Delta_j^{\T}f}(k)=\varphi(r_jk)\widehat f(k).
\]

The following standard periodization statement supplies the operator bounds required by the inverse multipliers.

\begin{lemma}[Periodization of compactly supported multipliers]\label{lem:periodization}
Let $m\in C_c^\infty(\R)$ and let $K=\mathcal F^{-1}m\in L^1(\R)$.  The multiplier on $\T$ with symbol $(m(k))_{k\in\Z}$ has convolution kernel
\[
 K_{\mathrm{per}}(x)
 :=
 \sum_{\ell\in\Z}K(x+2\pi\ell)
\]
up to the fixed Fourier-normalization constant, and
\begin{equation}\label{eq:periodized-kernel-bound}
 \norm{K_{\mathrm{per}}}_{L^1(\T)}
 \leq
 C\norm K_{L^1(\R)}.
\end{equation}
Consequently, every uniformly $L^1$-bounded compactly supported real-line multiplier family induces a uniformly bounded family on $C(\T)$.
\end{lemma}

\begin{proof}
Poisson periodization identifies the Fourier coefficient of $K_{\mathrm{per}}$ at $k\in\Z$ with $m(k)$.  Tonelli's theorem gives
\[
 \int_0^{2\pi}
 \sum_{\ell\in\Z}\abs{K(x+2\pi\ell)}\,\dd x
 =
 \int_\R\abs{K(y)}\,\dd y.
\]
Young's inequality on $\T$ then gives the multiplier bound.
\end{proof}

We now specialize to the H\"older modulus
\[
 \omega_\alpha(r):=r^\alpha,
 \qquad
 \lambda_j:=r_j^\alpha=2^{-\alpha j}.
\]
The periodic scale derivative and its band operator are defined by the same Fourier symbols as on $\R$.

\begin{definition}[Periodic exact scale coordinate]\label{def:periodic-scale-coordinate}
For $f\in C^\alpha(\T)$ and $j\geq0$, define
\[
 d_j^{\alpha,\T}f
 :=
 \frac{r_j}{\lambda_j}\partial_x\Delta_j^{\T}f.
\]
Define the periodic band operator by
\begin{equation}\label{eq:periodic-band-inverse}
 \widehat{I_{j,0}^{\alpha,\T}z}(k)
 :=
 \lambda_j\widetilde\varphi(r_jk)\frac1{ir_jk}\widehat z(k),
 \qquad k\in\Z,
\end{equation}
where the symbol is set to zero at $k=0$; this value is immaterial because $\widetilde\varphi$ is supported away from zero.  The complete periodic analysis is
\[
 \cD_{\alpha,0}^{\T}f
 :=
 \bigl(\Delta_{-1}^{\T}f,(d_j^{\alpha,\T}f)_{j\geq0}\bigr),
\]
with labelled realized range
\[
 \cA_{\alpha,0}^{\T}:=\Ran\cD_{\alpha,0}^{\T}.
\]
At this point no inverse or transported norm is used; both are defined only after injectivity is proved in \cref{prop:periodic-reconstruction}.
\end{definition}

The periodic block bounds needed below are the torus counterpart of \cref{prop:dyadic-bounds}.  We state them explicitly so that the interface estimate does not rely on an unnamed periodic analogue.

\begin{lemma}[Periodic dyadic bounds]\label{lem:periodic-dyadic-bounds}
For every integer $q\geq0$, there is $C_{q,\alpha}>0$ such that
\begin{equation}\label{eq:periodic-dyadic-bounds}
 \norm{\partial_x^q\Delta_j^{\T}f}_\infty
 \leq
 C_{q,\alpha}r_j^{\alpha-q}\norm f_{C^\alpha(\T)}
 \qquad(f\in C^\alpha(\T),\ j\geq0).
\end{equation}
The low block satisfies
\[
 \norm{\Delta_{-1}^{\T}f}_{C^q}
 \leq
 C_q\norm f_\infty.
\]
\end{lemma}

\begin{proof}
Let $K=\mathcal F^{-1}\varphi$.  The periodic kernel of $\Delta_j^{\T}$ is the periodization of
\[
 K_j(x):=r_j^{-1}K(x/r_j).
\]
Since $\int_\R K=0$, its convolution with a periodic lift of $f$ can be written as
\[
 \Delta_j^{\T}f(x)
 =
 \int_\R K(z)\bigl(f(x-r_jz)-f(x)\bigr)\,\dd z,
\]
where $f$ is read periodically.  The geodesic H\"older bound and periodicity give
\[
 \abs{f(x-r_jz)-f(x)}
 \leq
 C\norm f_{C^\alpha(\T)}r_j^\alpha(1+\abs z)^\alpha.
\]
Therefore
\[
 \norm{\Delta_j^{\T}f}_\infty
 \leq
 Cr_j^\alpha\norm f_{C^\alpha(\T)}
 \int_\R\abs{K(z)}(1+\abs z)^\alpha\,\dd z.
\]
The Fourier support is contained in a fixed annulus of radius $r_j^{-1}$.  Bernstein's inequality on $\T$ gives
\[
 \norm{\partial_x^q\Delta_j^{\T}f}_\infty
 \leq
 C_qr_j^{-q}\norm{\Delta_j^{\T}f}_\infty,
\]
which proves \eqref{eq:periodic-dyadic-bounds}.  The low-frequency symbol has finite support on $\Z$, so every $C^q$ norm of $\Delta_{-1}^{\T}f$ is bounded by $C_q\norm f_\infty$.
\end{proof}

The periodic symbol product is the same as the real-line product in \cref{thm:band-inversion}.  The only additional analytic point is uniform boundedness, which follows from \cref{lem:periodization}.

\begin{proposition}[Periodic analysis--reconstruction]\label{prop:periodic-reconstruction}
For every $f\in C^\alpha(\T)$ and every $j\geq0$,
\begin{equation}\label{eq:periodic-band-identity}
 I_{j,0}^{\alpha,\T}d_j^{\alpha,\T}f
 =
 \Delta_j^{\T}f.
\end{equation}
Moreover,
\begin{equation}\label{eq:periodic-inverse-bound}
 \norm{I_{j,0}^{\alpha,\T}z}_\infty
 \leq
 C\lambda_j\norm z_\infty.
\end{equation}
The analysis map $\cD_{\alpha,0}^{\T}$ is injective.  Every $a=(u_{-1},z)\in\cA_{\alpha,0}^{\T}$ therefore has a unique representing function $f\in C^\alpha(\T)$, and that function is given by
\begin{equation}\label{eq:periodic-reconstruction-series}
 f
 =
 u_{-1}+\sum_{j=0}^\infty I_{j,0}^{\alpha,\T}z_j,
\end{equation}
where the series is the periodic Littlewood--Paley reconstruction of the representing function.
\end{proposition}

\begin{proof}
For every integer frequency $k$ in the support of $\varphi(r_j\cdot)$, the multiplier product is
\[
 \lambda_j\widetilde\varphi(r_jk)\frac1{ir_jk}
 \frac{ir_jk}{\lambda_j}
 =
 \widetilde\varphi(r_jk)
 =1.
\]
This proves \eqref{eq:periodic-band-identity}.  The symbol
\[
 \widetilde\varphi(\zeta)(i\zeta)^{-1}
\]
is smooth and compactly supported away from zero.  Its dilates have uniformly $L^1$-bounded real-line kernels, so \cref{lem:periodization} gives \eqref{eq:periodic-inverse-bound}.

Suppose
\[
 \cD_{\alpha,0}^{\T}f
 =
 \cD_{\alpha,0}^{\T}g.
\]
Then $\Delta_{-1}^{\T}(f-g)=0$ and $d_j^{\alpha,\T}(f-g)=0$ for every $j\geq0$.  Applying \eqref{eq:periodic-band-identity} gives
\[
 \Delta_j^{\T}(f-g)=0
 \qquad(j\geq0).
\]
The periodic Littlewood--Paley reconstruction then yields $f-g=0$, proving injectivity.

Finally, if $a=(u_{-1},z)=\cD_{\alpha,0}^{\T}f$, then \eqref{eq:periodic-band-identity} gives $I_{j,0}^{\alpha,\T}z_j=\Delta_j^{\T}f$.  Retaining the low block and summing the periodic Littlewood--Paley decomposition proves \eqref{eq:periodic-reconstruction-series}.  Injectivity makes the representing function unique.
\end{proof}

\begin{definition}[Periodic realized-range inverse and transported norm]\label{def:periodic-realized-inverse}
After \cref{prop:periodic-reconstruction}, define
\[
 \cI_{\alpha,0}^{\T}\bigl(\cD_{\alpha,0}^{\T}f\bigr):=f
 \qquad(f\in C^\alpha(\T)).
\]
This is well defined by injectivity and is represented by \eqref{eq:periodic-reconstruction-series}.  For $a\in\cA_{\alpha,0}^{\T}$ define
\[
 \norm a_{\alpha,0}^{\mathrm{tr}}
 :=
 \norm{\cI_{\alpha,0}^{\T}a}_{C^\alpha(\T)}.
\]
Thus $\cD_{\alpha,0}^{\T}$ and $\cI_{\alpha,0}^{\T}$ are inverse isometries between $C^\alpha(\T)$ and the realized range equipped with the transported norm.
\end{definition}

A rough lift needs a finite-dimensional path whose two channels have the same H\"older amplitude.  The exact coordinate supplies the derivative channel in normalized form; the factor $\lambda_j$ reintroduces the physical amplitude.

Let $E$ be a finite-dimensional real normed vector space.  Fix linear maps
\[
 A_j:\R^2\longrightarrow E,
 \qquad j\geq-1,
\]
satisfying
\begin{equation}\label{eq:representation-bound}
 M_A:=\sup_{j\geq-1}\norm{A_j}<\infty.
\end{equation}
The maps $A_j$ are the finite-dimensional representation data used to combine the value and quadrature channels.

\begin{definition}[Amplitude-reinsertion map]\label{def:amplitude-reinsertion}
For $a=(u_{-1},z)\in\cA_{\alpha,0}^{\T}$, define
\[
 Z_{-1}(a)
 :=
 \bigl(u_{-1},\partial_xu_{-1}\bigr),
\]
and, for $j\geq0$,
\begin{equation}\label{eq:reinserted-channel}
 Z_j(a)
 :=
 \bigl(I_{j,0}^{\alpha,\T}z_j,\lambda_jz_j\bigr).
\end{equation}
The \emph{amplitude-reinsertion map} associated with $A=(A_j)_{j\geq-1}$ is
\[
 \mathfrak R_A:\cA_{\alpha,0}^{\T}
 \longrightarrow
 \prod_{j\geq-1}C^\infty(\T;E),
 \qquad
 \mathfrak R_A(a)
 :=
 (A_jZ_j(a))_{j\geq-1}.
\]
The first component in \eqref{eq:reinserted-channel} is reconstructed from the exact derivative coordinate, while the second restores its H\"older amplitude.
\end{definition}

The target space makes the dyadic decay and derivative growth explicit.

\begin{definition}[Dyadic H\"older path space]\label{def:scale-path-space-interface}
For $0<\alpha<1$, let $\cS^\alpha(E)$ be the space of families
\[
 x=(x_{-1},x_0,x_1,\ldots),
 \qquad
 x_j\in C^\infty(\T;E),
\]
such that $x_{-1}$ has Fourier support in a fixed compact set, each $x_j$ for $j\geq0$ is supported in a fixed annulus of radius $2^j$, and
\begin{equation}\label{eq:scale-path-norm-interface}
 \norm x_{\cS^\alpha}
 :=
 \norm{x_{-1}}_{C^1}
 +
 \sup_{j\geq0}
 \left(
 2^{\alpha j}\norm{x_j}_\infty
 +
 2^{(\alpha-1)j}\norm{\partial_xx_j}_\infty
 \right)
 <\infty.
\end{equation}
This is the path space used in the rough-lift construction from \cref{sec:scale-paths} onward.
\end{definition}

The next theorem is the exact interface between the two halves of the paper.  It is not merely the observation that both use Littlewood--Paley blocks: it reconstructs the rough-path input from the information-preserving coordinate and proves the required norm bound.

\begin{theorem}[Amplitude-reinsertion interface]\label{thm:amplitude-reinsertion}
Let $0<\alpha<1$, let $A=(A_j)_{j\geq-1}$ satisfy \eqref{eq:representation-bound}, and let $a\in\cA_{\alpha,0}^{\T}$.  Then
\[
 \mathfrak R_A(a)\in\cS^\alpha(E)
\]
and
\begin{equation}\label{eq:reinsertion-bound}
 \norm{\mathfrak R_A(a)}_{\cS^\alpha}
 \leq
 C M_A\norm a_{\alpha,0}^{\mathrm{tr}}.
\end{equation}
If
\[
 a=\cD_{\alpha,0}^{\T}f
 \qquad(f\in C^\alpha(\T)),
\]
then the reinserted channels are exactly
\begin{equation}\label{eq:reinsertion-identity}
 Z_j(a)
 =
 \bigl(\Delta_j^{\T}f,r_j\partial_x\Delta_j^{\T}f\bigr)
 \qquad(j\geq0).
\end{equation}
Thus the finite-dimensional dyadic path used by Fourier normal ordering is a continuous realization of the exact scale coordinate.
\end{theorem}

\begin{proof}
Let
\[
 f:=\cI_{\alpha,0}^{\T}a.
\]
By definition of the transported norm,
\[
 \norm f_{C^\alpha(\T)}=\norm a_{\alpha,0}^{\mathrm{tr}}.
\]
The realized-range identity gives
\[
 z_j=d_j^{\alpha,\T}f
 =
 \frac{r_j}{\lambda_j}\partial_x\Delta_j^{\T}f.
\]
Using \eqref{eq:periodic-band-identity},
\[
 I_{j,0}^{\alpha,\T}z_j=\Delta_j^{\T}f,
\]
while multiplication by $\lambda_j$ gives
\[
 \lambda_jz_j=r_j\partial_x\Delta_j^{\T}f.
\]
These are the two components in \eqref{eq:reinsertion-identity}.

The estimate \eqref{eq:periodic-dyadic-bounds} with $q=0,1,2$ gives
\[
 \norm{\Delta_j^{\T}f}_\infty
 +
 r_j\norm{\partial_x\Delta_j^{\T}f}_\infty
 \leq
 C\lambda_j\norm f_{C^\alpha(\T)}.
\]
Differentiating both channels and using Bernstein gives
\begin{align*}
 \norm{\partial_x\Delta_j^{\T}f}_\infty
 &\leq
 Cr_j^{-1}\lambda_j\norm f_{C^\alpha(\T)},\\
 \norm{r_j\partial_x^2\Delta_j^{\T}f}_\infty
 &\leq
 Cr_j^{-1}\lambda_j\norm f_{C^\alpha(\T)}.
\end{align*}
Therefore
\[
 \norm{Z_j(a)}_\infty
 \leq
 C2^{-\alpha j}\norm f_{C^\alpha(\T)},
\]
and
\[
 \norm{\partial_xZ_j(a)}_\infty
 \leq
 C2^{(1-\alpha)j}\norm f_{C^\alpha(\T)}.
\]
Applying $A_j$ and using \eqref{eq:representation-bound},
\begin{align*}
 2^{\alpha j}\norm{A_jZ_j(a)}_\infty
 &\leq
 CM_A\norm f_{C^\alpha(\T)},\\
 2^{(\alpha-1)j}\norm{\partial_x(A_jZ_j(a))}_\infty
 &\leq
 CM_A\norm f_{C^\alpha(\T)}.
\end{align*}
The low block is smooth and obeys the same bound with its finite set of Fourier modes absorbed into $C$.  Taking the supremum in \eqref{eq:scale-path-norm-interface} proves \eqref{eq:reinsertion-bound}.
\end{proof}

For any two coordinate choices that synthesize the same first-level path $X$, the fixed path-level FNO regularization in \cref{sec:fno} produces the same lift.
 \section{Dyadic paths and rough-path metrics}\label{sec:scale-paths}

The interface theorem produces a family of smooth frequency blocks with uniform H\"older weights.  Fourier normal ordering acts on the synthesized path, so we now establish the Banach structure, cutoff convergence, and roughness depth of the abstract space $\cS^\alpha(E)$ introduced in \cref{def:scale-path-space-interface}.

For $x=(x_j)_{j\geq-1}\in\cS^\alpha(E)$, define the cutoff tower and cutoff path by
\[
 x^{\leq N}:=(x_{-1},x_0,\ldots,x_N,0,0,\ldots),
 \qquad
 X_N:=\sum_{-1\leq j\leq N}x_j.
\]
The limiting first-level path will be denoted by
\[
 X:=\sum_{j\geq-1}x_j.
\]

The norm in \eqref{eq:scale-path-norm-interface} controls both the amplitude and the spatial derivative of each block, and therefore controls increments by the minimum of these two estimates.

\begin{lemma}[Single-block increment estimate]\label{lem:single-block-increment}
For $x\in\cS^\alpha(E)$, $j\geq0$, and $h\in\R$,
\begin{equation}\label{eq:single-block-increment}
 \norm{x_j(\cdot+h)-x_j}_\infty
 \leq
 C\norm x_{\cS^\alpha}
 2^{-\alpha j}\min\{1,2^j\abs h\}.
\end{equation}
\end{lemma}

\begin{proof}
The uniform part of \eqref{eq:scale-path-norm-interface} gives
\[
 \norm{x_j(\cdot+h)-x_j}_\infty
 \leq
 2\norm{x_j}_\infty
 \leq
 2\norm x_{\cS^\alpha}2^{-\alpha j}.
\]
The mean-value theorem and the derivative part of the same norm give
\[
 \norm{x_j(\cdot+h)-x_j}_\infty
 \leq
 \abs h\norm{\partial_xx_j}_\infty
 \leq
 \norm x_{\cS^\alpha}2^{-\alpha j}2^j\abs h.
\]
Taking the smaller bound proves \eqref{eq:single-block-increment}.
\end{proof}

The dyadic sum of the minimum in \eqref{eq:single-block-increment} is the basic H\"older summation used throughout the rough-lift estimates.

\begin{lemma}[Dyadic summation]\label{lem:dyadic-summation}
For $0<\theta<1$ and $0<h\leq1$,
\begin{equation}\label{eq:dyadic-summation}
 \sum_{j\geq0}2^{-\theta j}\min\{1,2^jh\}
 \leq
 C_\theta h^\theta.
\end{equation}
At $\theta=1$, the left-hand side is bounded by
\[
 Ch\left(1+\log\frac1h\right).
\]
\end{lemma}

\begin{proof}
Choose $J\in\Nzero$ with $2^Jh\simeq1$.  For $j\leq J$, use $\min\{1,2^jh\}=2^jh$; for $j>J$, use the bound $1$.  Then
\[
 \sum_{j\geq0}2^{-\theta j}\min\{1,2^jh\}
 \leq
 h\sum_{j\leq J}2^{(1-\theta)j}
 +
 \sum_{j>J}2^{-\theta j}
 \leq
 C_\theta h^\theta.
\]
If $\theta=1$, the first sum has $J+1$ terms of size $h$.
\end{proof}

The next proposition verifies that the abstract input class is complete and that ultraviolet cutoffs are uniformly bounded operations.

\begin{proposition}[Completeness of the scale-path space]\label{prop:scale-path-complete}
For every finite-dimensional $E$ and $0<\alpha<1$, the normed space $\cS^\alpha(E)$ is Banach.  The cutoff maps
\[
 x\longmapsto x^{\leq N}
\]
have operator norm at most one.
\end{proposition}

\begin{proof}
Let $(x^{(n)})_n$ be Cauchy in $\cS^\alpha(E)$.  For each fixed $j$, both $x_j^{(n)}$ and $\partial_xx_j^{(n)}$ converge uniformly.  Hence $x_j^{(n)}$ converges in $C^1$ to a limit $x_j$.  Fourier coefficients outside the prescribed compact support vanish for every $n$ and therefore for $x_j$; on $\T$, each such support contains finitely many modes, so $x_j$ is smooth.  Passing to the limit in
\[
 2^{\alpha j}\norm{x_j^{(n)}}_\infty
 +
 2^{(\alpha-1)j}\norm{\partial_xx_j^{(n)}}_\infty
\]
shows $x\in\cS^\alpha(E)$ and $x^{(n)}\to x$ in norm.  Removing high-index components can only decrease the supremum, proving the cutoff bound.
\end{proof}

The first-level path converges at every lower H\"older exponent.  The same exponent loss controls the cutoff tail in the stronger scale-path norm.

\begin{proposition}[First-level and scale-tower cutoff tails]\label{prop:scale-tail}
Let $0<\beta<\alpha<1$ and $x\in\cS^\alpha(E)$.  Then
\[
 X_N\longrightarrow X:=\sum_{j\geq-1}x_j
\]
in $C^\beta(\T;E)$, and
\begin{equation}\label{eq:first-level-tail}
 \norm{X-X_N}_{C^\beta}
 \leq
 C2^{-(\alpha-\beta)N}\norm x_{\cS^\alpha}.
\end{equation}
Moreover,
\begin{equation}\label{eq:scale-tower-tail}
 \norm{x-x^{\leq N}}_{\cS^\beta}
 \leq
 C2^{-(\alpha-\beta)N}\norm x_{\cS^\alpha}.
\end{equation}
\end{proposition}

\begin{proof}
For $j>N$, \cref{lem:single-block-increment} and
\[
 \min\{1,2^j\abs h\}
 \leq
 2^{\beta j}\abs h^\beta
\]
give
\[
 \norm{x_j(\cdot+h)-x_j}_\infty
 \leq
 C\norm x_{\cS^\alpha}2^{-(\alpha-\beta)j}\abs h^\beta.
\]
Summing over $j>N$ yields the H\"older seminorm in \eqref{eq:first-level-tail}; the uniform tail is bounded by $\sum_{j>N}2^{-\alpha j}$.  For \eqref{eq:scale-tower-tail}, each $j>N$ satisfies
\begin{align*}
 2^{\beta j}\norm{x_j}_\infty
 +2^{(\beta-1)j}\norm{\partial_xx_j}_\infty
 &\leq
 2^{-(\alpha-\beta)j}\norm x_{\cS^\alpha}.
\end{align*}
Taking the supremum proves the claim.
\end{proof}

A $\beta$-H\"older rough path requires tensor levels through degree $m=\lfloor1/\beta\rfloor$.  We choose the scale exponent $\alpha$ so that one fixed integer $m$ governs the construction.

\begin{definition}[Roughness depth]\label{def:roughness-depth}
For $0<\alpha<1$, let $m=m(\alpha)$ be the unique integer satisfying
\begin{equation}\label{eq:roughness-depth}
 \frac1{m+1}<\alpha\leq\frac1m.
\end{equation}
Fix
\begin{equation}\label{eq:beta-range}
 \frac1{m+1}<\beta<\alpha.
\end{equation}
Then
\[
 r\beta<1
 \quad(1\leq r\leq m),
 \qquad
 (m+1)\beta>1.
\]
The first inequalities describe the critical tensor levels; the last is the extension threshold.
\end{definition}

We now define the algebraic target and the two metrics whose different root conventions produce different stability exponents.

\begin{definition}[Truncated tensor group]\label{def:tensor-group}
Let
\[
 T^{(m)}(E):=\bigoplus_{r=0}^mE^{\otimes r}
\]
be the truncated tensor algebra with concatenation product.  Let
\[
 L^{(m)}(E):=\bigoplus_{r=1}^mL_r(E)
\]
be the free step-$m$ Lie algebra generated by $E$, realized inside $T^{(m)}(E)$, and define
\[
 G^{(m)}(E):=\exp L^{(m)}(E).
\]
Every tensor level is equipped with one fixed admissible norm.  Since $E$ is finite-dimensional, all such choices are equivalent.
\end{definition}

\begin{definition}[Multiplicative functional and rough-path metrics]\label{def:rough-metrics}
Put
\[
 \Delta_2:=\{(s,t):0\leq s\leq t\leq2\pi\}.
\]
A map
\[
 \mathbf X:\Delta_2\longrightarrow G^{(m)}(E)
\]
is \emph{multiplicative} if
\[
 \mathbf X_{s,t}=\mathbf X_{s,u}\mathbf X_{u,t}
 \qquad(0\leq s\leq u\leq t\leq2\pi).
\]
For two multiplicative functionals $\mathbf X,\mathbf Y$, define
\begin{equation}\label{eq:layerwise-metric}
 d_{\beta,m}^{\mathrm{lay}}(\mathbf X,\mathbf Y)
 :=
 \max_{1\leq r\leq m}
 \sup_{s<t}
 \frac{\norm{\pi_r(\mathbf X_{s,t}-\mathbf Y_{s,t})}}{\abs{t-s}^{r\beta}},
\end{equation}
and
\begin{equation}\label{eq:homogeneous-metric}
 d_{\beta,m}^{\mathrm{hom}}(\mathbf X,\mathbf Y)
 :=
 \max_{1\leq r\leq m}
 \sup_{s<t}
 \frac{\norm{\pi_r(\mathbf X_{s,t}-\mathbf Y_{s,t})}^{1/r}}{\abs{t-s}^{\beta}}.
\end{equation}
On bounded subsets, these coordinate seminorms are equivalent to the standard inhomogeneous and homogeneous rough-path metrics.  A group-valued multiplicative functional satisfying the graded H\"older bounds is weakly geometric.  It is strong geometric at exponent $\beta$ if it is a limit, in the $\beta$-H\"older rough-path topology, of signatures of bounded-variation paths.
\end{definition}

The FNO theorem used in the next section is stated for compactly supported paths on $\R$.  A fixed based localization converts the periodic input into that form without changing increments on the cut interval.

\begin{definition}[Based compact localization]\label{def:based-localization}
Fix $T_0:=2\pi$ and choose $\chi_0\in C_c^\infty(\R)$ with
\[
 \chi_0=1
 \quad\text{on a neighborhood of }[0,T_0].
\]
For a periodic path $Y:\T\to E$, let $\widetilde Y$ be its periodic extension to $\R$ and define
\[
 LY:=\chi_0\bigl(\widetilde Y-Y(0)\bigr).
\]
The map $L$ is invariant under addition of constants, sends constant paths to zero, and satisfies
\[
 LY(t)=Y(t)-Y(0)
 \qquad(0\leq t\leq T_0).
\]
\end{definition}

\begin{lemma}[Localization bound]\label{lem:localization-bound}
For every $0<\eta<1$,
\begin{equation}\label{eq:localization-bound}
 \norm{LY}_{C^\eta(\R)}
 \leq
 C_{\chi_0,\eta}\norm Y_{C^\eta(\T)}.
\end{equation}
If $Y$ is smooth and periodic, then $LY$ is smooth and compactly supported.
\end{lemma}

\begin{proof}
Periodic extension is bounded from $C^\eta(\T)$ to the local $C^\eta$ space on $\R$.  Multiplication by the fixed function $\chi_0$ is bounded on $C^\eta(\R)$, and only finitely many periods meet $\supp\chi_0$.  These facts give \eqref{eq:localization-bound}; smoothness and compact support are immediate for smooth $Y$.
\end{proof}
 \section{Fourier-normal-ordered rough enhancement}\label{sec:fno}

The first-level path $X$ produced in \cref{sec:scale-paths} is H\"older but its canonical iterated integrals need not be obtained as limits of the Fourier cutoffs $X_N$.  We therefore fix one Fourier-normal-ordering scheme and separate the imported existence theorem from the quantitative estimates proved for the dyadic scale-path class.

For a compact interval $I\subset\R$, a finite-dimensional normed space $F$, and $\theta>0$, let $C_2^\theta(I;F)$ be the space of continuous maps
\[
 U:\{(s,t)\in I^2:s\leq t\}\longrightarrow F
\]
with $U_{t,t}=0$ and
\[
 \norm U_{C_2^\theta(I;F)}
 :=
 \sup_{s<t}\frac{\norm{U_{s,t}}}{\abs{t-s}^\theta}<\infty.
\]

The analytic existence input is Unterberger's Fourier-normal-ordering theorem.  We use its rough lift and its lower-exponent approximation conclusion without reproving the tree estimates that establish existence.

\begin{theorem}[Fourier normal ordering, imported]\label{thm:unterberger}
Let $0<\gamma<1$ satisfy $1/\gamma\notin\N$, and put
\[
 m:=\lfloor1/\gamma\rfloor.
\]
For every compactly supported $\gamma$-H\"older path $Y:\R\to E$, the fixed Fourier-normal-ordering construction of \cite[Main Theorem and Sections~2--4]{UnterbergerFNO2010} produces a step-$m$ multiplicative group-like $\gamma$-H\"older functional
\[
 \cR_\gamma(Y)
 =
 \bigl(1,\cR_\gamma^1(Y),\ldots,\cR_\gamma^m(Y)\bigr)
\]
whose first level is
\[
 \pi_1\cR_\gamma(Y)_{s,t}=Y(t)-Y(s).
\]
If
\[
 \frac1{m+1}<\beta<\gamma,
\]
the approximation result in the same source implies that this step-$m$ lift is strong geometric in the $\beta$-H\"older rough-path topology.
\end{theorem}

The mixed-input estimate in \cref{prop:mixed-fno} is obtained by assigning different inputs to the integration vertices in the vertexwise estimates of \cite{UnterbergerFNO2010}.  Fourier denominators are indexed by integration trees; the later Lie projection uses Hall trees and leaves those denominators unchanged.

Fix once and for all a real basis
\[
 (e_1,\ldots,e_{d_E}),
 \qquad d_E:=\dim E,
\]
of $E$.  It is used for all coordinate decorations in the FNO formulas.  Let
\[
 E_{\C}:=E\otimes_{\R}\C
\]
be the complexification, equipped with a fixed conjugation-invariant complex norm extending the norm of $E$.  We write
\[
 E_{\C}^{\otimes r}
 :=
 \underbrace{E_{\C}\otimes_{\C}\cdots\otimes_{\C}E_{\C}}_{r\text{ factors}}
 \cong
 (E^{\otimes_{\R}r})_{\C},
\]
and equip these tensor powers with fixed conjugation-invariant tensor norms.  The same real basis is used as a complex basis of $E_{\C}$.  The one-sided signed Fourier shells introduced below act in these complexified spaces.  The final mixed coefficients for real inputs are proved in \cref{prop:mixed-fno} to lie in the real subspaces $E^{\otimes r}$.

\begin{definition}[Integration and Hall trees]\label{def:tree-species}
A \emph{rooted integration tree} $T$ of degree $r$ has a vertex set $V(T)$ with $\abs{V(T)}=r$, one vertex for each input differential, and edges oriented toward the root.  Fix a bijection and a coordinate decoration
\[
 \nu_T:V(T)\longrightarrow\{1,\ldots,r\},
 \qquad
 \ell_T:V(T)\longrightarrow\{1,\ldots,d_E\}.
\]
For $v\in V(T)$, write $w\succeq_Tv$ when the oriented path from $w$ to the root passes through $v$, and define the partial output frequency
\begin{equation}\label{eq:partial-output}
 \Lambda_v^T(\xi)
 :=
 \sum_{w\succeq_Tv}\xi_{\nu_T(w)}.
\end{equation}
In particular,
\[
 \Lambda_{\operatorname{root}(T)}^T(\xi)
 =
 \xi_1+\cdots+\xi_r.
\]
The quantities in \eqref{eq:partial-output} index actual time primitives in the FNO skeleton formula.

A \emph{Hall bracket tree} is a rooted binary tree whose leaves label input slots and whose internal vertices label Lie brackets in a fixed Hall basis \cite{Reutenauer1993}.  Hall trees are introduced only after tensor coefficients have been projected to the free Lie algebra.  No Fourier denominator is attached merely because a Hall bracket vertex is present.
\end{definition}

For each homogeneous tensor level, the mixed coefficient must first be defined at the vertices of the regularized FNO formula.  The polarization identity is then a consequence of this multilinear construction, not its definition.

\begin{definition}[Vertex-assignment FNO coefficient]\label{def:polarized-fno}
For $1\leq r\leq m$, put
\[
 P_{\gamma,r}(Y):=\pi_r\cR_\gamma(Y).
\]
Fix one regularized degree-$r$ summand of the FNO formula, specified by a decorated integration tree, a normal-order region, a dyadic multi-index, a coordinate decoration, and, when present, an admissible cut.  Its diagonal Fourier integrand contains one input differential at each vertex $v\in V(T)$.  For compactly supported paths $Y_1,\ldots,Y_r$, replace the input differential at $v$ by the corresponding differential of $Y_{a(v)}$, where
\[
 a:V(T)\longrightarrow\{1,\ldots,r\}
\]
ranges over all bijections, and average over the $r!$ assignments.  Perform this replacement in every skeleton and cut-boundary summand of the fixed scheme, retaining the same joint regularization domain and dyadic multi-index.  The resulting formal series is denoted by
\[
 \cR_{\gamma,r}^{\mathrm{va}}(Y_1,\ldots,Y_r).
\]
Because the fixed shell decomposition is one-sided in frequency, its shellwise realization is a priori $E_{\C}^{\otimes r}$-valued, even for real inputs.  The superscript ``va'' records its vertex-assignment construction.  Convergence, real multilinearity, symmetry, diagonal recovery, and return to the real tensor space are proved in \cref{prop:mixed-fno}; after that proposition we write $\cR_{\gamma,r}:=\cR_{\gamma,r}^{\mathrm{va}}$.
\end{definition}

\begin{definition}[Signed FNO output shells]\label{def:fno-output-shells}
For $\sigma\in C_c^\infty(\R)$ and an $E_{\C}$-valued tempered distribution $f$, define the Fourier multiplier $\operatorname{Op}(\sigma)$ componentwise by
\[
 \widehat{\operatorname{Op}(\sigma)f}(\xi)
 :=
 \sigma(\xi)\widehat f(\xi).
\]
For bounded or H\"older functions this agrees with convolution by the inverse Fourier transform of $\sigma$.  The same notation is used componentwise for $E_{\C}^{\otimes r}$-valued distributions.

Fix nonnegative real-valued functions
\[
 \varphi_K,\psi_K\in C_c^\infty(\R),
 \qquad K\in\Z,
\]
forming the signed smooth dyadic output-shell data of the chosen FNO regularization, with
\begin{equation}\label{eq:fno-shell-square-partition}
 \varphi_K=\psi_K^2,
 \qquad
 \sum_{K\in\Z}\varphi_K(\xi)=1
 \quad(\xi\in\R).
\end{equation}
The families have uniformly finite overlap and are uniformly bounded in every order-zero symbol seminorm: for every $q\in\N_0$,
\begin{equation}\label{eq:fno-shell-symbol-bounds}
 \sup_{K\in\Z}\sup_{\xi\in\R}
 (1+\abs\xi)^q
 \left(
   \abs{\partial_\xi^q\varphi_K(\xi)}
   +
   \abs{\partial_\xi^q\psi_K(\xi)}
 \right)
 <\infty.
\end{equation}
The central functions $\varphi_0,\psi_0$ are supported in a fixed compact neighborhood of the origin and are even:
\[
 \varphi_0(-\xi)=\varphi_0(\xi),
 \qquad
 \psi_0(-\xi)=\psi_0(\xi).
\]
For $K\neq0$ one has
\[
 \supp\varphi_K\cup\supp\psi_K
 \subset
 \bigl\{\xi:\operatorname{sgn}(K)\xi>0,\ c_\varphi2^{\abs K}\leq\abs\xi\leq C_\varphi2^{\abs K}\bigr\}.
\]
The positive and negative shells are related by conjugate reflection: for every $K>0$,
\[
 \varphi_{-K}(\xi)=\varphi_K(-\xi),
 \qquad
 \psi_{-K}(\xi)=\psi_K(-\xi).
\]
These identities expose the real structure of the signed-shell scheme.  Individual one-sided projections remain complex, while the complete reconstruction returns to the real tensor space.
Such data are obtained from a fixed signed dyadic cover by normalizing a smooth square sum and reflecting the positive shells; in this paper they are part of the fixed FNO scheme.  For $K\neq0$ these one-sided symbols do not preserve real-valuedness separately.  Accordingly all individual signed-shell projections below are formed in $E_{\C}$ or $E_{\C}^{\otimes r}$; real-valuedness is recovered only after the complete shell reconstruction and real polarization.  We write
\[
 \varphi_K(D):=\operatorname{Op}(\varphi_K),
 \qquad
 \psi_K(D):=\operatorname{Op}(\psi_K).
\]
The notation $\sigma(D)$ is functional calculus for the already fixed differential operator $D=-i\partial_x$; no expression of the form $D(\sigma)$ is used.

For an admissible cut $c$, let $T_0(T,c)$ denote $T$ when $c=\varnothing$ and the root component of the multiple-cut decomposition when $c\neq\varnothing$.  Define the time-dependent output frequency by
\begin{equation}\label{eq:fno-output-frequency}
 \Lambda_{\mathrm{out}}^{T,c}(\xi)
 :=
 \sum_{v\in V(T_0(T,c))}\xi_{\nu_T(v)}.
\end{equation}
Thus the output shell of a cut term is the shell of the root-component output, not the full-tree total output including the basepoint pruned components.
\end{definition}

\begin{lemma}[Uniform input-shell multiplier bound]\label{lem:uniform-input-shell-multiplier}
For every $0<\gamma<1$ there is a constant $C_\gamma$, depending only on $\gamma$ and the fixed shell data, such that
\begin{equation}\label{eq:uniform-input-shell-multiplier}
 \norm{\psi_K(D)Y}_{L^\infty(\R;E_{\C})}
 \leq
 C_\gamma 2^{-\abs K\gamma}\norm Y_{C^\gamma(\R;E)}
 \qquad(K\in\Z)
\end{equation}
for every compactly supported $Y\in C^\gamma(\R;E)$, viewed through the canonical real inclusion $E\hookrightarrow E_{\C}$.
\end{lemma}

\begin{proof}
The central multiplier $\psi_0(D)$ has a fixed $L^1$ kernel, so the assertion for $K=0$ follows from the uniform norm of $Y$.  Let $K\neq0$ and put $R_K:=2^{\abs K}$.  After the signed rescaling $\xi=\operatorname{sgn}(K)R_K\zeta$, the symbols $\psi_K$ are supported in one fixed annulus and have uniformly bounded derivatives by \eqref{eq:fno-shell-symbol-bounds}.  Their inverse Fourier kernels $L_K$ therefore satisfy
\[
 \norm{L_K}_{L^1}\leq C,
 \qquad
 \int_\R\abs x^\gamma\abs{L_K(x)}\,dx
 \leq
 C_\gamma R_K^{-\gamma}.
\]
Since $\psi_K(0)=0$, one has $\int L_K=0$, and hence
\[
 \psi_K(D)Y(x)
 =
 \int_\R L_K(y)\bigl(Y(x-y)-Y(x)\bigr)\,dy.
\]
The $C^\gamma$ increment bound gives \eqref{eq:uniform-input-shell-multiplier}.
\end{proof}

\begin{lemma}[Parameter-uniform two-time Besov--H\"older reconstruction]\label{lem:two-time-besov-reconstruction}
Let $I\subset\R$ be compact, let $F$ be a finite-dimensional real or complex normed space, and let $0<\alpha<1$.  In the complex case all norm estimates are understood after restriction of scalars to $\R$.  For each $K\in\Z$ and $s\in I$, let
\[
 H_{K,s}:\R\longrightarrow F
\]
be bounded and continuous, jointly continuous in $(s,t)\in I\times\R$, and suppose that the Fourier support of $t\mapsto H_{K,s}(t)$ is contained in $\supp\varphi_K$.  Assume
\begin{equation}\label{eq:two-time-besov-Malpha}
 M_\alpha
 :=
 \sup_{K\in\Z}
 2^{\abs K\alpha}
 \sup_{s\in I}
 \norm{H_{K,s}}_{L^\infty(\R;F)}
 <\infty.
\end{equation}
Define
\[
 U_K(s,t):=H_{K,s}(t)-H_{K,s}(s),
 \qquad s\leq t,
\]
and
\[
 U:=\sum_{K\in\Z}U_K,
 \qquad
 U^{\leq N}:=\sum_{\abs K\leq N}U_K.
\]
Then the series defining $U$ converges uniformly on the time simplex, $U\in C_2^\alpha(I;F)$, and
\begin{equation}\label{eq:two-time-besov-critical-bound}
 \norm U_{C_2^\alpha(I;F)}
 \leq
 C_\alpha M_\alpha.
\end{equation}
For every $0<\theta<\alpha$ and $N\geq0$,
\begin{equation}\label{eq:two-time-besov-tail-bound}
 \norm{U-U^{\leq N}}_{C_2^\theta(I;F)}
 \leq
 C_{\alpha,\theta}
 2^{-N(\alpha-\theta)}M_\alpha.
\end{equation}
The natural shell truncations converge in $C_2^\theta$ for every $0<\theta<\alpha$; at the endpoint $C_2^\alpha$ we use the reconstructed bound.
\end{lemma}

\begin{proof}
The estimate \eqref{eq:two-time-besov-Malpha} and $\alpha>0$ give uniform convergence because
\[
 \sum_{K\in\Z}\sup_{s\in I}\norm{H_{K,s}}_\infty
 \leq
 M_\alpha\sum_{K\in\Z}2^{-\abs K\alpha}<\infty.
\]
In particular $U$ is continuous and vanishes on the diagonal.  The shell support and Bernstein's inequality give, uniformly in $s$,
\[
 \norm{\partial_tH_{K,s}}_\infty
 \leq
 C2^{\abs K}\norm{H_{K,s}}_\infty.
\]
Hence, with $h=\abs{t-s}$,
\begin{equation}\label{eq:two-time-shell-increment}
 \norm{U_K(s,t)}
 \leq
 C M_\alpha 2^{-\abs K\alpha}
 \min\{1,2^{\abs K}h\}.
\end{equation}
For $0<h\leq1$, choose $J\geq0$ with $2^{-J-1}<h\leq2^{-J}$.  Summing \eqref{eq:two-time-shell-increment} over $\abs K\leq J$ and $\abs K>J$ gives
\[
 \sum_{\abs K\leq J}\norm{U_K(s,t)}
 \leq
 C M_\alpha h\sum_{n=0}^J2^{n(1-\alpha)}
 \leq
 C M_\alpha h^\alpha,
\]
and
\[
 \sum_{\abs K>J}\norm{U_K(s,t)}
 \leq
 C M_\alpha\sum_{n>J}2^{-n\alpha}
 \leq
 C M_\alpha h^\alpha.
\]
For $h\geq1$, the same conclusion follows directly from uniform convergence after enlarging the constant.  This proves \eqref{eq:two-time-besov-critical-bound}.

If $0<\theta<\alpha$, then $\min\{1,x\}\leq x^\theta$ for $x\geq0$.  Therefore
\[
 \sum_{\abs K>N}\norm{U_K(s,t)}
 \leq
 C M_\alpha h^\theta
 \sum_{n>N}2^{-n(\alpha-\theta)}
 \leq
 C_{\alpha,\theta}
 2^{-N(\alpha-\theta)}M_\alpha h^\theta,
\]
which is \eqref{eq:two-time-besov-tail-bound}.
\end{proof}

The main internal estimate is multilinear rather than diagonal.  Its proof must retain the original joint dyadic summation for cut terms; independently summing the cut components would lose the constraints that make the FNO estimate convergent.

\begin{lemma}[Mixed regularized tree estimate]\label{lem:mixed-regularized-tree}
Fix a decorated integration tree $T$ of degree $r$, a normal-order region $\mathfrak n$, an admissible cut $c$ (with $c=\varnothing$ allowed), and a vertex-assignment bijection
\[
 a:V(T)\xrightarrow{\sim}\{1,\ldots,r\}.
\]
Refine $\mathfrak n$, if necessary, into finitely many maximum-pattern regions and keep the same notation for a refined region.  Then, for every $v\in V(T)$, there is a fixed descendant
\[
 w_{\max}^{\mathfrak n}(v)\succeq_Tv
\]
which carries maximal dyadic frequency in the subtree rooted at $v$ for every multi-index in that region.  This choice depends on $\mathfrak n$ but not on the individual multi-index $k$.

Let $Z_{T,\mathfrak n,c}^{\mathrm{reg}}$ be the common regularized dyadic domain of the fixed FNO tree or cut formula, and define
\begin{equation}\label{eq:fno-shell-domain}
 Z_{T,\mathfrak n,c}^{\mathrm{reg}}(K)
 :=
 \left\{
 k\in Z_{T,\mathfrak n,c}^{\mathrm{reg}}:
 \supp\varphi_K
 \cap
 \Lambda_{\mathrm{out}}^{T,c}
 \left(
   \prod_{v\in V(T)}\supp\varphi_{k_v}
 \right)
 \neq\varnothing
 \right\}.
\end{equation}
Denote the mixed dyadic summand obtained by placing $Y_{a(v)}$ at the vertex $v$ by
\[
 \cR_{T,\mathfrak n,c,k}^{a}(Y_1,\ldots,Y_r).
\]
All signed-shell summands, primitives, and increments in this lemma are taken in $E_{\C}^{\otimes r}$.  Individual one-sided shells are complex-valued, and real-valuedness is recovered after complete reconstruction and polarization.
For the uncut term write
\[
 \cR_{T,\mathfrak n,\varnothing,k}^{a}(Y_1,\ldots,Y_r)_{s,t}
 =G_{T,\mathfrak n,k}^{a}(t)-G_{T,\mathfrak n,k}^{a}(s),
\]
and for a cut term use the source boundary representation
\[
 \cR_{T,\mathfrak n,c,k}^{a}(Y_1,\ldots,Y_r)_{s,t}
 =A_{T,\mathfrak n,c,k,s}^{a}(t)-A_{T,\mathfrak n,c,k,s}^{a}(s).
\]
For compactly supported inputs, the regularized skeleton primitives are bounded on the full endpoint line; in cut terms only the basepoint is restricted to $I$.  Define the output-shell seminorm by
\[
 \norm{\cR_{T,\mathfrak n,c,k}^{a}}_{K,I}
 :=
 \begin{cases}
  \norm{\varphi_K(D)G_{T,\mathfrak n,k}^{a}}_{L^\infty(\R)},&c=\varnothing,\\[1mm]
  \displaystyle\sup_{s\in I}
  \norm{\varphi_K(D)A_{T,\mathfrak n,c,k,s}^{a}}_{L^\infty(\R)},&c\neq\varnothing.
 \end{cases}
\]
Then, for every compact interval $I\subset\R$,
\begin{equation}\label{eq:mixed-tree-shell-bound}
 \sup_{K\in\Z}
 2^{\abs K r\gamma}
 \sum_{k\in Z_{T,\mathfrak n,c}^{\mathrm{reg}}(K)}
 \norm{\cR_{T,\mathfrak n,c,k}^{a}(Y_1,\ldots,Y_r)}_{K,I}
 \leq
 C_{I,T,c,\gamma}
 \prod_{b=1}^r\norm{Y_b}_{C^\gamma(\R)}.
\end{equation}
Since $1\leq r\leq m=\lfloor1/\gamma\rfloor$ and $1/\gamma\notin\N$, one has $0<r\gamma<1$.  Estimate \eqref{eq:mixed-tree-shell-bound} implies that the following \emph{projected} shell primitives converge absolutely in their displayed uniform norms:
\begin{equation}\label{eq:projected-shell-primitive}
 H_{K,s}^{a}(t)
 :=
 \begin{cases}
 \displaystyle
 \sum_{k\in Z_{T,\mathfrak n,\varnothing}^{\mathrm{reg}}(K)}
 \varphi_K(D)G_{T,\mathfrak n,k}^{a}(t),&c=\varnothing,\\[4mm]
 \displaystyle
 \sum_{k\in Z_{T,\mathfrak n,c}^{\mathrm{reg}}(K)}
 \varphi_K(D)A_{T,\mathfrak n,c,k,s}^{a}(t),&c\neq\varnothing.
 \end{cases}
\end{equation}
In the uncut case the right side is independent of $s$.  Define the projected output-shell increment by
\begin{equation}\label{eq:projected-shell-increment}
 U_K^{a}(s,t)
 :=
 H_{K,s}^{a}(t)-H_{K,s}^{a}(s).
\end{equation}
Thus the primitive, rather than its based difference, has Fourier support in $\supp\varphi_K$.  The family satisfies the hypotheses of \cref{lem:two-time-besov-reconstruction} with $\alpha=r\gamma$.  Its reconstruction
\[
 U^{a}:=\sum_{K\in\Z}U_K^{a}
 \in C_2^{r\gamma}(I;E_{\C}^{\otimes r})
\]
satisfies
\begin{equation}\label{eq:mixed-tree-C2-bound}
 \norm {U^{a}}_{C_2^{r\gamma}(I)}
 \leq
 C_{I,T,c,\gamma}
 \prod_{b=1}^r\norm{Y_b}_{C^\gamma(\R)}.
\end{equation}
If $N\geq0$ and
\[
 U^{a,\leq N}:=\sum_{\abs K\leq N}U_K^{a},
\]
then, for every $0<\theta<r\gamma$,
\begin{equation}\label{eq:mixed-tree-lower-tail}
 \norm{U^{a}-U^{a,\leq N}}_{C_2^\theta(I)}
 \leq
 C_{I,T,c,\gamma,\theta}
 2^{-N(r\gamma-\theta)}
 \prod_{b=1}^r\norm{Y_b}_{C^\gamma(\R)}.
\end{equation}
The partition identity \eqref{eq:fno-shell-square-partition} shows, after the absolutely summable interchange justified in the proof, that the projected shell sum reconstructs every original FNO summand exactly through the partition of unity; a summand may meet finitely many neighboring shells, but their projected contributions add to that summand.  The restriction defining $Z_{T,\mathfrak n,c}^{\mathrm{reg}}(K)$ merely omits terms on which $\varphi_K(D)$ is zero.  The output-shell truncations converge in every $C_2^\theta$ with $0<\theta<r\gamma$; the endpoint $C_2^{r\gamma}$ is controlled through reconstruction.  All constants are independent of the assignment $a$.
\end{lemma}

\begin{proof}
We record the dyadic weights and constraints from the source proof, because only the input attached to a vertex is changed here.  On the refined normal-order region, the regularization domain gives the partial-output comparison of \cite[Lemma~4.4]{UnterbergerFNO2010} with the fixed maximum-pattern descendant $w_{\max}^{\mathfrak n}(v)$.  With the notation of \cite[equations~(4.13)--(4.19)]{UnterbergerFNO2010}, the multiplier weight is
\[
 \Theta_{1,T}^{\mathfrak n}(k)
 =
 \prod_{v\in V(T)}
 2^{\abs{k_v}-\abs{k_{w_{\max}^{\mathfrak n}(v)}}}.
\]
All endpoint dependence of the Fourier integrands occurs through unit-modulus exponential factors.  Consequently the ensuing absolute-integral estimates are uniform on the full free-endpoint line; in a cut term the displayed constant is also uniform for the basepoint $s\in I$.  The uniform shell multiplier estimate \eqref{eq:uniform-input-shell-multiplier}, applied at the vertex $v$, gives
\begin{equation}\label{eq:mixed-vertex-block}
 \norm{\psi_{k_v}(D)
        Y_{a(v)}^{\ell_T(v)}}_\infty
 \leq
 C_\gamma 2^{-\abs{k_v}\gamma}
 \norm{Y_{a(v)}}_{C^\gamma}.
\end{equation}
Thus the path-dependent part factors as
\begin{equation}\label{eq:mixed-explicit-weight}
 \Theta_{1,T}^{\mathfrak n}(k)
 \prod_{v\in V(T)}
 \norm{\psi_{k_v}(D)
        Y_{a(v)}^{\ell_T(v)}}_\infty
 \leq
 C_{T,\gamma}
 \Omega_{T,\gamma}^{\mathfrak n}(k)
 \prod_{b=1}^r\norm{Y_b}_{C^\gamma},
\end{equation}
where
\[
 \Omega_{T,\gamma}^{\mathfrak n}(k)
 :=
 \prod_{v\in V(T)}
 2^{\abs{k_v}(1-\gamma)-\abs{k_{w_{\max}^{\mathfrak n}(v)}}}.
\]
Neither $\Theta_{1,T}^{\mathfrak n}$ nor $\Omega_{T,\gamma}^{\mathfrak n}$ depends on the identity of the path assigned to a vertex.

For $c=\varnothing$, the normal-order and output-shell constraints are exactly:
\begin{enumerate}[label=(U\arabic*),leftmargin=2.8em]
\item $k\in Z_{T,\mathfrak n}^{\mathrm{reg}}$, so each primitive output is comparable from below with the fixed maximal descendant frequency in its subtree;
\item the Minkowski sum of the vertex shells meets the shell of $\Lambda_{\mathrm{out}}^{T,\varnothing}=\Lambda_{\operatorname{root}(T)}^T$;
\item if $v_{\max}^{\mathfrak n}$ is the fixed maximal vertex of the region, then $\left|\abs{k_{v_{\max}^{\mathfrak n}}}-\abs K\right|\leq C_T$.  For noncentral shells this follows from $2^{\abs K}\asymp_{\varphi}\abs{\Lambda_{\mathrm{out}}^{T,\varnothing}(\xi)}\asymp_T\abs{\xi_{\nu_T(v_{\max}^{\mathfrak n})}}\asymp_{\varphi}2^{\abs{k_{v_{\max}^{\mathfrak n}}}}$; if a central shell occurs, compact support and the same partial-output comparison leave only finitely many indices.  No signed-index comparison is used; all dyadic weights below depend only on absolute indices;
\item all remaining indices satisfy the tree-order inequalities used in the reduction of \cite[equations~(4.19)--(4.21)]{UnterbergerFNO2010}.
\end{enumerate}
For each branch ending immediately below an uppermost node $n(w)$, use the convention of \cite[equations~(4.20)--(4.21)]{UnterbergerFNO2010} and write
\[
 \operatorname{Br}(w\twoheadrightarrow n(w))
 =
 (v_0=w,v_1,\ldots,v_{p_w-1}),
\]
where $v_{p_w-1}$ is joined directly to $n(w)$,
\[
 n(w)\notin\operatorname{Br}(w\twoheadrightarrow n(w)),
 \qquad
 p_w=\abs{\operatorname{Br}(w\twoheadrightarrow n(w))}.
\]
Thus the shared node index $k_{n(w)}$ is not summed inside any branch product; it is retained for the subsequent reduced-node summation and is handled exactly once.  The branch summation in the source proof gives
\[
 2^{-\abs{k_w}\gamma}
 \prod_{v\in\operatorname{Br}(w\twoheadrightarrow n(w))\setminus\{w\}}
 \sum_{\abs{k_v}\leq\abs{k_w}}
 2^{\abs{k_v}(1-\gamma)-\abs{k_w}}
 \leq
 C_{T,\gamma}2^{-\abs{k_w}\gamma p_w}.
\]
Summing the branches at an uppermost node and replacing them by one weighted vertex gives
\[
 \sum_{k\in Z_{T,\mathfrak n}^{\mathrm{reg}}(K)}
 \Omega_{T,\gamma}^{\mathfrak n}(k)
 \leq
 C_{T,\gamma}2^{-\abs K r\gamma};
\]
this is the reduced-tree induction in \cite[equations~(4.20)--(4.21)]{UnterbergerFNO2010}.  Combining it with \eqref{eq:mixed-explicit-weight} proves \eqref{eq:mixed-tree-shell-bound} in the uncut case.

Now let $c$ be nontrivial and let $T_0=T_0(T,c)$ be its root component.  The multiple-cut formula \cite[equation~(4.22)]{UnterbergerFNO2010} writes the boundary summand as one root-component skeleton increment times basepoint skeleton values of the pruned components.  All factors retain one common multi-index $k$ and one joint domain $Z_{T,\mathfrak n,c}^{\mathrm{reg}}(K)$.  That domain imposes:
\begin{enumerate}[label=(C\arabic*),leftmargin=2.8em]
\item the partial-output lower bound on every root-component primitive;
\item intersection of the root-component output $\Lambda_{\mathrm{out}}^{T,c}$ with the output shell $K$ as in \eqref{eq:fno-shell-domain};
\item $\abs{k_w}\leq\abs{k_{w_{\max}^{\mathfrak n,c}}}$ for every root-component vertex $w$, where $w_{\max}^{\mathfrak n,c}$ is the fixed maximal root-component vertex of the refined region;
\item $\abs{k_w}\leq\abs{k_v}$ whenever the root of a pruned component is attached above $w$.
\end{enumerate}
For each pruned component $L$, summing its internal indices with its root index fixed yields
\[
 C_{L,\gamma}2^{-\abs{k_{\rho_L}}\gamma\abs{V(L)}},
\]
by \cite[equations~(4.23)--(4.24)]{UnterbergerFNO2010}.  The remaining root-component and attachment indices are still summed jointly.  Use the source notation
\[
 R(w):=\{v:\ v\text{ is the root of a pruned component attached directly above }w\}.
\]
Split
\[
 R(w)=R(w)_{>}\sqcup R(w)_{<}
\]
according to the fixed normal order relative to $w_{\max}^{\mathfrak n,c}$, put
\[
 R_T^{>}:=\bigcup_wR(w)_{>},
 \qquad
 R_T^{<}:=\bigcup_wR(w)_{<},
\]
and let $L_v$ be the pruned component rooted at $v$.  Conditions (C1)--(C2), the root-component partial-output comparison, and shell localization give $\left|\abs{k_{w_{\max}^{\mathfrak n,c}}}-\abs K\right|\leq C_T$: for noncentral shells, $2^{\abs K}\asymp_{\varphi}\abs{\Lambda_{\mathrm{out}}^{T,c}(\xi)}\asymp_T\abs{\xi_{\nu_T(w_{\max}^{\mathfrak n,c})}}\asymp_{\varphi}2^{\abs{k_{w_{\max}^{\mathfrak n,c}}}}$, while central cases involve only a fixed finite set of indices.  Thus only magnitude scales are compared; no sign alignment is required, and the finite sign choices are absorbed into the constant.  After the preceding internal sums, the reduced dyadic weight is
\begin{equation}\label{eq:mixed-cut-reduced-weight}
 \Omega_{T,c,\gamma}^{\mathrm{red}}(K,k_{\mathrm{root}})
 :=
 \left(
  \prod_{w\in V(T_0)}
  2^{-\abs{k_w}\gamma
     \left(1+\sum_{v\in R(w)_{<}}\abs{V(L_v)}\right)}
 \right)
 2^{-\abs K\gamma\sum_{v\in R_T^{>}}\abs{V(L_v)}}.
\end{equation}
The final factor is the contribution of \cite[equation~(4.27)]{UnterbergerFNO2010}; the extra weights in the product over $w$ are those produced by \cite[equation~(4.28)]{UnterbergerFNO2010}.  The same reduced-tree summation as in the uncut case gives
\[
 \sum_{k_{\mathrm{root}}\ \mathrm{subject\ to}\ (\mathrm{C1})--(\mathrm{C4})}
 \Omega_{T,c,\gamma}^{\mathrm{red}}(K,k_{\mathrm{root}})
 \leq
 C_{T,c,\gamma}2^{-\abs K\gamma\abs{V(T)}},
\]
which is \cite[equation~(4.29)]{UnterbergerFNO2010}.  Replacing one common path by the assigned paths changes only the vertex estimate \eqref{eq:mixed-vertex-block}; it changes none of the dyadic weights or constraints.  Restoring the already summed pruned-component indices proves \eqref{eq:mixed-tree-shell-bound} for cut terms.

Estimate \eqref{eq:mixed-tree-shell-bound} proves absolute convergence of the projected sums in \eqref{eq:projected-shell-primitive}, uniformly in the displayed endpoint and basepoint variables.  Since every summand is smooth, the resulting family $H_{K,s}^{a}(t)$ is jointly continuous.  The same estimate gives
\[
 \sup_{K\in\Z}
 2^{\abs K r\gamma}
 \sup_{s\in I}
 \norm{H_{K,s}^{a}}_{L^\infty(\R;E_{\C}^{\otimes r})}
 \leq
 C
 \prod_{b=1}^r\norm{Y_b}_{C^\gamma}.
\]
Apply \cref{lem:two-time-besov-reconstruction} with $F=E_{\C}^{\otimes r}$ and $\alpha=r\gamma$ to obtain \eqref{eq:mixed-tree-C2-bound} and \eqref{eq:mixed-tree-lower-tail}.

The interchange of the shell sum and the internal dyadic sum is justified by the explicit absolutely summable estimate
\begin{equation}\label{eq:mixed-double-sum}
\begin{aligned}
&\sum_{K\in\Z}
 \sum_{k\in Z_{T,\mathfrak n,c}^{\mathrm{reg}}(K)}
 \begin{cases}
  \norm{\varphi_K(D)G_{T,\mathfrak n,k}^{a}}_{L^\infty(\R;E_{\C}^{\otimes r})},&c=\varnothing,\\[1mm]
  \displaystyle\sup_{s\in I}
  \norm{\varphi_K(D)A_{T,\mathfrak n,c,k,s}^{a}}_{L^\infty(\R;E_{\C}^{\otimes r})},&c\neq\varnothing
 \end{cases}\\
&\hspace{20mm}\leq
 C\sum_{K\in\Z}2^{-\abs K r\gamma}
 \prod_{b=1}^r\norm{Y_b}_{C^\gamma}
 <\infty.
\end{aligned}
\end{equation}
For a fixed internal multi-index $k$, the Fourier support of its root-component output is contained in the image of the product of the vertex supports under $\Lambda_{\mathrm{out}}^{T,c}$.  Hence the terms with $k\notin Z_{T,\mathfrak n,c}^{\mathrm{reg}}(K)$ have zero $K$-th projection by \eqref{eq:fno-shell-domain}.  Hence Tonelli's theorem for the absolutely summable projected family and \eqref{eq:fno-shell-square-partition} give, in distributions,
\begin{align*}
 \sum_{K\in\Z}H_{K,s}^{a}
 &=
 \sum_{K\in\Z}
 \sum_{k\in Z_{T,\mathfrak n,c}^{\mathrm{reg}}(K)}
 \varphi_K(D)
 \begin{cases}
  G_{T,\mathfrak n,k}^{a},&c=\varnothing,\\
  A_{T,\mathfrak n,c,k,s}^{a},&c\neq\varnothing
 \end{cases}\\
 &=
 \sum_{k\in Z_{T,\mathfrak n,c}^{\mathrm{reg}}}
 \left(\sum_{K\in\Z}\varphi_K(D)\right)
 \begin{cases}
  G_{T,\mathfrak n,k}^{a},&c=\varnothing,\\
  A_{T,\mathfrak n,c,k,s}^{a},&c\neq\varnothing
 \end{cases}\\
 &=
 \sum_{k\in Z_{T,\mathfrak n,c}^{\mathrm{reg}}}
 \begin{cases}
  G_{T,\mathfrak n,k}^{a},&c=\varnothing,\\
  A_{T,\mathfrak n,c,k,s}^{a},&c\neq\varnothing.
 \end{cases}
\end{align*}
Thus the projected shells reconstruct every original FNO summand exactly through the partition of unity, even though a summand may contribute to finitely many neighboring shells.  The reconstructed increment is the original joint FNO tree or cut series.  At $\theta=r\gamma$, the critical shell supremum is merely bounded and the truncation tail need not tend to zero.  The constants are assignment-independent because $a$ is bijective and the product of the vertex norms is always $\prod_{b=1}^r\norm{Y_b}_{C^\gamma}$.
\end{proof}

\begin{proposition}[Mixed-input continuity of the fixed FNO scheme]\label{prop:mixed-fno}
Fix the regularization domains used in \cref{thm:unterberger}.  Let $I\subset\R$ be compact and let $1\leq r\leq m$.  The shellwise vertex-assignment construction in \cref{def:polarized-fno} defines a continuous symmetric $r$-linear map
\[
 \cR_{\gamma,r}:(C_c^\gamma(\R;E))^r
 \longrightarrow C_2^{r\gamma}(I;E^{\otimes r})
\]
satisfying
\begin{equation}\label{eq:mixed-fno-estimate}
 \norm{\cR_{\gamma,r}(Y_1,\ldots,Y_r)}_{C_2^{r\gamma}(I;E^{\otimes r})}
 \leq
 C_{I,r,\gamma}
 \prod_{a=1}^r\norm{Y_a}_{C^\gamma(\R)}.
\end{equation}
Here $\cR_{\gamma,r}^{\leq N}$ denotes the finite sum, over all degree-$r$ FNO combinatorial data and vertex assignments, of the projected increments \eqref{eq:projected-shell-increment} with $\abs K\leq N$.  For every $0<\theta<r\gamma$, individual one-sided shell projections make this truncation naturally an element of
\[
 C_2^\theta(I;E_{\C}^{\otimes r}).
\]
The complete coefficient $\cR_{\gamma,r}(Y_1,\ldots,Y_r)$ is embedded from the real subspace $E^{\otimes r}$ into $E_{\C}^{\otimes r}$ when the difference below is formed.  In this complexified target, the output-shell truncations satisfy
\begin{equation}\label{eq:mixed-fno-lower-truncation}
 \norm{\cR_{\gamma,r}(Y_1,\ldots,Y_r)
       -\cR_{\gamma,r}^{\leq N}(Y_1,\ldots,Y_r)}_{C_2^\theta(I;E_{\C}^{\otimes r})}
 \leq
 C_{I,r,\gamma,\theta}
 2^{-N(r\gamma-\theta)}
 \prod_{a=1}^r\norm{Y_a}_{C^\gamma(\R)}.
\end{equation}
For general $C^\gamma$ inputs these truncations converge in each strict lower topology $C_2^\theta$, $0<\theta<r\gamma$.

Its diagonal is the original FNO coefficient:
\begin{equation}\label{eq:fno-diagonal-recovery}
 \cR_{\gamma,r}(Y,\ldots,Y)=P_{\gamma,r}(Y).
\end{equation}
Consequently $P_{\gamma,r}$ is a continuous homogeneous polynomial of degree $r$, and the signed polarization identity
\begin{equation}\label{eq:fno-polarization}
 \cR_{\gamma,r}(Y_1,\ldots,Y_r)
 =
 \frac1{2^rr!}
 \sum_{\epsilon\in\{-1,1\}^r}
 \left(\prod_{a=1}^r\epsilon_a\right)
 P_{\gamma,r}\!\left(\sum_{a=1}^r\epsilon_aY_a\right)
\end{equation}
holds.
\end{proposition}

\begin{proof}
Apply \cref{lem:mixed-regularized-tree} to every decorated integration tree, refined normal-order region, coordinate decoration, admissible cut, and vertex assignment occurring at degree $r$.  There are finitely many such combinatorial data at fixed degree.  The critical shell supremum bounds sum to \eqref{eq:mixed-fno-estimate}, and the lower-exponent tail bounds sum to \eqref{eq:mixed-fno-lower-truncation}.

For each fixed shell $K$, the projected primitive in \eqref{eq:projected-shell-primitive}, and hence the increment in \eqref{eq:projected-shell-increment}, is real $r$-linear in the assigned real inputs, with values a priori in $E_{\C}^{\otimes r}$.  Moreover,
\[
 \sum_{K\in\Z}\norm{U_K^{a}}_{L^\infty(\{(s,t)\in I^2:s\leq t\};E_{\C}^{\otimes r})}
 \leq
 C\sum_{K\in\Z}2^{-\abs K r\gamma}
 \prod_{a=1}^r\norm{Y_a}_{C^\gamma}<\infty.
\]
Hence the projected shell reconstruction converges in the uniform topology, which passes real multilinearity to the limit.  Averaging over the vertex bijections makes the resulting map symmetric.  Its continuity into the critical complexified space $C_2^{r\gamma}(I;E_{\C}^{\otimes r})$ follows directly from the shell supremum estimate, not from convergence of finite shell sums in that norm.

If $Y_1=\cdots=Y_r=Y$ is real, every output shell agrees with the corresponding output shell of the diagonal FNO coefficient.  Thus diagonal recovery holds shellwise, hence in distributions; both sides have continuous $C_2^{r\gamma}$ reconstructions, so uniqueness of the continuous representative yields \eqref{eq:fno-diagonal-recovery}.  The imported FNO lift is real for real $Y$, and therefore
\[
 P_{\gamma,r}(Y)\in C_2^{r\gamma}(I;E^{\otimes r}).
\]
The real polarization identity for the symmetric real $r$-linear map just constructed gives
\[
 \cR_{\gamma,r}(Y_1,\ldots,Y_r)
 =
 \frac1{2^rr!}
 \sum_{\epsilon\in\{-1,1\}^r}
 \left(\prod_{a=1}^r\epsilon_a\right)
 P_{\gamma,r}\!\left(\sum_{a=1}^r\epsilon_aY_a\right).
\]
Every term on the right belongs to the real subspace $E^{\otimes r}$.  Hence the complexified shell reconstruction takes values in $C_2^{r\gamma}(I;E^{\otimes r})$ for all real inputs, and the displayed identity is exactly \eqref{eq:fno-polarization}.  This proves both the stated real target and the polarization formula.
\end{proof}

The multilinear estimate yields local Lipschitz continuity of every homogeneous coefficient by a telescoping identity.

\begin{corollary}[Local Lipschitz continuity]\label{cor:fno-local-lipschitz}
Let $Y,Z\in C_c^\gamma(\R;E)$ satisfy
\[
 \norm Y_{C^\gamma}+\norm Z_{C^\gamma}\leq R.
\]
Then, for $1\leq r\leq m$ and compact $I\subset\R$,
\begin{equation}\label{eq:fno-local-lipschitz}
 \norm{\pi_r(\cR_\gamma(Y)-\cR_\gamma(Z))}_{C_2^{r\gamma}(I)}
 \leq
 C_{I,r,\gamma,R}\norm{Y-Z}_{C^\gamma}.
\end{equation}
\end{corollary}

\begin{proof}
By diagonalization,
\[
 P_{\gamma,r}(Y)=\cR_{\gamma,r}(Y,\ldots,Y).
\]
The difference of two diagonal values is
\begin{align*}
 P_{\gamma,r}(Y)-P_{\gamma,r}(Z)
 =
 \sum_{q=1}^r
 \cR_{\gamma,r}
 (Z,\ldots,Z,\underset{q\text{th slot}}{Y-Z},Y,\ldots,Y).
\end{align*}
Applying \eqref{eq:mixed-fno-estimate} to every term gives \eqref{eq:fno-local-lipschitz}.
\end{proof}

We now apply the fixed FNO scheme to the localized dyadic cutoffs.  The exponent $\gamma$ is chosen below $\alpha$ so that the scale tail converges in the norm used by the FNO theorem, and above $\beta$ so that the resulting lift has the required roughness depth.

\begin{definition}[FNO lift of a dyadic scale path]\label{def:fno-scale-lift}
Assume
\[
 \frac1{m+1}<\beta<\alpha\leq\frac1m,
\]
and choose
\begin{equation}\label{eq:gamma-choice}
 \beta<\gamma<\alpha,
 \qquad
 1/\gamma\notin\N.
\end{equation}
For $x\in\cS^\alpha(E)$, let $X_N$ and $X$ be the paths in \cref{prop:scale-tail}.  Define
\begin{equation}\label{eq:fno-cutoff-lifts}
 \mathbf X^N
 :=
 \Res_{[0,2\pi]}\cR_\gamma(LX_N),
 \qquad
 \mathbf X
 :=
 \Res_{[0,2\pi]}\cR_\gamma(LX).
\end{equation}
The first levels are $X_N(t)-X_N(s)$ and $X(t)-X(s)$ on the cut interval.
\end{definition}

The next theorem states both the ultraviolet rate and the distinction between the layerwise and homogeneous metrics.  The latter loss is not analytic: it is produced by the $r$th root in \eqref{eq:homogeneous-metric}.

\begin{theorem}[Quantitative rough enhancement]\label{thm:rough-enhancement}
Let $x\in\cS^\alpha(E)$ and assume \eqref{eq:gamma-choice}.  The functionals in \eqref{eq:fno-cutoff-lifts} are step-$m$ multiplicative group-like lifts.  On every ball
\[
 \norm x_{\cS^\alpha}\leq R,
\]
they satisfy
\begin{equation}\label{eq:layerwise-cutoff-rate}
 d_{\beta,m}^{\mathrm{lay}}(\mathbf X^N,\mathbf X)
 \leq
 C_R2^{-(\alpha-\gamma)N},
\end{equation}
\begin{equation}\label{eq:homogeneous-cutoff-rate}
 d_{\beta,m}^{\mathrm{hom}}(\mathbf X^N,\mathbf X)
 \leq
 C_R2^{-(\alpha-\gamma)N/m}.
\end{equation}
The limiting lift $\mathbf X$ is strong geometric at exponent $\beta$.

If $x,y$ belong to the same radius-$R$ ball and
\[
 \delta:=\norm{x-y}_{\cS^\alpha}\leq1,
\]
then
\begin{equation}\label{eq:layerwise-stability}
 d_{\beta,m}^{\mathrm{lay}}(\mathbf X[x],\mathbf X[y])
 \leq
 C_R\delta,
\end{equation}
and
\begin{equation}\label{eq:homogeneous-stability}
 d_{\beta,m}^{\mathrm{hom}}(\mathbf X[x],\mathbf X[y])
 \leq
 C_R\delta^{1/m}.
\end{equation}
Thus the lift map is locally Lipschitz for the layerwise metric and locally $1/m$-H\"older for the homogeneous metric.
\end{theorem}

\begin{proof}
By \eqref{eq:scale-tower-tail} with exponent $\gamma$ and by the first-level synthesis estimate,
\[
 \norm{X-X_N}_{C^\gamma(\T)}
 \leq
 C2^{-(\alpha-\gamma)N}\norm x_{\cS^\alpha}.
\]
The localization bound \eqref{eq:localization-bound} gives
\begin{equation}\label{eq:localized-cutoff-tail}
 \norm{LX-LX_N}_{C^\gamma(\R)}
 \leq
 C2^{-(\alpha-\gamma)N}\norm x_{\cS^\alpha}.
\end{equation}
The localized paths remain in one bounded subset of $C_c^\gamma(\R;E)$.  Applying \eqref{eq:fno-local-lipschitz} to \eqref{eq:localized-cutoff-tail} gives, for $1\leq r\leq m$,
\[
 \norm{\pi_r(\mathbf X^N_{s,t}-\mathbf X_{s,t})}
 \leq
 C_R2^{-(\alpha-\gamma)N}\abs{t-s}^{r\gamma}.
\]
Since $0\leq t-s\leq2\pi$ and $\gamma>\beta$,
\[
 \abs{t-s}^{r\gamma}
 \leq
 (2\pi)^{r(\gamma-\beta)}\abs{t-s}^{r\beta}.
\]
Taking the maximum in \eqref{eq:layerwise-metric} proves \eqref{eq:layerwise-cutoff-rate}.

At level $r$, taking the $r$th root gives
\[
 \frac{\norm{\pi_r(\mathbf X^N_{s,t}-\mathbf X_{s,t})}^{1/r}}{\abs{t-s}^{\beta}}
 \leq
 C_R2^{-(\alpha-\gamma)N/r}.
\]
For sufficiently large $N$, the slowest exponent occurs at $r=m$; enlarging the constant for finitely many smaller cutoffs proves \eqref{eq:homogeneous-cutoff-rate}.

For two inputs $x,y$, the synthesis and localization maps satisfy
\[
 \norm{LX[x]-LX[y]}_{C^\gamma(\R)}
 \leq
 C\norm{x-y}_{\cS^\alpha}=C\delta.
\]
The same application of \cref{cor:fno-local-lipschitz} gives \eqref{eq:layerwise-stability}.  At level $r$, the homogeneous bound is $C_R\delta^{1/r}$; because $0\leq\delta\leq1$,
\[
 \max_{1\leq r\leq m}\delta^{1/r}=\delta^{1/m},
\]
which proves \eqref{eq:homogeneous-stability}.  Strong geometricity at the lower exponent $\beta$ is the approximation conclusion in \cref{thm:unterberger}.
\end{proof}

\begin{corollary}[Dependence only on the synthesized path]\label{cor:fno-coordinate-independence}
Suppose two dyadic coordinate systems synthesize the same first-level path $X$ and the same based localization $L$ and regularization $\cR_\gamma$ are used.  Then their limiting lifts coincide.
\end{corollary}

\begin{proof}
Both constructions equal the path-level object
\[
 \Res_{[0,2\pi]}\cR_\gamma(LX).
\]
\end{proof}

The cutoff $\mathbf X^N$ is the selected FNO lift of $X_N$.  Strong geometricity follows from the bounded-variation approximation in \cref{thm:unterberger}.
 \section{Smooth cutoffs and typed Magnus--BCH reconstruction}\label{sec:typed-reconstruction}

The limiting object $\mathbf X$ is a rough path, so no classical derivative $G^{-1}\dot G$ is assigned to it.  The local-to-interval structure is instead derived at each smooth ultraviolet cutoff $N$.  Before differentiating the FNO series, we prove normal convergence after arbitrary endpoint differentiation.  This is the analytic step that legitimizes every one-time density used later in the section.

\subsection{Two Hopf structures and one group path}

FNO uses rooted integration trees to construct regularized primitives, but Chen multiplicativity is expressed through words.  These are different Hopf-algebra structures and their convolutions must not be identified.

\begin{definition}[Rooted-tree and shuffle convolutions]\label{def:two-convolutions}
Let $H_{\mathrm{rt}}(E)$ be the commutative Hopf algebra generated by decorated rooted integration trees \cite{ConnesKreimer1998,Foissy2002}.  Its product is disjoint union and its coproduct is
\[
 \Delta_{\mathrm{rt}}T
 =
 T\otimes1+1\otimes T
 +
 \sum_{c\in\operatorname{Adm}'(T)}P_c(T)\otimes R_c(T),
\]
where $c$ ranges over nontrivial admissible cuts.  For linear maps $f,g:H_{\mathrm{rt}}(E)\to A$ into a commutative algebra $A$, define
\[
 f\star_{\mathrm{rt}}g
 :=
 m_A(f\otimes g)\Delta_{\mathrm{rt}}.
\]
This convolution organizes rooted-tree cuts.

Let $H_{\mathrm{Sh}}(E)$ be the shuffle Hopf algebra on words in a fixed basis $(e_i)$ of $E$.  Its product is the shuffle product and its coproduct is deconcatenation:
\[
 \Delta_{\mathrm{Sh}}(e_{i_1}\cdots e_{i_r})
 =
 \sum_{q=0}^r
 (e_{i_1}\cdots e_{i_q})
 \otimes
 (e_{i_{q+1}}\cdots e_{i_r}).
\]
With antipode $S_{\mathrm{Sh}}$, define
\[
 f\star_{\mathrm{Sh}}g
 :=
 m_A(f\otimes g)\Delta_{\mathrm{Sh}}.
\]
This convolution composes endpoint word characters.  The operations $\star_{\mathrm{rt}}$ and $\star_{\mathrm{Sh}}$ have different domains and different coproducts.
\end{definition}

Unterberger writes endpoint words in the outer-to-inner integration order, whereas we use chronological tensor words.  Define the convention-change involution
\[
 \operatorname{rev}(e_{i_1}\cdots e_{i_r})
 :=e_{i_r}\cdots e_{i_1}.
\]
Word reversal is a coalgebra anti-automorphism for deconcatenation:
\[
 \Delta_{\mathrm{Sh}}\operatorname{rev}
 =
 (\operatorname{rev}\otimes\operatorname{rev})
 \tau\Delta_{\mathrm{Sh}}.
\]
Let $\bar S_{\mathrm{Sh}}$ denote the antipode acting on words written in the outer-to-inner convention.  Transporting the antipode through the convention change gives the independent definition
\begin{equation}\label{eq:barred-shuffle-antipode}
 \bar S_{\mathrm{Sh}}
 =
 \operatorname{rev}^{-1}S_{\mathrm{Sh}}\operatorname{rev}.
\end{equation}
In the shuffle Hopf algebra the antipode is signed word reversal, so \eqref{eq:barred-shuffle-antipode} is compatible with the standard explicit antipode formula.  Unterberger's endpoint formula is
\[
 \bar\chi_t^Y\star_{\mathrm{Sh}}
 (\bar\chi_s^Y\circ \bar S_{\mathrm{Sh}})
\]
in the outer-to-inner convention; see \cite[Definition~3.7, equations~(3.34)--(3.39)]{UnterbergerFNO2010}.  Define the chronological endpoint character by
\[
 \chi_t^Y:=\bar\chi_t^Y\circ\operatorname{rev}.
\]
For maps $f,g$ on the shuffle Hopf algebra, the anti-coalgebra identity implies
\[
 (f\star_{\mathrm{Sh}}g)\circ\operatorname{rev}
 =
 (g\circ\operatorname{rev})\star_{\mathrm{Sh}}(f\circ\operatorname{rev}).
\]
Using \eqref{eq:barred-shuffle-antipode}, the outer-to-inner endpoint formula therefore becomes
\begin{equation}\label{eq:shuffle-increment}
 [\cR_\gamma(Y)_{s,t}]_{i_1\cdots i_r}
 =
 \Bigl(
 (\chi_s^Y\circ S_{\mathrm{Sh}})\star_{\mathrm{Sh}}\chi_t^Y
 \Bigr)(e_{i_1}\cdots e_{i_r}).
\end{equation}
Under the chronological character--tensor correspondence, convolution represents concatenation in the displayed order.  Hence \eqref{eq:shuffle-increment} is the character of $G(s)^{-1}G(t)$, consistently with the interval convention used throughout the paper.  The tree and cut summands are assembled before this single multiplicative tensor path is formed; compare Chen's composition identity \cite{Chen1957}.  No second interval subtraction is introduced below.

\subsection{Termwise differentiation}

At a smooth cutoff, each input Fourier transform is Schwartz.  Primitive denominators still require a uniform infrared comparison with descendant frequencies; this comparison is supplied by the fixed FNO regularization domain.

\begin{lemma}[Normal convergence after endpoint differentiation]\label{lem:termwise-differentiation}
Let $Y\in C_c^\infty(\R;E)$.  Fix:
\[
 r\geq1,
 \qquad
 \text{a decorated integration tree }T\text{ of degree }r,
 \qquad
 \text{an admissible cut }c,
\]
\[
 a,b\in\Nzero,
 \qquad
 K\subset\R^2\text{ compact}.
\]
Let $Z_T^{\mathrm{reg}}$ be the dyadic multi-indices retained by the fixed FNO regularization domain, and let
\[
 \cR_{T,k}^{c}(Y)_{s,t}
\]
denote the corresponding skeleton or cut-boundary summand.  There exists
\[
 w_{T,c,a,b,Y,K}\in\ell^1(Z_T^{\mathrm{reg}})
\]
such that
\begin{equation}\label{eq:termwise-majorant}
 \sup_{(s,t)\in K}
 \abs{\partial_s^a\partial_t^b\cR_{T,k}^{c}(Y)_{s,t}}
 \leq
 w_{T,c,a,b,Y,K}(k)
 \qquad(k\in Z_T^{\mathrm{reg}}).
\end{equation}
Consequently,
\begin{equation}\label{eq:termwise-normal-convergence}
 \sum_{k\in Z_T^{\mathrm{reg}}}
 \sup_{(s,t)\in K}
 \abs{\partial_s^a\partial_t^b\cR_{T,k}^{c}(Y)_{s,t}}
 <\infty.
\end{equation}
Every fixed-degree coefficient of $\cR_\gamma(Y)_{s,t}$ is therefore $C^\infty$ in $(s,t)$, and every endpoint derivative is obtained by differentiating the regularized FNO series term by term.
\end{lemma}

\begin{proof}
Write $\widehat Y^\ell$ for a coordinate Fourier transform.  Every input differential contributes a factor
\[
 i\xi_{\nu_T(v)}\widehat Y^{\ell_T(v)}(\xi_{\nu_T(v)}).
\]
Partition the fixed regularization domain into its finitely many normal-order maximum-pattern regions.  On one such region $\mathfrak n$, choose for every integration vertex $v$ a fixed descendant
\[
 w_{\max}^{\mathfrak n}(v)\succeq_Tv
\]
which carries maximal dyadic frequency in the subtree rooted at $v$ for every multi-index in that region.  On this region, the partial-output comparison in \cite[Lemma~4.4, equation (4.1)]{UnterbergerFNO2010} gives
\begin{equation}\label{eq:partial-output-comparison}
 \frac12
 \abs{\xi_{\nu_T(w_{\max}^{\mathfrak n}(v))}}
 <
 \abs{\Lambda_v^T(\xi)}
 \leq
 \abs{V(T)}
 \abs{\xi_{\nu_T(w_{\max}^{\mathfrak n}(v))}}.
\end{equation}
The normal-order relation also gives
\begin{equation}\label{eq:vertex-max-comparison}
 \abs{\xi_{\nu_T(v)}}
 \leq
 C_T\abs{\xi_{\nu_T(w_{\max}^{\mathfrak n}(v))}}.
\end{equation}
Combining \eqref{eq:partial-output-comparison} and \eqref{eq:vertex-max-comparison}, every primitive factor satisfies the infrared estimate
\begin{equation}\label{eq:infrared-cancellation}
 \left|
 \frac{i\xi_{\nu_T(v)}
       \widehat Y^{\ell_T(v)}(\xi_{\nu_T(v)})}
      {\Lambda_v^T(\xi)}
 \right|
 \leq
 C_T
 \abs{\widehat Y^{\ell_T(v)}(\xi_{\nu_T(v)})}.
\end{equation}
Thus no partial-output denominator produces an uncontrolled singularity at the origin.

Applying $\partial_s^a\partial_t^b$ inserts a polynomial
\[
 P_{a,b}\bigl((\xi_{\nu_T(v)})_v,(\Lambda_v^T(\xi))_v\bigr)
\]
of degree at most $a+b$ in input, total-output, and partial-output frequencies.  Near the origin, \eqref{eq:infrared-cancellation} controls every denominator.  At infinity, every $\widehat Y^\ell$ is Schwartz.  Choose an integer $M$ larger than
\[
 a+b+r+\abs{V(T)}.
\]
Then
\[
 \abs{\widehat Y^\ell(\xi)}
 \leq
 C_{Y,M}(1+\abs\xi)^{-M},
\]
and the endpoint-differentiated Fourier integrand has an $L^1$ majorant uniform on $K$.

Now restrict the integral to the dyadic box indexed by $k\in Z_T^{\mathrm{reg}}$.  Because the Littlewood--Paley partition has finite overlap, the integrals of the preceding $L^1$ majorant over these boxes form a summable sequence.  Define that sequence to be
\[
 w_{T,c,a,b,Y,K}(k).
\]
This proves \eqref{eq:termwise-majorant} and \eqref{eq:termwise-normal-convergence} for an uncut skeleton on one maximum-pattern region.

For a cut term, use the multiple-cut boundary expansion of \cite[equation~(4.22)]{UnterbergerFNO2010}.  Each product has the form
\begin{equation}\label{eq:cut-boundary-product}
 [\delta\cR\mathrm{SkI}_{T_0,k|_{V_0}}]_{s,t}
 \prod_{\ell=1}^q
 [\cR\mathrm{SkI}_{T_\ell,k|_{V_\ell}}]_s,
 \qquad
 V(T)=V_0\sqcup\cdots\sqcup V_q,
\end{equation}
where $T_0$ is the root component and $T_1,\ldots,T_q$ are the successive pruned components.  The restrictions $k|_{V_\ell}$ are not summed independently: they remain coupled by the single domain $Z_T^{\mathrm{reg}}$.  All $t$-derivatives act on the increment factor in \eqref{eq:cut-boundary-product}, whereas $s$-derivatives are distributed among that factor and the basepoint skeleton factors.  The Leibniz formula is therefore a finite sum of products
\[
 \partial_s^{a_0}\partial_t^b
 [\delta\cR\mathrm{SkI}_{T_0,k|_{V_0}}]_{s,t}
 \prod_{\ell=1}^q
 \partial_s^{a_\ell}
 [\cR\mathrm{SkI}_{T_\ell,k|_{V_\ell}}]_s,
 \qquad
 a_0+\cdots+a_q=a.
\]
For each derivative allocation, the partial-output comparison and the cancellation \eqref{eq:infrared-cancellation} bound the full product by one full-vertex Fourier majorant
\begin{equation}\label{eq:cut-joint-majorant}
 W_{\mathbf a,\mathbf b}(k)
 :=
 C\int_{\prod_{v\in V(T)}\supp\varphi_{k_v}}
 (1+\abs\xi)^L
 \prod_{v\in V(T)}
 \abs{\widehat Y^{\ell_T(v)}(\xi_{\nu_T(v)})}
 \,\dd\xi,
\end{equation}
where $L$ depends only on $a,b,T,c$ and $\abs\xi:=\sum_v\abs{\xi_{\nu_T(v)}}$.  Choose the Schwartz order of every coordinate of $Y$ larger than $L+r+1$.  Finite overlap of the dyadic boxes then gives the joint estimate
\begin{align}\label{eq:cut-joint-l1}
 \sum_{k\in Z_T^{\mathrm{reg}}}W_{\mathbf a,\mathbf b}(k)
 &\leq
 C\int_{\R^r}
 (1+\abs\xi)^L
 \prod_{v\in V(T)}
 \abs{\widehat Y^{\ell_T(v)}(\xi_{\nu_T(v)})}
 \,\dd\xi
 <\infty.
\end{align}
There are only finitely many Leibniz multi-indices and finitely many products in a fixed cut formula.  Summing \eqref{eq:cut-joint-l1} over them defines an $\ell^1$ majorant on one maximum-pattern region and proves \eqref{eq:termwise-majorant}--\eqref{eq:termwise-normal-convergence} there.  Each original normal-order domain was refined into only finitely many maximum-pattern regions, so the regional majorants may be added to obtain $w_{T,c,a,b,Y,K}\in\ell^1(Z_T^{\mathrm{reg}})$.  No product of independently summed $\ell^1$ sequences is used.

At fixed degree there are finitely many trees, cuts, normal-order regions, coordinate decorations, and word labels.  The Weierstrass test therefore permits arbitrary endpoint differentiation of their full sum.
\end{proof}

The preceding lemma gives one smooth group path at each cutoff.  The logarithm and Maurer--Cartan derivative are formed only after all tree and cut summations have been completed.

\begin{proposition}[One smooth cutoff group path]\label{prop:one-group-path}
Let $x\in\cS^\alpha(E)$ and let $\mathbf X^N$ be the cutoff lift in \eqref{eq:fno-cutoff-lifts}.  Define
\begin{equation}\label{eq:cutoff-group-path}
 G_N(t):=\mathbf X^N_{0,t}
 \qquad(0\leq t\leq2\pi).
\end{equation}
Then $G_N:[0,2\pi]\to G^{(m)}(E)$ is smooth and
\[
 \mathbf X^N_{s,t}=G_N(s)^{-1}G_N(t).
\]
The quantities
\[
 A_N:=G_N^{-1}\dot G_N,
 \qquad
 \Omega_N:=\log G_N
\]
are smooth Lie-valued paths.  Their relation is the ordinary finite-step Magnus identity.  No second forest-valued interval coefficient is compared with them.
\end{proposition}

\begin{proof}
Multiplicativity and group-likeness follow from \cref{thm:unterberger}.  By \cref{lem:termwise-differentiation}, every tensor coefficient of $G_N(t)=\mathbf X^N_{0,t}$ is smooth.  Since $G^{(m)}(E)$ is finite-step nilpotent, inversion, logarithm, multiplication, and differentiation are polynomial operations in the homogeneous coordinates.  Therefore $A_N$ and $\Omega_N$ are smooth and Lie-valued.

Differentiating
\[
 G_N=e^{\Omega_N}
\]
in the finite-dimensional nilpotent tensor group gives the Magnus relation.  Formula \eqref{eq:shuffle-increment} has already assembled the rooted-tree terms into the single path $G_N$, so no additional interval construction remains to be identified.
\end{proof}

\subsection{One-time densities and two-time coefficients}

Write the homogeneous decompositions
\[
 A_N=\sum_{r=1}^mA_{r,N},
 \qquad
 \Omega_N=\sum_{r=1}^m\Omega_{r,N}.
\]
The first degree is fixed by the first-level path and is not part of the counterterm construction:
\begin{equation}\label{eq:typed-degree-one}
 A_{1,N}=\dot X_N,
 \qquad
 \Omega_{1,N}=X_N-X_N(0),
 \qquad
 I_{1,N}(s,t)=X_N(t)-X_N(s).
\end{equation}
No symbols $\Gamma_{1,N}$ or $D_{1,N}^{\mathrm{prep}}$ are introduced.

The higher-degree objects have four distinct types.  The first two are local in one time variable, the third is a one-time primitive, and the fourth is a based two-time interval coefficient.

\begin{definition}[Typed smooth-cutoff objects]\label{def:typed-objects}
Let $2\leq r\leq m$.  Define the one-time counterterm density
\begin{equation}\label{eq:counterterm-density}
 \Gamma_{r,N}(t):=-A_{r,N}(t)\in L_r(E).
\end{equation}
Let $(c_q)_{q\geq0}$ be determined by
\begin{equation}\label{eq:bernoulli-coefficients}
 \frac{z}{1-e^{-z}}
 =
 \sum_{q\geq0}c_qz^q
 =
 1+\frac12z+\frac1{12}z^2-\frac1{720}z^4+\cdots.
\end{equation}
Define the prepared one-time density
\begin{equation}\label{eq:prepared-density}
 D_{r,N}^{\mathrm{prep}}
 :=
 \sum_{q=1}^{r-1}c_q
 \sum_{\substack{i_1+\cdots+i_q+p=r\\i_\nu\geq1,\ p\geq1}}
 \ad_{\Omega_{i_1,N}}\cdots
 \ad_{\Omega_{i_q,N}}A_{p,N}.
\end{equation}
Define the one-time primitive
\begin{equation}\label{eq:typed-primitive}
 P_{r,N}(t)
 :=
 \Omega_{r,N}(t)-\Omega_{r,N}(0).
\end{equation}
Finally, define the based two-time logarithmic coefficient
\begin{equation}\label{eq:renormalized-interval}
 I_{r,N}^{\mathrm{ren}}(s,t)
 :=
 \pi_r\log\bigl(G_N(s)^{-1}G_N(t)\bigr)
 =
 \pi_r\BCH(-\Omega_N(s),\Omega_N(t)).
\end{equation}
All four objects are constructed from the same smooth group path $G_N$.  None is assigned to the limiting rough path $\mathbf X$.
\end{definition}

The homogeneous Magnus formula identifies the difference between the prepared density and the counterterm as the time derivative of the logarithmic coordinate.

\begin{proposition}[Homogeneous Magnus recursion]\label{prop:homogeneous-magnus}
For $1\leq r\leq m$,
\begin{equation}\label{eq:homogeneous-magnus}
 \dot\Omega_{r,N}
 =
 \sum_{q=0}^{r-1}c_q
 \sum_{\substack{i_1+\cdots+i_q+p=r\\i_\nu\geq1,\ p\geq1}}
 \ad_{\Omega_{i_1,N}}\cdots
 \ad_{\Omega_{i_q,N}}A_{p,N}.
\end{equation}
\end{proposition}

\begin{proof}
For the left Maurer--Cartan derivative $A_N=G_N^{-1}\dot G_N$, differentiation of $G_N=e^{\Omega_N}$ gives
\begin{equation}\label{eq:full-magnus}
 \dot\Omega_N
 =
 \frac{\ad_{\Omega_N}}{1-e^{-\ad_{\Omega_N}}}A_N
 =
 \sum_{q\geq0}c_q\ad_{\Omega_N}^qA_N.
\end{equation}
Taking the homogeneous degree-$r$ component of \eqref{eq:full-magnus} yields \eqref{eq:homogeneous-magnus}.
\end{proof}

\begin{theorem}[Typed local evolution]\label{thm:typed-local-evolution}
For every smooth cutoff $N$ and every $2\leq r\leq m$,
\begin{equation}\label{eq:typed-local-evolution}
 \dot P_{r,N}
 =
 D_{r,N}^{\mathrm{prep}}-\Gamma_{r,N}.
\end{equation}
Equivalently,
\begin{equation}\label{eq:typed-primitive-increment}
 P_{r,N}(t)-P_{r,N}(s)
 =
 \int_s^t
 \bigl(D_{r,N}^{\mathrm{prep}}(u)-\Gamma_{r,N}(u)\bigr)
 \,\dd u.
\end{equation}
\end{theorem}

\begin{proof}
In \eqref{eq:homogeneous-magnus}, the term with $q=0$ is
\[
 c_0A_{r,N}=A_{r,N}=-\Gamma_{r,N}.
\]
The sum of all terms with $q\geq1$ is exactly the definition \eqref{eq:prepared-density}.  Therefore
\[
 \dot\Omega_{r,N}
 =
 D_{r,N}^{\mathrm{prep}}-\Gamma_{r,N}.
\]
Since $P_{r,N}=\Omega_{r,N}-\Omega_{r,N}(0)$, one has $\dot P_{r,N}=\dot\Omega_{r,N}$, proving \eqref{eq:typed-local-evolution}.  Integration from $s$ to $t$ gives \eqref{eq:typed-primitive-increment}.
\end{proof}

A primitive increment is not yet the logarithm of the based group increment.  Noncommutativity forces a lower-degree BCH endpoint polynomial.

\begin{definition}[BCH boundary package]\label{def:bch-package}
For $2\leq r\leq m$, define
\begin{equation}\label{eq:bch-package}
 B_{r,N}(s,t)
 :=
 \pi_r\BCH(-\Omega_N(s),\Omega_N(t))
 -
 \bigl(\Omega_{r,N}(t)-\Omega_{r,N}(s)\bigr).
\end{equation}
By homogeneity of the BCH series, $B_{r,N}$ is a universal Lie polynomial in
\[
 \Omega_{1,N}(s),\ldots,\Omega_{r-1,N}(s),
 \qquad
 \Omega_{1,N}(t),\ldots,\Omega_{r-1,N}(t).
\]
It contains no degree-$r$ unknown.
\end{definition}

\begin{theorem}[Exact local-to-interval reconstruction]\label{thm:bch-reconstruction}
For every smooth cutoff $N$ and every $2\leq r\leq m$,
\begin{equation}\label{eq:exact-bch-reconstruction}
 I_{r,N}^{\mathrm{ren}}(s,t)
 =
 P_{r,N}(t)-P_{r,N}(s)+B_{r,N}(s,t).
\end{equation}
Consequently,
\begin{equation}\label{eq:integrated-bch-reconstruction}
 I_{r,N}^{\mathrm{ren}}(s,t)
 =
 \int_s^t
 \bigl(D_{r,N}^{\mathrm{prep}}(u)-\Gamma_{r,N}(u)\bigr)
 \,\dd u
 +B_{r,N}(s,t).
\end{equation}
More generally, the same formulas hold in every degree $2\leq r\leq M$ for an arbitrary smooth based path in a step-$M$ nilpotent group.
\end{theorem}

\begin{proof}
By \eqref{eq:renormalized-interval} and \eqref{eq:bch-package},
\[
 I_{r,N}^{\mathrm{ren}}(s,t)
 =
 \Omega_{r,N}(t)-\Omega_{r,N}(s)+B_{r,N}(s,t).
\]
Using \eqref{eq:typed-primitive},
\[
 \Omega_{r,N}(t)-\Omega_{r,N}(s)
 =
 P_{r,N}(t)-P_{r,N}(s),
\]
which proves \eqref{eq:exact-bch-reconstruction}.  Substitute \eqref{eq:typed-primitive-increment} to obtain \eqref{eq:integrated-bch-reconstruction}.

For a smooth based path $H:[0,2\pi]\to G^{(M)}(E)$, define
\[
 A_H:=H^{-1}\dot H,
 \qquad
 \Omega_H:=\log H,
\]
and repeat \eqref{eq:counterterm-density}--\eqref{eq:bch-package} through degree $M$.  The proof uses only the homogeneous Magnus identity and the BCH identity in the finite-step group, so the same formulas follow in every degree $r\leq M$.
\end{proof}

The complete based density absorbs the derivative of the endpoint package.  This is the precise object whose integral equals the interval logarithm.

\begin{corollary}[Based prepared density]\label{cor:based-prepared-density}
For fixed $s$ and $2\leq r\leq m$, define
\begin{equation}\label{eq:based-prepared-density}
 D_{r,N}^{\mathrm{based,prep}}(s,u)
 :=
 D_{r,N}^{\mathrm{prep}}(u)+\partial_uB_{r,N}(s,u).
\end{equation}
Then
\begin{equation}\label{eq:based-density-integral}
 I_{r,N}^{\mathrm{ren}}(s,t)
 =
 \int_s^t
 \bigl(D_{r,N}^{\mathrm{based,prep}}(s,u)-\Gamma_{r,N}(u)\bigr)
 \,\dd u,
\end{equation}
and
\begin{equation}\label{eq:based-density-derivative}
 \partial_tI_{r,N}^{\mathrm{ren}}(s,t)
 =
 D_{r,N}^{\mathrm{based,prep}}(s,t)-\Gamma_{r,N}(t).
\end{equation}
\end{corollary}

\begin{proof}
Differentiate \eqref{eq:integrated-bch-reconstruction} in the upper endpoint:
\[
 \partial_tI_{r,N}^{\mathrm{ren}}(s,t)
 =
 D_{r,N}^{\mathrm{prep}}(t)-\Gamma_{r,N}(t)
 +\partial_tB_{r,N}(s,t).
\]
Using \eqref{eq:based-prepared-density} gives \eqref{eq:based-density-derivative}.  Since $I_{r,N}^{\mathrm{ren}}(s,s)=0$, integration from $s$ to $t$ gives \eqref{eq:based-density-integral}.
\end{proof}
 \section{Root resonance and endpoint gauges}\label{sec:resonance-gauge}

The local-to-interval identity in \cref{thm:bch-reconstruction} is independent of the explicit Fourier chart once the smooth group path is constructed.  The Fourier chart has one decisive feature: differentiation in the outer endpoint removes the primitive denominator associated with the total output.  We first prove this statement and then state its degree-two resonance test.  The final subsection describes how different lifts with the same first level are related.

\subsection{Absence of the total-root denominator}

Let $T$ be a decorated rooted integration tree and recall the partial outputs $\Lambda_v^T$ from \eqref{eq:partial-output}.  A skeleton chart is first written as an iterated primitive and is then differentiated in its outer endpoint to obtain a one-time density.  The root primitive is therefore treated differently from every proper subtree primitive.

\begin{definition}[Chart grammar of the fixed FNO scheme]\label{def:fno-chart-grammar}
Fix a decorated rooted integration tree $T$.  The class $\mathfrak C_{\mathrm{FNO}}(T)$ is the smallest class of labelled chart expressions generated by the following rules.
\begin{enumerate}[label=(G\arabic*),leftmargin=2.8em]
\item The uncut regularized skeleton of $T$ belongs to $\mathfrak C_{\mathrm{FNO}}(T)$; every vertex initially carries its singleton label in $V(T)$.
\item If an admissible cut occurs in the multiple-cut boundary formula \eqref{eq:cut-boundary-product}, replace the current component by its root component and its pruned components, and restrict the existing labels to those components.
\item If a connected regularized skeleton component appearing in that recursive formula has already been evaluated and is used as one effective input of a larger component, it may be contracted to one effective vertex.  The new vertex carries the union of the labels of the contracted component.  The contraction is only a bookkeeping representation of a product factor already present in the fixed formula: it introduces neither a new summand nor a new time primitive.  Every reciprocal primitive factor already present in the contracted component is retained in an inherited denominator list, with its original component-vertex record and represented support.  Only an explicit cancellation under rule~(G4) may remove such a factor.
\item Common algebraic factors may be cancelled, after which every remaining reciprocal primitive factor is retained explicitly in the denominator list.
\item The final density chart is obtained by differentiating the unique outer endpoint-dependent root component in its outer endpoint.
\end{enumerate}
The grammar contains only the operations present in the fixed skeleton and multiple-cut formulas.  It is not an independent enlargement of the FNO regularization scheme.
\end{definition}

\begin{lemma}[Coverage of the fixed FNO charts]\label{lem:fno-chart-coverage}
Every uncut skeleton-density summand and every cut-boundary density summand of the fixed FNO construction used in \cref{lem:termwise-differentiation} is represented by an element of $\mathfrak C_{\mathrm{FNO}}(T)$ for its original integration tree $T$.
\end{lemma}

\begin{proof}
The uncut summand is rule~(G1).  The boundary expansion \eqref{eq:cut-boundary-product}, which is the multiple-cut formula of \cite[equation~(4.22)]{UnterbergerFNO2010}, replaces an integration tree by one root component and finitely many pruned components; this is rule~(G2).  Iterating the same formula on a component gives all successive cut terms, so induction on the number of cut edges covers every boundary product used by the fixed construction.

Whenever a connected sub-skeleton has already been integrated and enters a remaining component as one coefficient, its diagrammatic replacement by an effective input is exactly rule~(G3); it is a relabelling of the same product term, records the represented original vertices, and does not add an integration variable.  Simplification of the resulting Fourier multiplier uses only cancellation of common algebraic factors, rule~(G4).  Finally, the one-time density is obtained by differentiating the outer increment factor, rule~(G5).  These are precisely the operations used in the summands denoted by $\cR_{T,k}^c$ in \cref{lem:termwise-differentiation}; no other chart-producing operation occurs there.
\end{proof}

\begin{lemma}[Primitive-denominator inheritance under cuts and contractions]\label{lem:primitive-denominator-inheritance}
Fix an original rooted integration tree $T$.  Every component integration tree $C$ occurring after cuts and contractions is equipped with a label map
\[
 \operatorname{lab}_C:V(C)\longrightarrow 2^{V(T)}\setminus\{\varnothing\}.
\]
The labels are pairwise disjoint.  An uncontracted original vertex $w$ has label $\operatorname{lab}_C(w)=\{w\}$, while an effective vertex created by contracting a connected subtree has as its label the union of the labels of all vertices in that subtree.  For $v\in V(C)$ define its represented original support by
\[
 \operatorname{supp}_T(C,v)
 :=
 \bigcup_{w\succeq_C v}\operatorname{lab}_C(w).
\]
Let $\mathfrak c\in\mathfrak C_{\mathrm{FNO}}(T)$.  By \cref{lem:fno-chart-coverage}, this includes every differentiated skeleton and cut-boundary chart used by the fixed FNO scheme.  Then:
\begin{enumerate}[label=(\roman*),leftmargin=2.4em]
\item Every denominator that survives is attached to an actual time primitive rooted at a vertex $v$ of an actual component integration tree $C$, and its frequency is the sum of the original input frequencies indexed by $\operatorname{supp}_T(C,v)$.
\item An admissible cut replaces a component by a root component and pruned components whose represented original vertex sets are pairwise disjoint subsets of the represented set of the original component.
\item Contracting an already formed connected component introduces no new primitive denominator.  Every reciprocal primitive factor already present in that component remains in an inherited denominator list, with the represented support determined by the label map.  Only an explicit algebraic cancellation may remove an inherited factor.  The effective input carries the union label and introduces no new time integration or primitive support.
\item Algebraic cancellation can only remove primitive factors.  The only primitive whose represented support can equal all inputs of an outer component is that component's root primitive, and differentiation in the outer endpoint cancels it.
\end{enumerate}
Consequently no cut, contraction, or cancellation can create a new surviving primitive whose represented support is larger than that of an existing integration component.  In particular, after outer-endpoint differentiation no surviving denominator has represented support $V(T)$.
\end{lemma}

\begin{proof}
Proceed by induction on the number of cut and contraction operations.  In the uncut skeleton, $C=T$, every label is a singleton, and the support of the primitive rooted at $v$ is exactly the original integration subtree below $v$.

For a cut, the multiple-cut boundary formula decomposes the current component into one root component and finitely many pruned components.  Restrict the old labels to their respective components.  Their represented vertex sets are pairwise disjoint and no denominator joins two components.

For a contraction, collapse a connected component $S$ to one effective vertex and assign it the union of the labels in $S$.  This changes only the diagrammatic presentation.  Each reciprocal primitive factor already present in $S$ is copied to the inherited denominator list together with the component and vertex at which it arose; its represented support is therefore unchanged.  Primitives above the effective vertex retain exactly the union of the original labels below them.  Since the contraction adds no integration variable, it creates no new denominator and cannot enlarge a primitive support beyond the support of the pre-existing outer component.  A factor disappears only when rule~(G4) records an explicit numerator--denominator cancellation.

Algebraic cancellation only deletes factors.  Finally, outer-endpoint differentiation multiplies by the output frequency of the root primitive of the unique time-dependent outer component and cancels that primitive.  The label bookkeeping therefore proves all four assertions and the final statement.
\end{proof}

\begin{proposition}[No total-root denominator in a differentiated chart]\label{prop:no-root-denominator}
Fix a chart $\mathfrak c\in\mathfrak C_{\mathrm{FNO}}(T)$ associated with an original rooted integration tree $T$, and let $c$ denote its admissible cut pattern; the uncut case is $c=\varnothing$.  By \cref{lem:fno-chart-coverage}, this covers every differentiated skeleton-density and cut-density chart of the fixed FNO construction.  Equip every actual component tree $C$ with the label map $\operatorname{lab}_C$ from \cref{lem:primitive-denominator-inheritance}.  For $(C,v)$ put
\begin{equation}\label{eq:labelled-component-output}
 \Lambda_v^{C,\operatorname{lab}}(\xi)
 :=
 \sum_{w\succeq_Cv}\ \sum_{u\in\operatorname{lab}_C(w)}
 \xi_{\nu_T(u)}
 =
 \sum_{u\in\operatorname{supp}_T(C,v)}\xi_{\nu_T(u)}.
\end{equation}
For an uncontracted component this reduces to the usual integration-subtree output.

Let $\mathscr P_{T,c}^{\mathrm{act}}$ be the finite set of primitive positions $(C,v)$ in the active component trees of the chart.  Let $\mathscr P_{T,c}^{\mathrm{inh}}$ be the finite disjoint union, with multiplicity, of the primitive positions stored in the inherited denominator lists created by rule~(G3).  Each inherited record retains the component tree and vertex at which the primitive was formed, hence also its label map and the frequency $\Lambda_v^{C,\operatorname{lab}}$.  Put
\[
 \mathscr P_{T,c}
 :=
 \mathscr P_{T,c}^{\mathrm{act}}
 \sqcup
 \mathscr P_{T,c}^{\mathrm{inh}}.
\]
There is a finite submultiset
\[
 \mathscr D_{T,c}\subseteq\mathscr P_{T,c}
\]
indexing precisely the primitive factors that remain in the denominator after outer-endpoint differentiation and explicit algebraic cancellation.  Equivalently,
\[
 \mathscr D_{T,c}
 =
 \mathscr D_{T,c}^{\mathrm{act}}
 \sqcup
 \bigsqcup_{C\ \mathrm{contracted}}
 \mathscr D_C^{\mathrm{inh}},
\]
where $\mathscr D_C^{\mathrm{inh}}$ denotes the inherited list stored when $C$ is contracted, and with any factors removed by rule~(G4) omitted from the right-hand side.  The actual chart multiplier can be written
\begin{equation}\label{eq:differentiated-chart-multiplier}
 M_{T,c}^{\mathrm{dens}}(\xi)
 =
 \frac{N_{T,c}(\xi)}
 {\displaystyle\prod_{(C,v)\in\mathscr D_{T,c}}
  i\Lambda_v^{C,\operatorname{lab}}(\xi)}.
\end{equation}
Here $N_{T,c}$ is the numerator that remains in the actual chart after all cancellations; it includes the fixed smooth regularization multipliers, coordinate coefficients, surviving input-differential factors, and the numerator parts of contracted component coefficients.  By definition it contains no reciprocal primitive factor: every surviving factor of the form $1/(i\Lambda)$, including every inherited factor, is listed in $\mathscr D_{T,c}$.  After contractions and cancellations, $N_{T,c}$ is treated as the actual numerator of the chart.

For $c=\varnothing$, the uncut root primitive is absent from $\mathscr D_{T,\varnothing}$.  For a nontrivial cut, the root primitive of the outer root component is likewise absent, and every surviving denominator has
\[
 \operatorname{supp}_T(C,v)\subsetneq V(T).
\]
Consequently the full-tree total output
\[
 \Lambda_{\operatorname{root}(T)}^T(\xi)
 =\sum_{u\in V(T)}\xi_{\nu_T(u)}
\]
is never a denominator in \eqref{eq:differentiated-chart-multiplier}.

This proposition concerns differentiated FNO charts; the component $\Gamma_{r,N}$ remains defined by $\Gamma_{r,N}=-A_{r,N}$.
\end{proposition}

\begin{proof}
For the uncut skeleton, the Fourier formula before endpoint differentiation contains one denominator for every actual integration vertex.  Its outer primitive contributes
\[
 \frac1{i\Lambda_{\operatorname{root}(T)}^T(\xi)}.
\]
Differentiation in the outer endpoint multiplies by $i\Lambda_{\operatorname{root}(T)}^T(\xi)$ and cancels this factor.  The remaining factors are indexed by proper primitives.

For a nontrivial cut, use the multiple-cut boundary expansion \eqref{eq:cut-boundary-product}.  A summand consists of the increment of a root component $T_0$ multiplied by basepoint skeleton values of pruned components $T_1,\ldots,T_q$, all on one common regularization domain.  The outer endpoint derivative cancels the root primitive of $T_0$.  Every basepoint component represents a proper subset of $V(T)$.

Now apply \cref{lem:primitive-denominator-inheritance}.  Its label map makes the frequency of every effective input and every active or inherited primitive record well defined by \eqref{eq:labelled-component-output}.  A contraction preserves all pre-existing reciprocal primitive factors in the inherited list and creates none; explicit cancellation may only delete listed factors.  Every inherited component is a proper component of the original tree, so each inherited support is a strict subset of $V(T)$.  Thus no denominator with represented support $V(T)$ can reappear.  Collect the active and inherited surviving primitive records in $\mathscr D_{T,c}$ and all remaining numerator factors in $N_{T,c}$; this gives \eqref{eq:differentiated-chart-multiplier} without imposing an unsupported factorization of the numerator.

Finally, the Dynkin--Hall projection changes only the Lie-basis expansion.  By \cref{def:tree-species}, a Hall bracket vertex is not an integration vertex and creates no Fourier denominator.  This completes the proof.
\end{proof}

\begin{example}[Exact opposite-frequency root resonance]\label{ex:root-resonance}
Let $e_1,e_2\in E$ satisfy
\[
 [e_1,e_2]=Z,
\]
let $a\in\R$, and let $\ell\in\N$.  Define the $2\pi$-periodic smooth path
\[
 x(t)
 :=
 a\bigl(\sin(\ell t)e_1+\cos(\ell t)e_2\bigr).
\]
Then
\begin{equation}\label{eq:root-resonance-density}
 \frac12[x(t),\dot x(t)]
 =
 -\frac12a^2\ell Z.
\end{equation}
The constant term in \eqref{eq:root-resonance-density} is produced by the frequency pair $(\ell,-\ell)$, whose total output is zero.  Hence a differentiated degree-two density multiplier containing
\[
 \frac1{\xi_1+\xi_2}
\]
would be undefined on a contribution that is actually finite.
\end{example}

\begin{proof}
Differentiation gives
\[
 \dot x(t)
 =
 a\ell\bigl(\cos(\ell t)e_1-\sin(\ell t)e_2\bigr).
\]
Using bilinearity, antisymmetry, and $[e_1,e_2]=Z$,
\begin{align*}
 [x(t),\dot x(t)]
 &=
 a^2\ell
 [\sin(\ell t)e_1+\cos(\ell t)e_2,
   \cos(\ell t)e_1-\sin(\ell t)e_2]\\
 &=
 -a^2\ell
 \bigl(\sin^2(\ell t)+\cos^2(\ell t)\bigr)Z\\
 &=
 -a^2\ell Z.
\end{align*}
Division by two proves \eqref{eq:root-resonance-density}.
\end{proof}

The based endpoint package removes the basepoint-dependent part of the local degree-two density.  This calculation shows why augmentation belongs to the complete based density in \eqref{eq:based-prepared-density}, not to an isolated local term.

\begin{proposition}[Degree-two basepoint cancellation]\label{prop:degree-two-basepoint}
For a smooth cutoff path, one has
\[
 \Omega_{1,N}=X_N-X_N(0),
\]
\begin{equation}\label{eq:degree-two-prepared}
 D_{2,N}^{\mathrm{prep}}(t)
 =
 \frac12[\Omega_{1,N}(t),\dot X_N(t)],
\end{equation}
and
\begin{equation}\label{eq:degree-two-boundary}
 B_{2,N}(s,t)
 =
 -\frac12[\Omega_{1,N}(s),\Omega_{1,N}(t)].
\end{equation}
Consequently,
\begin{equation}\label{eq:degree-two-based-density}
 D_{2,N}^{\mathrm{prep}}(t)
 +\partial_tB_{2,N}(s,t)
 =
 \frac12[X_N(t)-X_N(s),\dot X_N(t)].
\end{equation}
\end{proposition}

\begin{proof}
The coefficient $c_1=1/2$ in \eqref{eq:bernoulli-coefficients} and \eqref{eq:prepared-density} give
\[
 D_{2,N}^{\mathrm{prep}}
 =
 \frac12[\Omega_{1,N},A_{1,N}]
 =
 \frac12[\Omega_{1,N},\dot X_N],
\]
which is \eqref{eq:degree-two-prepared}.  The homogeneous degree-two part of
\[
 \BCH(-\Omega_N(s),\Omega_N(t))
\]
is
\[
 \Omega_{2,N}(t)-\Omega_{2,N}(s)
 -\frac12[\Omega_{1,N}(s),\Omega_{1,N}(t)],
\]
so \eqref{eq:degree-two-boundary} follows from \eqref{eq:bch-package}.  Differentiating it gives
\[
 \partial_tB_{2,N}(s,t)
 =
 -\frac12[X_N(s)-X_N(0),\dot X_N(t)].
\]
Adding \eqref{eq:degree-two-prepared} yields
\begin{align*}
 D_{2,N}^{\mathrm{prep}}(t)+\partial_tB_{2,N}(s,t)
 &=
 \frac12[X_N(t)-X_N(0),\dot X_N(t)]\\
 &\quad-
 \frac12[X_N(s)-X_N(0),\dot X_N(t)]\\
 &=
 \frac12[X_N(t)-X_N(s),\dot X_N(t)],
\end{align*}
which is \eqref{eq:degree-two-based-density}.
\end{proof}

\subsection{Endpoint gauges and scheme dependence}

Different regularizations of the same first-level path need not give equal lifts.  Their exact algebraic relation is determined by a one-time higher-level gauge.  This algebraic classification is separate from the analytic convergence estimates.

\begin{theorem}[Endpoint-gauge classification]\label{thm:endpoint-gauge}
Let $\mathbf X$ and $\widetilde{\mathbf X}$ be step-$M$ multiplicative functionals on $[0,T]$ with the same first level.  Define
\[
 G(t):=\mathbf X_{0,t},
 \qquad
 \widetilde G(t):=\widetilde{\mathbf X}_{0,t},
 \qquad
 K(t):=G(t)^{-1}\widetilde G(t).
\]
Then
\begin{equation}\label{eq:gauge-normalization}
 K(0)=1,
 \qquad
 \pi_1K(t)=0,
\end{equation}
and
\begin{equation}\label{eq:endpoint-gauge-transform}
 \widetilde{\mathbf X}_{s,t}
 =
 K(s)^{-1}\mathbf X_{s,t}K(t).
\end{equation}
The normalized gauge $K$ is unique.  Conversely, every map $K:[0,T]\to G^{(M)}(E)$ satisfying \eqref{eq:gauge-normalization} transforms a multiplicative functional by \eqref{eq:endpoint-gauge-transform} into another multiplicative functional with the same first level.

If $\mathbf X$ and $\widetilde{\mathbf X}$ are continuous, then the normalized gauge is continuous.  Define the covariant gauge defect
\begin{equation}\label{eq:covariant-gauge-defect}
 K_{s,t}^{\mathbf X}
 :=
 K(s)^{-1}\mathbf X_{s,t}K(t)\mathbf X_{s,t}^{-1}.
\end{equation}
If
\begin{equation}\label{eq:gauge-holder-control}
 \norm{\pi_rK_{s,t}^{\mathbf X}}
 \leq
 C_K\abs{t-s}^{r\beta}
 \qquad(2\leq r\leq M),
\end{equation}
then the transform in \eqref{eq:endpoint-gauge-transform} preserves the weakly geometric $\beta$-H\"older class.  Condition \eqref{eq:gauge-holder-control} preserves weak geometricity; same-exponent strong geometricity requires approximation compatibility.
\end{theorem}

\begin{proof}
Multiplicativity gives
\[
 \mathbf X_{s,t}=G(s)^{-1}G(t),
 \qquad
 \widetilde{\mathbf X}_{s,t}=\widetilde G(s)^{-1}\widetilde G(t).
\]
Using $\widetilde G=GK$,
\begin{align*}
 \widetilde{\mathbf X}_{s,t}
 &=
 (G(s)K(s))^{-1}G(t)K(t)\\
 &=
 K(s)^{-1}G(s)^{-1}G(t)K(t)\\
 &=
 K(s)^{-1}\mathbf X_{s,t}K(t),
\end{align*}
which proves \eqref{eq:endpoint-gauge-transform}.  Since $G(0)=\widetilde G(0)=1$, one has $K(0)=1$.  Equality of the first levels gives
\[
 \pi_1\widetilde G(t)=\pi_1G(t),
\]
and the step-$M$ group product then gives $\pi_1K(t)=0$.  Setting $s=0$ in \eqref{eq:endpoint-gauge-transform} recovers $K(t)=G(t)^{-1}\widetilde G(t)$, proving uniqueness.

Conversely, direct multiplication yields
\begin{align*}
 \widetilde{\mathbf X}_{s,u}\widetilde{\mathbf X}_{u,t}
 &=
 K(s)^{-1}\mathbf X_{s,u}K(u)
 K(u)^{-1}\mathbf X_{u,t}K(t)\\
 &=
 K(s)^{-1}\mathbf X_{s,t}K(t)
 =
 \widetilde{\mathbf X}_{s,t}.
\end{align*}
The condition $\pi_1K=0$ leaves the first level unchanged.  Continuity follows from continuity of multiplication and inversion in the finite-step group.

By \eqref{eq:covariant-gauge-defect},
\[
 \widetilde{\mathbf X}_{s,t}
 =
 K_{s,t}^{\mathbf X}\mathbf X_{s,t}.
\]
The first level of $K_{s,t}^{\mathbf X}$ vanishes.  The graded tensor-product estimates, together with \eqref{eq:gauge-holder-control} and the graded H\"older estimates for $\mathbf X$, therefore give
\[
 \norm{\pi_r\widetilde{\mathbf X}_{s,t}}
 \leq
 C\abs{t-s}^{r\beta}
 \qquad(1\leq r\leq M).
\]
All factors are group-like, so weak geometricity is preserved.  No bounded-variation approximation is constructed by this estimate; hence strong geometricity at exponent $\beta$ requires a separate approximation-compatible condition.
\end{proof}

\begin{corollary}[Driver transformation]\label{cor:driver-transform}
Assume $G$, $\widetilde G$, and $K$ in \cref{thm:endpoint-gauge} are differentiable.  If
\[
 A:=G^{-1}\dot G,
 \qquad
 \widetilde A:=\widetilde G^{-1}\dot{\widetilde G},
\]
then
\begin{equation}\label{eq:driver-transform}
 \widetilde A
 =
 K^{-1}AK+K^{-1}\dot K.
\end{equation}
\end{corollary}

\begin{proof}
Differentiate $\widetilde G=GK$:
\[
 \dot{\widetilde G}=\dot GK+G\dot K.
\]
Left multiplication by $\widetilde G^{-1}=K^{-1}G^{-1}$ gives
\[
 \widetilde A
 =
 K^{-1}G^{-1}\dot GK+K^{-1}\dot K
 =
 K^{-1}AK+K^{-1}\dot K.
\]
\end{proof}

The endpoint values determine whether the gauge acts by conjugation on a loop holonomy.

\begin{corollary}[Loop gauge condition]\label{cor:loop-boundary}
Under the gauge transform \eqref{eq:endpoint-gauge-transform},
\begin{equation}\label{eq:loop-transform}
 \widetilde{\mathbf X}_{0,T}
 =
 K(0)^{-1}\mathbf X_{0,T}K(T).
\end{equation}
If
\[
 K(T)=K(0),
\]
then the loop holonomies are conjugate.  In the normalized convention $K(0)=1$, a periodic gauge gives equality of loop holonomies.  For a nonperiodic gauge, the transformation has no general conjugacy form.
\end{corollary}

\begin{proof}
Set $(s,t)=(0,T)$ in \eqref{eq:endpoint-gauge-transform}.  If $K(T)=K(0)$, equation \eqref{eq:loop-transform} is conjugation by $K(0)$.  If the two endpoint factors differ, the gauge relation has no general conjugacy form; a special holonomy may nevertheless be accidentally conjugate or have the same trace.
\end{proof}

 \section{Controlled extension to arbitrary finite tensor order}\label{sec:extension}

The FNO theorem constructs precisely the critical levels $1\leq r\leq m$.  Applications may require a fixed larger tensor step $M$.  The rough-path extension theorem supplies the unique controlled extension, while a separate smooth ODE realizes that extension at each cutoff and identifies its higher driver components.

The first statement is an imported consequence of the Lyons extension theorem in the exponent range $m<1/\beta<m+1$.

\begin{theorem}[Controlled rough-path extension]\label{thm:controlled-extension}
Let $\mathbf X$ be a step-$m$ strong geometric $\beta$-H\"older rough path with
\[
 \frac1{m+1}<\beta<\frac1m.
\]
For every integer $M\geq m$, there exists a step-$M$ multiplicative extension
\[
 \mathbf X^{(M)}:\Delta_2\longrightarrow G^{(M)}(E)
\]
whose projection to step $m$ is $\mathbf X$ and which satisfies, on every step-$m$ rough-path ball of radius $R$,
\begin{equation}\label{eq:controlled-extension-bound}
 \norm{\pi_r\mathbf X^{(M)}_{s,t}}
 \leq
 C_{M,R}\abs{t-s}^{r\beta}
 \qquad(1\leq r\leq M).
\end{equation}
The extension is unique among step-$M$ multiplicative functionals satisfying a control of the form \eqref{eq:controlled-extension-bound}.  The extension map is continuous, locally Lipschitz in the standard inhomogeneous rough-path metric on bounded sets, and preserves strong geometricity.
\end{theorem}

\begin{proof}
Put
\[
 p:=\frac1\beta.
\]
The exponent assumption gives
\[
 m<p<m+1,
 \qquad
 \lfloor p\rfloor=m.
\]
A $\beta$-H\"older multiplicative functional has finite $p$-variation controlled by
\[
 \omega(s,t):=C\abs{t-s}.
\]
The Lyons extension theorem \cite{Lyons1998,LyonsVictoir2007,FrizVictoir2010} therefore gives a unique multiplicative extension to every degree $M$, with factorial control
\[
 \norm{\pi_r\mathbf X^{(M)}_{s,t}}
 \leq
 C_{M,R}\omega(s,t)^{r/p}.
\]
Since $r/p=r\beta$, this is \eqref{eq:controlled-extension-bound}.  The same theorem gives continuity, local Lipschitz dependence in the inhomogeneous metric, and preservation of the geometric closure.
\end{proof}

For a smooth cutoff, the controlled extension can be constructed directly by retaining the critical driver and solving an ODE in the larger nilpotent group.  The essential estimate is that a degree-$r$ Picard term needs at least $\lceil r/m\rceil$ time integrations.

\begin{proposition}[Smooth finite-cutoff realization]\label{prop:smooth-controlled-extension}
Let
\[
 G_N^{(m)}:[0,2\pi]\longrightarrow G^{(m)}(E)
\]
be the smooth cutoff group path from \eqref{eq:cutoff-group-path}, and let
\[
 A_N^{(m)}:=(G_N^{(m)})^{-1}\dot G_N^{(m)}
 \in
 \bigoplus_{r=1}^mL_r(E).
\]
For $M\geq m$, let
\[
 \iota_m^M:
 \bigoplus_{r=1}^mL_r(E)
 \longrightarrow
 \bigoplus_{r=1}^ML_r(E)
\]
be the canonical graded linear inclusion.  Define $G_N^{(M)}$ by
\begin{equation}\label{eq:smooth-extension-ode}
 \dot G_N^{(M)}
 =
 G_N^{(M)}\iota_m^M(A_N^{(m)}),
 \qquad
 G_N^{(M)}(0)=1.
\end{equation}
Then:
\begin{enumerate}[label=(\roman*),leftmargin=2.4em]
\item the canonical quotient $q_m^M:G^{(M)}(E)\to G^{(m)}(E)$ satisfies
\[
 q_m^M(G_N^{(M)}(t))=G_N^{(m)}(t);
\]
\item the increments of $G_N^{(M)}$ satisfy
\begin{equation}\label{eq:smooth-extension-holder}
 \norm{\pi_r\bigl((G_N^{(M)}(s))^{-1}G_N^{(M)}(t)\bigr)}
 \leq
 C_{N,M}\abs{t-s}^{r\beta}
 \qquad(1\leq r\leq M);
\end{equation}
\item $G_N^{(M)}$ is the unique controlled step-$M$ extension of $G_N^{(m)}$ from \cref{thm:controlled-extension};
\item if
\[
 A_N^{(M)}:=(G_N^{(M)})^{-1}\dot G_N^{(M)},
 \qquad
 \Gamma_{r,N}^{(M)}:=-\pi_rA_N^{(M)},
\]
then
\begin{equation}\label{eq:no-new-driver-counterterms}
 \pi_rA_N^{(M)}=\Gamma_{r,N}^{(M)}=0
 \qquad(m<r\leq M).
\end{equation}
\end{enumerate}
Thus the extension may generate higher logarithmic coefficients, but it introduces no new driver counterterms above the critical step.
\end{proposition}

\begin{proof}
Equation \eqref{eq:smooth-extension-ode} gives
\begin{equation}\label{eq:extended-driver}
 A_N^{(M)}
 =
 \iota_m^M(A_N^{(m)}).
\end{equation}
The differential of $q_m^M$ satisfies
\[
 \dd q_m^M\circ\iota_m^M=\id
\]
on degrees at most $m$.  Hence both paths
\[
 q_m^M(G_N^{(M)})
 \quad\text{and}\quad
 G_N^{(m)}
\]
solve the same right-driven ODE with driver $A_N^{(m)}$ and initial value $1$.  Uniqueness of the smooth ODE gives
\[
 q_m^M(G_N^{(M)}(t))=G_N^{(m)}(t).
\]

Fix $0\leq s<t\leq2\pi$ and write
\[
 H_{s,t}:=(G_N^{(M)}(s))^{-1}G_N^{(M)}(t).
\]
The Picard expansion of \eqref{eq:smooth-extension-ode} on $[s,t]$ is a finite sum in the step-$M$ tensor algebra:
\begin{equation}\label{eq:picard-expansion}
 H_{s,t}
 =
 1+
 \sum_{n=1}^{M}
 \int_{s<u_1<\cdots<u_n<t}
 \iota_m^M(A_N^{(m)}(u_1))\cdots
 \iota_m^M(A_N^{(m)}(u_n))
 \,\dd u_1\cdots\dd u_n.
\end{equation}
Every factor in the integrand has homogeneous degree between $1$ and $m$.  Therefore a product of $n$ factors has degree at most $nm$.  A contribution to degree $r$ requires
\[
 nm\geq r,
 \qquad
 n\geq\left\lceil\frac rm\right\rceil.
\]
Let
\[
 C_N:=\sup_{u\in[0,2\pi]}\norm{A_N^{(m)}(u)}.
\]
Taking the degree-$r$ component of \eqref{eq:picard-expansion} gives
\begin{align}
 \norm{\pi_rH_{s,t}}
 &\leq
 \sum_{n\geq\lceil r/m\rceil}^{M}
 \frac{(C_N\abs{t-s})^n}{n!}
 \nonumber\\
 &\leq
 C_{N,M}\abs{t-s}^{\lceil r/m\rceil}.
 \label{eq:picard-degree-bound}
\end{align}
For $r>m$, the exponent assumption $\beta<1/m$ gives
\[
 r\beta<\frac rm\leq\left\lceil\frac rm\right\rceil.
\]
For $r\leq m$, smoothness gives the sharper bound $O(\abs{t-s})$, and $r\beta<1$.  On the compact interval, both cases imply
\[
 \abs{t-s}^{\lceil r/m\rceil}
 \leq
 C\abs{t-s}^{r\beta},
\]
which proves \eqref{eq:smooth-extension-holder}.

Thus $G_N^{(M)}$ lies in the controlled class of \cref{thm:controlled-extension}.  It projects to $G_N^{(m)}$, so uniqueness in that class identifies it with the Lyons extension.  Finally, \eqref{eq:extended-driver} has no components above degree $m$, which gives \eqref{eq:no-new-driver-counterterms}.
\end{proof}

\part{Fractional Brownian motion and measurable locality}

\section{A measurable scale--FNO lift on a finite interval}\label{sec:interval-lift}

The rough-enhancement theorem above is formulated for periodic scale paths.  A fractional Brownian sample is given on a finite interval.  We insert one fixed deterministic loop extension before applying the periodic scale analysis.  The extension is part of the construction scheme; it is also the first explicit source of possible nonlocality.

Let $E$ be a finite-dimensional normed space and let
\[
 C_0^{\alpha}([0,T];E)
 :=\{x\in C^{\alpha}([0,T];E):x(0)=0\}.
\]

\begin{definition}[Based loop extension]\label{def:loop-extension}
For $x\in C_0([0,T];E)$ define the $2\pi$-periodic path $\mathscr L_Tx$ on $[0,2\pi]$ by
\begin{equation}\label{eq:loop-extension}
 (\mathscr L_Tx)(\theta)
 :=
 \begin{cases}
 x(T\theta/\pi),&0\le\theta\le\pi,\\[1mm]
 x(T(2\pi-\theta)/\pi),&\pi\le\theta\le2\pi.
 \end{cases}
\end{equation}
The endpoint values are zero, so the formula defines a continuous loop.  The first half records the original path and the second half returns along the same trace.
\end{definition}

The extension must preserve the H\"older class uniformly because the scale estimates are applied after extension.

\begin{lemma}[H\"older control of the loop extension]\label{lem:loop-extension-bound}
For every $0<\alpha<1$ there is $C_{\alpha,T}<\infty$ such that
\[
 \norm{\mathscr L_Tx}_{C^{\alpha}(\T;E)}
 \le C_{\alpha,T}\norm x_{C^{\alpha}([0,T];E)}.
\]
The map $\mathscr L_T:C_0^{\alpha}([0,T];E)\to C^{\alpha}(\T;E)$ is linear and continuous.
\end{lemma}

\begin{proof}
On each half-circle, \eqref{eq:loop-extension} is the composition of $x$ with an affine map of Lipschitz constant $T/\pi$, so
\[
 \abs{(\mathscr L_Tx)(\theta)-(\mathscr L_Tx)(\phi)}
 \le [x]_{\alpha}(T/\pi)^{\alpha}\abs{\theta-\phi}^{\alpha}
\]
when $\theta,\phi$ lie in the same half.  If they lie in different halves, join both points to the nearest common endpoint $\pi$ or to the seam $0=2\pi$ and use
\[
 a^{\alpha}+b^{\alpha}\le2^{1-\alpha}(a+b)^{\alpha}.
\]
The geodesic distance on $\T$ is the minimum of the two resulting arc lengths.  This gives the stated seminorm bound; the uniform bound is immediate from $x(0)=0$.
\end{proof}

Let $\mathscr{RP}_{\beta,M}([0,T];E)$ denote the metric space of step-$M$ strong geometric $\beta$-H\"older rough paths on $[0,T]$, equipped with a standard homogeneous rough-path metric.  This is the target used for Borel measurability below.

\begin{definition}[Finite-interval scale--FNO lift]\label{def:interval-scale-fno}
Fix exponents
\[
 \frac1{m+1}<\beta<\gamma<\alpha\le\frac1m.
\]
Fix once and for all the based compact localization $\mathcal L$ from \cref{def:based-localization} and one Fourier-normal-ordering regularization $\mathcal R_\gamma$ from \cref{thm:unterberger}; these are part of the construction data and are suppressed from the notation below.  For $x\in C_0^{\alpha}([0,T];E)$, apply the periodic Littlewood--Paley analysis componentwise to $\mathscr L_Tx$.  Equivalently, apply the exact periodic coordinate of \cref{def:periodic-scale-coordinate} to each scalar component and use the value-channel projection in \cref{thm:amplitude-reinsertion}.  Denote the resulting dyadic scale path by
\[
 \mathfrak s_T(x)\in\cS^{\alpha}(E).
\]
Let $\mathbf Z(x)$ be its selected FNO lift from \cref{def:fno-scale-lift}.  Define
\[
 \mathfrak R_T^{\alpha,\beta}(x)_{s,t}
 :=
 \mathbf Z(x)_{\pi s/T,\,\pi t/T},
 \qquad 0\le s\le t\le T.
\]
Thus $\mathfrak R_T^{\alpha,\beta}(x)$ is the restriction of the loop lift to the first half-circle, pulled back to the original time interval.
\end{definition}

\begin{theorem}[Quantitative finite-interval enhancement]\label{thm:interval-enhancement}
The map in \cref{def:interval-scale-fno} has first level
\[
 \pi_1\mathfrak R_T^{\alpha,\beta}(x)_{s,t}=x(t)-x(s)
\]
and takes values in $\mathscr{RP}_{\beta,m}([0,T];E)$.  On every ball
\[
 \norm x_{C^{\alpha}}\le R
\]
its ultraviolet cutoffs satisfy
\[
 d_{\beta,m}^{\mathrm{lay}}
 \bigl(\mathfrak R_{T,N}^{\alpha,\beta}(x),
       \mathfrak R_T^{\alpha,\beta}(x)\bigr)
 \le C_{R,T}2^{-(\alpha-\gamma)N},
\]
\[
 d_{\beta,m}^{\mathrm{hom}}
 \bigl(\mathfrak R_{T,N}^{\alpha,\beta}(x),
       \mathfrak R_T^{\alpha,\beta}(x)\bigr)
 \le C_{R,T}2^{-(\alpha-\gamma)N/m}.
\]
The map $x\mapsto\mathfrak R_T^{\alpha,\beta}(x)$ is locally Lipschitz in the layerwise metric and locally $1/m$-H\"older in the homogeneous metric.  For every $M\ge m$, \cref{thm:controlled-extension} gives a unique step-$M$ extension in the standard controlled class.
\end{theorem}

\begin{proof}
By \cref{lem:loop-extension-bound} and the periodic dyadic estimates,
\[
 \norm{\mathfrak s_T(x)}_{\cS^{\alpha}}
 \le C_{\alpha,T}\norm x_{C^{\alpha}}.
\]
The synthesized first-level path of $\mathfrak s_T(x)$ is $\mathscr L_Tx$.  Therefore \cref{thm:rough-enhancement} gives a strong geometric $\beta$-H\"older lift on the circle, with the stated cutoff and stability bounds.  Restriction to $[0,\pi]$ and the affine time change $t\mapsto\pi t/T$ preserve multiplicativity and multiply the H\"older constants by a fixed power of $\pi/T$.  On the first half-circle, \eqref{eq:loop-extension} gives
\[
 (\mathscr L_Tx)(\pi t/T)=x(t),
\]
which proves the first-level identity.  The controlled extension is exactly \cref{thm:controlled-extension}.
\end{proof}

The stochastic application requires a map on the canonical continuous-path sigma-algebra, not merely continuity for the stronger H\"older norm.  The ultraviolet construction supplies such a version.

\begin{proposition}[Borel realization]\label{prop:borel-realization}
Let
\[
 \Omega_{\alpha,T}(E)
 :=\{x\in C_0([0,T];E):[x]_{\alpha}<\infty\}.
\]
This is a Borel subset of $C_0([0,T];E)$.  The map $\mathfrak R_T^{\alpha,\beta}$ has a Borel realization on $\Omega_{\alpha,T}(E)$ with values in $\mathscr{RP}_{\beta,m}([0,T];E)$.
\end{proposition}

\begin{proof}
The seminorm $[x]_{\alpha}$ is the supremum over rational pairs of the Borel functions
\[
 x\longmapsto
 \frac{\abs{x(t)-x(s)}}{\abs{t-s}^{\alpha}},
\]
so $\Omega_{\alpha,T}(E)$ is Borel.  For fixed $N$, the loop extension and the finite Littlewood--Paley cutoff are continuous from the uniform path topology into every fixed smooth H\"older norm, because only finitely many Fourier modes are retained.  The fixed FNO map is locally Lipschitz in the $C^{\gamma}$ norm by \cref{cor:fno-local-lipschitz}; equivalently, its smooth-cutoff coefficients are given by the normally convergent formulas of \cref{lem:termwise-differentiation}.  Hence each cutoff map
\[
 x\longmapsto\mathfrak R_{T,N}^{\alpha,\beta}(x)
\]
is Borel on $C_0([0,T];E)$.  By \cref{thm:interval-enhancement}, these maps converge in the rough-path metric at every $x\in\Omega_{\alpha,T}(E)$.  A pointwise limit of Borel maps into the Polish finite-step rough-path space is Borel.
\end{proof}

\begin{corollary}[A Borel rough lift for every fractional Brownian Hurst index]\label{cor:fbm-scale-fno}
Let $B^H$ be $d$-dimensional fractional Brownian motion on $[0,T]$ and let $0<\beta<H$.  There is a full-$\mu_H^{(d)}$ Borel set on which the fixed scale--FNO scheme produces a Borel step-$M$ strong geometric $\beta$-H\"older lift, where $M=\lfloor1/\beta\rfloor$.

If $H\le1/4$, the restriction of this lift to no positive-$\mu_H^{(d)}$ Borel subset can be local.
\end{corollary}

\begin{proof}
Choose $\alpha\in(\beta,H)$ so close to $\beta$ that no reciprocal threshold $1/n$ lies strictly between them, and avoid the countable set of reciprocal endpoints.  Let $m$ be determined by $1/(m+1)<\alpha\le1/m$, and choose
\[
 \beta<\beta_0<\gamma<\alpha,
 \qquad
 \frac1{m+1}<\beta_0.
\]
Such a choice is always possible because the reciprocal thresholds are discrete away from zero; when $\beta=1/n$, take $\alpha>\beta$ and then $m=n-1$.  Fractional Brownian sample paths belong to $C_0^{\alpha}([0,T];\R^d)$ almost surely.  Apply \cref{prop:borel-realization,thm:interval-enhancement} at exponent $\beta_0$ and then \cref{thm:controlled-extension} to the step $M=\lfloor1/\beta\rfloor$.  The resulting $\beta_0$-H\"older bounds imply the weaker $\beta$-H\"older bounds on the compact interval.  This gives the full-measure Borel construction.

If its restriction to a positive-measure Borel set were local, its off-diagonal second level would be a positive-domain local Chen coordinate, contradicting \cref{thm:phase-transition}(i).
\end{proof}

The fixed loop extension and FNO regularization give a measurable lift for every $H>0$.  The theorem below shows that every lift is necessarily nonlocal on positive-measure domains when $H\leq1/4$.

\section{Locality on a Borel domain}

For $x\in\Omega_d$, $0\le s\le t\le T$, and $J\subset\{1,\dots,d\}$, define the increment restriction
\[
  r_{s,t}^J(x):=\bigl(x_u^a-x_s^a\bigr)_{u\in[s,t],\,a\in J}.
\]
It takes values in the corresponding continuous-path space with its Borel sigma-algebra.

\begin{definition}[Positive-domain local Chen coordinate]\label{def:positive-domain-area}
Let $B\subset\Omega_d$ be Borel. A \emph{positive-domain local level-two Chen coordinate} on $B$ is a family of finite Borel maps
\[
  \A^{ab}_{s,t}:B\longrightarrow\R,
  \qquad 0\le s\le t\le T,
\]
for at least one pair $a\ne b$, satisfying:
\begin{enumerate}[label=(\roman*)]
\item for every $s\le u\le t$ and every $x\in B$,
\begin{equation}\label{eq:chen-area}
  \A^{ab}_{s,t}(x)
  =\A^{ab}_{s,u}(x)+\A^{ab}_{u,t}(x)
   +x^a_{s,u}x^b_{u,t};
\end{equation}
\item $\A^{ab}_{s,t}$ is measurable with respect to
\[
  \sigma\bigl(r_{s,t}^{\{a,b\}}\vert_B\bigr).
\]
\end{enumerate}
\end{definition}

The second condition is the relative-Borel version of interval locality.

\begin{lemma}[Relative locality is fibrewise]\label{lem:fibrewise}
Let $r:B\to Y$ be a map into a measurable space and let $F:B\to\R$ be finite and $\sigma(r)$-measurable. Then
\[
  r(x)=r(y)\quad\Longrightarrow\quad F(x)=F(y).
\]
\end{lemma}

\begin{proof}
If $F(x)<F(y)$, choose $q\in\mathbb Q$ strictly between them. The set $\{F<q\}$ belongs to $\sigma(r)$ and is therefore saturated under the equivalence relation $r(x)=r(y)$, a contradiction. Interchanging $x$ and $y$ gives equality.
\end{proof}

\section{Gaussian four corners and measurable rectangles}

\subsection{Finite-dimensional Gaussian resampling}

Let $(W,\Hh,\mu)$ be an abstract Wiener space, with the Cameron--Martin space continuously embedded in $W$.

\begin{lemma}[Cameron--Martin coordinate resampling]\label{lem:resampling}
Let $U\subset\Hh$ be finite-dimensional with orthonormal basis $(e_i)_{i=1}^m$. One can realise a $W$-valued variable $X\sim\mu$ as
\[
  X=R_U+\sum_{i=1}^m\xi_i e_i,
\]
where $\xi=(\xi_i)$ is standard Gaussian and independent of the Gaussian residual $R_U$. Replacing $\xi$ by an independent copy $\xi'$ while retaining $R_U$ produces another variable $X'\sim\mu$.
\end{lemma}

\begin{proof}
Write $i:\Hh\hookrightarrow W$ for the Cameron--Martin embedding and identify $e_i$ with $i(e_i)$. Let $I(e_i)$ be the Paley--Wiener Gaussian functionals. They admit simultaneous measurable-linear versions on a common full-measure linear subspace. On that subspace set
\[
  \xi_i=I(e_i),
  \qquad
  R_U=X-\sum_i I(e_i)e_i.
\]
For every $\ell\in W^*$, the defining covariance identity of the Paley--Wiener map gives
\[
  \Cov\bigl(\ell(X),I(e_i)\bigr)=\ell(e_i),
\]
and hence
\[
  \Cov\bigl(\ell(R_U),\xi_i\bigr)
  =\ell(e_i)-\sum_j\ell(e_j)\delta_{ij}=0.
\]
Every finite family consisting of continuous linear images of $R_U$ and the coordinates of $\xi$ is jointly Gaussian. Vanishing cross-covariance therefore implies that the $W$-valued Gaussian variable $R_U$ and $\xi$ are independent. If $\xi'$ is an independent standard Gaussian vector, then for every $\ell_1,\ell_2\in W^*$ the covariance of
$R_U+\sum_i\xi_i'e_i$ equals that of $X$. Centred Gaussian measures on a separable Banach space are determined by their continuous-linear covariances, so the resampled variable has law $\mu$.
\end{proof}

\begin{lemma}[Gaussian four-corner lower bound]\label{lem:four-corner}
Let $(W_i,\Hh_i,\mu_i)$, $i=1,2$, be abstract Wiener spaces, let $(S,\nu)$ be a probability space, and let
\[
  B\subset W_1\times W_2\times S
\]
be Borel with
\[
  \beta:=(\mu_1\otimes\mu_2\otimes\nu)(B)>0.
\]
Fix finite-dimensional $U\subset\Hh_1$ and $V\subset\Hh_2$. Use \cref{lem:resampling} to construct $X,X'\in W_1$, $Y,Y'\in W_2$, and $Z\sim\nu$, retaining the same residual in each resampled pair. Then
\[
  \Pp\bigl((X^\varepsilon,Y^\delta,Z)\in B
  \text{ for all }\varepsilon,\delta\in\{0,1\}\bigr)
  \ge\beta^4.
\]
The bound is independent of $\dim U$ and $\dim V$.
\end{lemma}

\begin{proof}
Condition on both Gaussian residuals and on $Z$, and write
\[
  f(\xi,\eta)=\one_B\left(R_1+\sum_i\xi_i e_i,
  R_2+\sum_j\eta_j f_j,Z\right).
\]
The conditional four-corner probability is
\[
  \mathcal B(f):=
  \E f(\xi,\eta)f(\xi',\eta)f(\xi,\eta')f(\xi',\eta').
\]
Two applications of Cauchy--Schwarz give
\[
\begin{aligned}
  \mathcal B(f)
  &=\E_{\eta,\eta'}
    \left(\E_\xi f(\xi,\eta)f(\xi,\eta')\right)^2\\
  &\ge
    \left(\E_\xi\left(\E_\eta f(\xi,\eta)\right)^2\right)^2
  \ge \left(\E_{\xi,\eta}f\right)^4.
\end{aligned}
\]
Taking expectation over the residual variables and applying Jensen once more yields the lower bound $\beta^4$.
\end{proof}

\subsection{A positive-domain rectangle theorem}

For a finite-dimensional Hilbert pair $U,V$, the Hilbert--Schmidt norm of a bilinear form $b:U\times V\to\R$ is the Frobenius norm of its matrix in orthonormal bases.

\begin{lemma}[Moments of a bilinear Gaussian chaos]\label{lem:bilinear-moments}
Let $g\in\R^m$ and $q\in\R^n$ be independent standard Gaussian vectors, let $A$ be an $m\times n$ matrix, and put $Q=g^{\mathsf T}Aq$ and $\sigma=\norm{A}_{\mathrm F}$. Then
\[
  \E Q^2=\sigma^2,
  \qquad
  \E Q^4\le9\sigma^4,
  \qquad
  \E\abs Q\ge\frac{\sigma}{3}.
\]
\end{lemma}

\begin{proof}
The second moment is immediate. Conditional on $g$, the variable $Q$ is Gaussian with variance $\norm{A^{\mathsf T}g}^2$, so
\[
  \E Q^4=3\E\norm{A^{\mathsf T}g}^4
  =3\left((\operatorname{tr}AA^{\mathsf T})^2+2\operatorname{tr}(AA^{\mathsf T})^2\right)
  \le9\sigma^4.
\]
Interpolation between $L^1,L^2,L^4$ gives
\[
  \E Q^2\le (\E\abs Q)^{2/3}(\E Q^4)^{1/3},
\]
which implies the last bound.
\end{proof}

We use the following degree-two case of the Carbery--Wright inequality \cite{CarberyWright2001}: for a real polynomial $P$ of degree at most two under a Gaussian probability measure,
\begin{equation}\label{eq:CW}
  \Pp(\abs P\le t)
  \le C_{\mathrm{CW}}
  \left(\frac{t}{\E\abs P}\right)^{1/2}.
\end{equation}

\begin{theorem}[Measurable rectangles are uniformly Hilbert--Schmidt]\label{thm:rectangle-HS}
Let $B\subset W_1\times W_2\times S$ be Borel with product Gaussian probability $\beta>0$, and let $F:B\to\R$ be finite and Borel. Let $\mathcal P$ be any collection of finite-dimensional pairs $(U,V)$ with $U\subset\Hh_1$ and $V\subset\Hh_2$. Suppose that for every $(U,V)\in\mathcal P$ there is a deterministic bilinear form $b_{U,V}$ such that whenever
\[
  (x,y,z),\ (x+h,y,z),\ (x,y+k,z),\ (x+h,y+k,z)\in B,
\]
with $h\in U$, $k\in V$, one has
\begin{equation}\label{eq:rectangle-identity}
\begin{aligned}
  &F(x+h,y+k,z)-F(x+h,y,z)\\
  &\qquad-F(x,y+k,z)+F(x,y,z)=b_{U,V}(h,k).
\end{aligned}
\end{equation}
Then there is a constant $C_{B,F}<\infty$, independent of $U,V$, such that
\[
  \norm{b_{U,V}}_{\HS(U,V)}\le C_{B,F}.
\]
\end{theorem}

\begin{proof}
Extend $F$ by zero to the ambient product and denote the extension by $\widetilde F$. Use \cref{lem:resampling} in $U,V$ and define the random rectangle
\[
\begin{aligned}
  \Rect_{U,V}:={}&\widetilde F(X',Y',Z)-\widetilde F(X',Y,Z)\\
  &-\widetilde F(X,Y',Z)+\widetilde F(X,Y,Z).
\end{aligned}
\]
Every corner has the same marginal law, so
\[
  \Pp(\abs{\Rect_{U,V}}>4M)
  \le4\Pp(\abs{\widetilde F}>M).
\]
Choose $M=M_{B,F}$ such that the right side is at most $\beta^4/2$. By \cref{lem:four-corner}, with probability at least $\beta^4/2$ all four corners lie in $B$ and $\abs{\Rect_{U,V}}\le4M$.

Let $A$ be the matrix of $b_{U,V}$ in orthonormal bases. On the four-corner event,
\[
  X'-X=\sqrt2\,g,
  \qquad
  Y'-Y=\sqrt2\,q,
  \qquad
  \Rect_{U,V}=2g^{\mathsf T}Aq.
\]
Thus, for $Q=g^{\mathsf T}Aq$,
\[
  \Pp(\abs Q\le2M)\ge\frac{\beta^4}{2}.
\]
If $\sigma=\norm A_{\mathrm F}>0$, \cref{lem:bilinear-moments} and \eqref{eq:CW} give
\[
  \frac{\beta^4}{2}
  \le C_{\mathrm{CW}}
       \left(\frac{6M}{\sigma}\right)^{1/2}.
\]
Hence $\sigma\le C M\beta^{-8}$, uniformly in the dimensions. The case $\sigma=0$ is trivial.
\end{proof}

\section{Locality and Chen force the Young rectangle}

For a scalar path $h$ acting in coordinate $a$, write $\tau_h^{(a)}x=x+he_a$.

\begin{lemma}[Separated-support rectangle]\label{lem:young-rectangle}
Let $B\subset\Omega_d$ carry a local level-two Chen coordinate and put
\[
  F(x)=\A^{12}_{0,T}(x).
\]
Let $h,k\in C^1([0,T])$ satisfy
\[
  h(0)=k(0)=0,
  \qquad
  \supp\dot h\cap\supp\dot k=\varnothing.
\]
Whenever all four paths
\[
  x,\quad \tau_h^{(1)}x,\quad \tau_k^{(2)}x,
  \quad \tau_h^{(1)}\tau_k^{(2)}x
\]
belong to $B$, one has
\begin{equation}\label{eq:young-rectangle}
\begin{aligned}
  &F(\tau_h^{(1)}\tau_k^{(2)}x)-F(\tau_h^{(1)}x)
   -F(\tau_k^{(2)}x)+F(x)\\
  &\hspace{42mm}=\int_0^T h(t)\,\dd k(t).
\end{aligned}
\end{equation}
The identity is simultaneous for finite-dimensional shift spaces whose derivative supports lie in two fixed disjoint compact sets.
\end{lemma}

\begin{proof}
The two compact derivative supports have positive distance. Choose a finite partition
\[
  0=t_0<\cdots<t_N=T
\]
so fine that on each cell $I_r=[t_r,t_{r+1}]$ at least one of $h,k$ is constant. Iterating \eqref{eq:chen-area} gives
\[
  \A^{12}_{0,T}
  =\sum_r\A^{12}_{t_r,t_{r+1}}
   +\sum_{r<q}x^1_{t_r,t_{r+1}}x^2_{t_q,t_{q+1}}.
\]
Take the mixed four-corner difference. Every local interval term cancels by \cref{lem:fibrewise}: on each cell one of the two increment paths is unchanged. The remaining cross terms give
\[
  \sum_{r<q}\bigl(h(t_{r+1})-h(t_r)\bigr)
                 \bigl(k(t_{q+1})-k(t_q)\bigr)
  =\sum_q h(t_q)\bigl(k(t_{q+1})-k(t_q)\bigr).
\]
On every cell where $k$ varies, $h$ is constant, so the last sum is exactly $\int h\,\dd k$. One partition works for all linear combinations whose derivative supports remain in the same two compact sets.
\end{proof}

\section{Cameron--Martin packets and the negative theorem}

The argument is invariant under deterministic time rescaling. More precisely, if $a=T/(2\pi)$ and
\[
  (\mathsf S_a x)(u):=a^{-H}x(au),\qquad 0\le u\le2\pi,
\]
then $\mathsf S_a$ sends the $H$-fBm law on $[0,T]$ to the $H$-fBm law on $[0,2\pi]$. A second-level coordinate is transported by
\[
  (\mathsf S_a^{(2)}\A)_{u,v}(\mathsf S_a x)
  :=a^{-2H}\A_{au,av}(x).
\]
This preserves Borel measurability, relative locality, Chen's identity, and positivity of the measure of the domain. It is therefore enough to prove the negative result on $[0,2\pi]$.

Let $\Hh_H$ be the one-dimensional fractional Brownian Cameron--Martin path space on $[0,2\pi]$.

\begin{lemma}[Periodic-loop Cameron--Martin estimate]\label{lem:periodic-CM}
Let $0<H<1/2$. There is $C_H<\infty$ such that every smooth $2\pi$-periodic loop $h$ with $h(0)=h(2\pi)=0$ belongs to $\Hh_H$ and
\begin{equation}\label{eq:periodic-CM}
  \norm h_{\Hh_H}^2
  \le C_H\sum_{r\in\mathbb Z\setminus\{0\}}
  \abs r^{2H-1}\abs{\widehat{\dot h}(r)}^2,
\end{equation}
where $\widehat{\dot h}(r)$ is the coefficient of $e^{irt}$.
\end{lemma}

\begin{proof}
Picard proves that the fractional Brownian Cameron--Martin space on $[0,1]$ is equivalent to the periodic Sobolev model and that
\[
  t,\qquad n^{-H-1/2}(1-\cos 2\pi nt),\qquad
  n^{-H-1/2}\sin 2\pi nt
\]
form a Riesz basis; see \cite[Theorems~6.9 and~6.12 in the published version, pp.~48--49]{Picard2011}. If $h(0)=h(1)=0$, the coefficient of $t$ is zero and $h$ has a unique expansion
\[
  h(t)=\sum_{n\ge1}\bigl(a_n(1-\cos 2\pi nt)+b_n\sin 2\pi nt\bigr)
\]
with
\[
  \norm h_{\Hh_H([0,1])}^2\asymp_H
  \sum_{n\ge1}n^{2H+1}(\abs{a_n}^2+\abs{b_n}^2).
\]
Differentiating this expansion shows that the right-hand side is, up to a fixed numerical factor, exactly
\[
  \sum_{r\in\mathbb Z\setminus\{0\}}\abs r^{2H-1}
       \abs{\widehat{\dot h}(r)}^2.
\]
Fractional Brownian self-similarity identifies the Cameron--Martin spaces on $[0,1]$ and on the fixed interval $[0,2\pi]$ under
$h(t)\mapsto (2\pi)^{-H}h(2\pi t)$. Rescaling changes only the constant and proves \eqref{eq:periodic-CM}.
\end{proof}

Choose non-negative, non-zero $m,n\in C^\infty(\mathbb T)$ with disjoint compact supports in $(0,2\pi)$ and Fourier series
\[
  m(t)=\sum_qc_qe^{iqt},
  \qquad
  n(t)=\sum_qd_qe^{iqt}.
\]
Then $c_0,d_0>0$. For an integer $M$, put $m_M(t)=m(Mt)$ and $n_M(t)=n(Mt)$, where the arguments are read modulo $2\pi$. Their supports are the disjoint inverse images of $\supp m$ and $\supp n$ under the same covering map. For $1\le k\le K$, define complex loops by
\[
  \dot h_k^+=m_Me^{ikt},\quad h_k^+(0)=0,
  \qquad
  \dot g_k^-=n_Me^{-ikt},\quad g_k^-(0)=0.
\]
If $M>K$, neither derivative has a zero Fourier mode, so both paths close at $2\pi$.

\begin{lemma}[Packet bounds and critical divergence]\label{lem:packets}
There are constants $C,c>0$, depending only on $H,m,n$, such that for every $K$ one can choose an integer $M\ge CK$ with $M>2K$ for which the following holds. Write
\[
  h_{k,1}=\Re h_k^+,\quad h_{k,2}=\Im h_k^+,
  \qquad
  g_{k,1}=\Re g_k^-,\quad g_{k,2}=\Im g_k^-.
\]
Define $U_K,V_K:\R^{2K}\to\Hh_H$ by
\[
  U_Ke_{k,a}=k^{1/2-H}h_{k,a},
  \qquad
  V_Ke_{k,a}=k^{1/2-H}g_{k,a}.
\]
Then
\[
  \norm{U_K}+\norm{V_K}\le C,
\]
and for the Young form $b(h,g)=\int_0^{2\pi}h\,\dd g$,
\begin{equation}\label{eq:packet-lower}
  \norm{U_K^*bV_K}_{\HS}^2
  \ge c\sum_{k=1}^K k^{-4H}.
\end{equation}
\end{lemma}

\begin{proof}
We first work in the complexifications. For
\[
  h=\sum_{k=1}^K a_kk^{1/2-H}h_k^+,
\]
the Fourier frequencies of $\dot h_k^+$ are $k+qM$. Since $M>2K$, the pairs $(k,q)$ give distinct frequencies. By \cref{lem:periodic-CM},
\[
\begin{aligned}
  \norm h_{\Hh_H}^2
  &\le C_H\sum_{k,q}\abs{a_k}^2k^{1-2H}
                 \abs{c_q}^2\abs{k+qM}^{2H-1}.
\end{aligned}
\]
For $q=0$ the product of the two powers of $k$ is one. For $q\ne0$, $M>2K$ implies
\[
  \abs{k+qM}\ge \abs q\,k.
\]
Since $2H-1<0$,
\[
  k^{1-2H}\abs{k+qM}^{2H-1}\le\abs q^{2H-1}.
\]
The rapid decay of $(c_q)$ therefore gives
$\norm h_{\Hh_H}^2\le C\sum_k\abs{a_k}^2$. The $g$-family is identical, and passing to real and imaginary parts changes only the constant. Hence $U_K,V_K$ are uniformly bounded.

At equal packet index,
\begin{equation}\label{eq:Ck}
  C_k:=\int_0^{2\pi}h_k^+\,\dd g_k^-
  =2\pi\sum_{q\in\mathbb Z}
       \frac{c_qd_{-q}}{i(k+qM)}.
\end{equation}
The $q=0$ term has magnitude $2\pi c_0d_0/k$. For $q\ne0$ and $M>2K$,
\[
  \abs{k+qM}\ge \tfrac12\abs q M,
\]
so the remainder in \eqref{eq:Ck} is bounded by
\[
  \frac{C}{M}\sum_{q\ne0}\frac{\abs{c_qd_{-q}}}{\abs q}
  \le \frac{C_{m,n}}{M}.
\]
Choose the constant in $M\ge CK$ so large that
$\abs{C_k}\ge c/k$ for all $1\le k\le K$.

Let $\widetilde B_k$ be the real $2\times2$ matrix of $b$ on the unnormalised pair
$(h_{k,1},h_{k,2})\times(g_{k,1},g_{k,2})$. Complex bilinearity gives
\[
\begin{aligned}
  C_k={}&b(h_{k,1},g_{k,1})-b(h_{k,2},g_{k,2})\\
       &+i\bigl(b(h_{k,1},g_{k,2})+b(h_{k,2},g_{k,1})\bigr),
\end{aligned}
\]
and hence $\norm{\widetilde B_k}_{\mathrm F}^2\ge\abs{C_k}^2/2$. Normalising both inputs multiplies this block by $k^{1-2H}$, so its squared Frobenius norm is at least $ck^{-4H}$. The squared Hilbert--Schmidt norm of the full matrix is the sum of the squares of all its entries and is therefore bounded below by the sum of these diagonal block contributions, which proves \eqref{eq:packet-lower}.
\end{proof}

\begin{proof}[Proof of \cref{thm:phase-transition}(i)]
Assume $\beta=\mu_H^{(d)}(B)>0$.  If the given off-diagonal coordinate is indexed by $(a,b)$, relabel the independent scalar components and assume without loss of generality that $(a,b)=(1,2)$.  Put $F(x)=\A^{12}_{0,2\pi}(x)$.  Since the components are independent, the law has the product form $\mu_H^{(d)}=\mu_H^{\otimes d}$.  Split the first two scalar coordinates from the remaining ones:
\[
  \Omega_d=W_1\times W_2\times S,
  \qquad
  \mu_H^{(d)}=\mu_H\otimes\mu_H\otimes\nu.
\]
For each $K$, let $U=\Ran U_K$ and $V=\Ran V_K$ be the packet subspaces. Their derivative supports lie in two fixed disjoint compact sets, so \cref{lem:young-rectangle} gives the rectangle identity \eqref{eq:rectangle-identity} with
\[
  b_{U,V}(h,k)=\int_0^{2\pi}h\,\dd k.
\]
By \cref{thm:rectangle-HS}, $\norm{b_{U,V}}_{\HS}\le C_{B,F}$ uniformly in $K$. The Hilbert--Schmidt ideal property and the uniform Bessel bounds give
\[
  \norm{U_K^*bV_K}_{\HS}
  \le \norm{U_K}\norm{V_K}\norm{b_{U,V}}_{\HS}
  \le C C_{B,F}.
\]
But \cref{lem:packets} gives
\[
  \norm{U_K^*bV_K}_{\HS}^2
  \ge c\sum_{k=1}^Kk^{-4H}.
\]
The right side grows like $K^{1-4H}$ if $H<1/4$ and like $\log K$ if $H=1/4$. This contradiction proves $\mu_H^{(d)}(B)=0$.
\end{proof}

\section{Local additive rigidity above order one half}

For an interval $I=[s,t]$, let
\[
  \F_I:=\sigma\bigl(B_u^{H,a}-B_s^{H,a}:u\in[s,t],\ 1\le a\le d\bigr)
\]
completed under the fractional Brownian law.

Assume in this section that $0<H\le1/2$.  For $0<H<1/2$, we use the established Gaussian maximal-correlation estimate from \cite{ChevyrevFerrucci2026}: for two equal intervals of length $\tau$ separated by a gap $g>0$,
\begin{equation}\label{eq:max-correlation}
  \rho\le 1\wedge C_H\left(\frac{\tau}{g}\right)^{2-2H}.
\end{equation}
For $H=1/2$, disjoint Brownian increment sigma-fields are independent, so their maximal correlation is zero.  If $F,G$ are centred square-integrable functionals of two Gaussian subspaces with maximal correlation at most $\rho$, Gebelein's inequality \cite{Gebelein1941} gives
\begin{equation}\label{eq:gebelein}
  \abs{\E[FG]}\le\rho\norm F_2\norm G_2.
\end{equation}

\begin{theorem}[Local additive rigidity]\label{thm:additive-rigidity}
Assume $0<H\le1/2$, let $V$ be finite-dimensional, and let
\[
 (A_{s,t})_{0\le s\le t\le T,\ s,t\in\mathbb Q}
\]
be a $V$-valued family indexed by rational endpoints.  For every rational $s<t$, assume that the random variable $A_{s,t}$ is $\F_{[s,t]}$-measurable.  Assume moreover that there is a common event $\Omega_0$ of probability one on which:
\begin{enumerate}[label=(\roman*),leftmargin=2.4em]
\item $A_{s,t}=A_{s,u}+A_{u,t}$ for all rational $s\le u\le t$;
\item for some $\theta>1/2$,
\[
 R_\theta^{\mathbb Q}(A)
 :=
 \sup_{\substack{s<t\\s,t\in\mathbb Q}}
 \frac{\norm{A_{s,t}}}{\abs{t-s}^\theta}
 <\infty.
\]
\end{enumerate}
Then the rational restriction admits a unique continuous additive extension
\[
 \widetilde A:\{(s,t):0\le s\le t\le T\}\longrightarrow V.
\]
There is a deterministic $\theta$-H\"older path $a:[0,T]\to V$ and a common event of probability one on which
\begin{equation}\label{eq:additive-rigidity-extension}
 \widetilde A_{s,t}=a(t)-a(s)
 \qquad(0\le s\le t\le T).
\end{equation}
The conclusion concerns the unique continuous extension of the rational family; no value assigned to the original family at an irrational endpoint is used or determined unless continuity is assumed separately.
\end{theorem}

\begin{proof}
It is enough to prove the scalar statement after applying a basis of $V^*$. Fix rational $a<b$, write $L=b-a$, and for $N\ge2$ set
\[
 t_i=a+iL/N,\qquad I_i=[t_i,t_{i+1}],\qquad h_N=L/N.
\]
Let $\chi_M(x)=(-M)\vee(x\wedge M)$. For an integer $K\ge1$, put
\[
 M_N=Kh_N^\theta,
 \qquad
 Y_{N,i}^K=\chi_{M_N}(A_{t_i,t_{i+1}}).
\]
Split the sum by parity:
\[
 S_{N,p}^K=\sum_{i\equiv p\, (2)}Y_{N,i}^K,
 \qquad
 S_N^K=S_{N,0}^K+S_{N,1}^K.
\]
Same-parity intervals at index distance $2r$ have length $h_N$ and gap $(2r-1)h_N$.  If $0<H<1/2$, \eqref{eq:max-correlation} gives
\[
 \rho_r\le C_H(2r-1)^{-(2-2H)},
 \qquad
 \sum_{r\ge1}\rho_r<\infty,
\]
because $2-2H>1$.  If $H=1/2$, the corresponding Brownian increment sigma-fields are independent and we set $\rho_r=0$.  In both cases, Gebelein's inequality and $\abs{Y_{N,i}^K}\le M_N$ give
\[
 \abs{\Cov(Y_{N,i}^K,Y_{N,j}^K)}
 \le \rho_rM_N^2,
 \qquad
 \sum_{r\ge1}\rho_r<\infty.
\]
Consequently
\begin{equation}\label{eq:variance-rigidity}
 \Var(S_N^K)
 \le C_H N M_N^2
 =C_HK^2L^{2\theta}N^{1-2\theta}.
\end{equation}
Along $N=2^n$ the right-hand side is summable. Chebyshev and Borel--Cantelli, applied to the thresholds $1/m$, $m\in\mathbb N$, yield
\begin{equation}\label{eq:bc-rigidity}
 S_{2^n}^K-\E S_{2^n}^K\longrightarrow0
 \qquad\text{almost surely}.
\end{equation}

Let $E_K=\{R_\theta^{\mathbb Q}(A)\le K\}$. On $E_K\cap\Omega_0$ no clipping occurs, and rational additivity gives $S_N^K=A_{a,b}$. Whenever $\Pp(E_K)>0$, put $c_n^K=\E S_{2^n}^K$. On $E_K\cap\Omega_0$, \eqref{eq:bc-rigidity} gives $A_{a,b}-c_n^K\to0$ almost surely. The deterministic sequence $(c_n^K)$ is Cauchy: for $\varepsilon>0$ and large $n,m$, each of the sets
\[
 E_K\cap\Omega_0\cap\{\abs{A_{a,b}-c_n^K}\le\varepsilon\},
 \qquad
 E_K\cap\Omega_0\cap\{\abs{A_{a,b}-c_m^K}\le\varepsilon\}
\]
has probability arbitrarily close to $\Pp(E_K)$, so their intersection has positive probability and $\abs{c_n^K-c_m^K}\le2\varepsilon$. Thus $c_n^K\to c_{a,b}^K$ and $A_{a,b}=c_{a,b}^K$ almost surely on $E_K\cap\Omega_0$.

Repeat the argument for every integer $K$ with positive $\Pp(E_K)$. The events are nested; hence their deterministic constants agree on every non-null intersection. Since $E_K\uparrow\Omega$ up to a null set, $A_{a,b}$ is one deterministic constant $c_{a,b}$ almost surely. Taking a countable intersection over rational endpoints gives deterministic additive constants on all rational intervals. Define $a(q)=c_{0,q}$ for rational $q$. On any non-null $E_K$,
\[
 \abs{a(r)-a(q)}=\abs{c_{q,r}}\le K\abs{r-q}^\theta,
\]
so $a$ extends uniquely to a deterministic $\theta$-H\"older path on $[0,T]$.

Set
\[
 \widetilde A_{s,t}:=a(t)-a(s)
 \qquad(0\le s\le t\le T).
\]
On a common full-probability event this agrees with $A_{s,t}$ for every rational pair.  It is continuous and additive.  Any other continuous extension agreeing on rational pairs coincides with it by density of $\mathbb Q^2$ in the time simplex.  This proves uniqueness and \eqref{eq:additive-rigidity-extension}.
\end{proof}

\begin{corollary}[Continuous-process form of local additive rigidity]\label{cor:additive-rigidity-continuous}
Assume the hypotheses of \cref{thm:additive-rigidity}.  Suppose in addition that a two-parameter process $(A_{s,t})_{0\le s\le t\le T}$ is already defined at all real endpoints and has, on a common event of probability one, a continuous version whose restriction to rational endpoints is the family in that theorem.  Then there is a deterministic $\theta$-H\"older path $a$ such that
\[
 A_{s,t}=a(t)-a(s)
\]
almost surely simultaneously for all $0\le s\le t\le T$.
\end{corollary}

\begin{proof}
On a common full-probability event, the given continuous version and the extension $\widetilde A$ from \cref{thm:additive-rigidity} agree on the dense set of rational pairs.  They therefore agree on the entire time simplex, and \eqref{eq:additive-rigidity-extension} gives the conclusion.
\end{proof}

\begin{remark}
The exponent $1/2$ enters through the summability of $2^{n(1-2\theta)}$, so the argument applies precisely for $\theta>1/2$.
\end{remark}

\section{Full-measure relative domains and local versions}\label{sec:full-measure-extension}

A full-measure domain may depend on information outside a given time interval.  Therefore one cannot preserve interval locality by assigning arbitrary values on its null complement.  The correct operation is to extend each relatively local coordinate through the trace sigma-field.

\begin{lemma}[Extension from a full-measure relative domain]\label{lem:trace-extension}
Let $(\Omega,\mathcal A,\mu)$ be a probability space, let $\mathcal G\subseteq\mathcal A$ be a sigma-field, and let $B\in\mathcal A$ satisfy $\mu(B)=1$.  Let $V$ be finite-dimensional.  If
\[
 F:B\longrightarrow V
\]
is measurable with respect to the trace sigma-field
\[
 \mathcal G\vert_B:=\{B\cap A:A\in\mathcal G\},
\]
then there exists a $\mathcal G$-measurable map
\[
 \widetilde F:\Omega\longrightarrow V
\]
such that $\widetilde F\vert_B=F$.
\end{lemma}

\begin{proof}
After choosing coordinates it is enough to treat $V=\R$.  Let $\psi:\R\to(0,1)$ be a Borel isomorphism and put $G=\psi\circ F$.  For $n\ge1$, let
\[
 G_n:=2^{-n}\lfloor 2^nG\rfloor.
\]
Each level set of the finite-valued trace-measurable function $G_n$ has the form $B\cap A_{n,k}$ with $A_{n,k}\in\mathcal G$.  Since the level sets are disjoint on $B$, replacing the $A_{n,k}$ successively by
\[
 A_{n,k}\setminus\bigcup_{j<k}A_{n,j}
\]
produces disjoint representatives without changing their intersections with $B$.  Assigning the value $0$ on the remaining complement defines a $\mathcal G$-measurable simple extension $\widetilde G_n$ of $G_n$.  Set $\widetilde G=\limsup_n\widetilde G_n$.  Then $\widetilde G$ is $\mathcal G$-measurable and equals $G$ on $B$.  Extend $\psi^{-1}$ arbitrarily from $(0,1)$ to $[0,1]$ and put $\widetilde F=\psi^{-1}(\widetilde G)$.  This gives the required extension.  Apply the scalar construction to each coordinate of $V$.
\end{proof}

\begin{lemma}[Borel core of a completed full-measure event]\label{lem:borel-core}
Let $\Omega$ be a Polish space, let $\mu$ be a Borel probability measure on $\Omega$, and let $E$ belong to the $\mu$-completion of the Borel sigma-field.  If $\mu(E)=1$, then there is a Borel set
\[
 E_0\subseteq E,
 \qquad
 \mu(E_0)=1.
\]
\end{lemma}

\begin{proof}
The complement $E^c$ is null in the completed sigma-field.  Hence there is a Borel set $N\supseteq E^c$ with $\mu(N)=0$.  Then $E_0:=\Omega\setminus N$ is Borel, has full measure, and is contained in $E$.
\end{proof}

\begin{corollary}[Countable rational local-coordinate extensions]\label{cor:rational-trace-extension}
Let $B\subseteq\Omega_d$ be Borel with $\mu_H^{(d)}(B)=1$.  Suppose that, for every rational $0\le s<t\le T$, a finite-dimensional coordinate $F_{s,t}:B\to V$ is measurable with respect to
\[
 \mathcal F_{[s,t]}\vert_B.
\]
Then the rational family admits extensions $\widetilde F_{s,t}$ for which each $\widetilde F_{s,t}$ is genuinely $\mathcal F_{[s,t]}$-measurable on $\Omega_d$.  Any countable family of algebraic identities and H\"older bounds that holds on $B$ holds almost surely for these extensions.  If the original family is continuous in $(s,t)$ on $B$, then every identity proved for all rational pairs holds, on a Borel subset $B_0\subseteq B$ of full measure, simultaneously for all real $0\le s\le t\le T$ after taking the original continuous versions.
\end{corollary}

\begin{proof}
Apply \cref{lem:trace-extension} componentwise to the countable set of rational intervals.  Because every extension agrees with the original coordinate on $B$, every countable identity holding on $B$ holds on the full probability space almost surely.  Let $E$ be the intersection of the resulting countably many full-probability events.  It belongs to the completed Borel sigma-field and has full measure.  Since $\Omega_d$ is Polish and $\mu_H^{(d)}$ is a Borel probability measure, \cref{lem:borel-core} gives a Borel set $E_0\subseteq E$ of full measure.  Put $B_0:=B\cap E_0$.  Then $B_0$ is Borel and has full measure.  On $B_0$, density of rational pairs and continuity of the original family extend the identities to the whole time simplex.
\end{proof}

\section{Automatic integrability and classification in the rough positive regime}

We use the full-law locality notion of Chevyrev--Ferrucci \cite{ChevyrevFerrucci2026}: a local $\eta$-rough path lift is a $\G^m(\R^d)$-valued $\eta$-H\"older random path with first level $B^H$, whose word coordinate over $[s,t]$ is measurable with respect to the corresponding coordinate-increment sigma-algebra on that interval.  We also allow a Borel lift on a full-measure Borel domain $B$ provided its rational-interval word coordinates are measurable for the corresponding trace local sigma-fields.  The domain formulation is handled through \cref{lem:trace-extension,cor:rational-trace-extension}.

For a two-parameter map $Z$, write
\[
  (\delta Z)_{s,u,t}:=Z_{s,t}-Z_{s,u}-Z_{u,t}.
\]

\begin{lemma}[Deterministic Wiener corrections]\label{lem:wiener-correction}
Let $H\in(0,1/2]$ and let $\varphi_{s,t}=a(t)-a(s)$ be a deterministic $\vartheta$-H\"older path with $\vartheta>1/2-H$. Componentwise define
\[
  J^{\rightarrow}_{s,t}
  :=\int_s^t\varphi_{s,r}\,\dd B_r^H,
  \qquad
  J^{\leftarrow}_{s,t}
  :=\int_s^t\dd B_r^H\,\varphi_{r,t},
\]
as first-chaos Wiener integrals. Then they are local and satisfy
\[
  \delta J^{\rightarrow}_{s,u,t}=\varphi_{s,u}B^H_{u,t},
  \qquad
  \delta J^{\leftarrow}_{s,u,t}=B^H_{s,u}\varphi_{u,t}.
\]
For every $p<\infty$,
\[
  \norm{J^{\rightarrow}_{s,t}}_{L^p}
  +\norm{J^{\leftarrow}_{s,t}}_{L^p}
  \le C_{p,H,\vartheta}\norm a_{C^\vartheta}\abs{t-s}^{\vartheta+H}.
\]
For every $\rho<\vartheta+H$ the two processes admit $\rho$-H\"older two-parameter versions whose H\"older seminorms have moments of every finite order.
\end{lemma}

\begin{proof}
For fixed $s<t$, the function $r\mapsto\varphi_{s,r}$ vanishes at $s$ and is $\vartheta$-H\"older. The embedding
\[
  C_0^\vartheta([s,t])\hookrightarrow\mathcal H_{s,t},
  \qquad
  \norm f_{\mathcal H_{s,t}}
  \le C_{H,\vartheta}\norm f_{C^\vartheta}\abs{t-s}^{\vartheta+H},
\]
is \cite[Lemma~3.6]{ChevyrevFerrucci2026}. The same bound holds for a function vanishing at the right endpoint. Indeed, the reflected fractional-Brownian increment field on $[s,t]$ has the same covariance, so $r\mapsto s+t-r$ induces an isometry of $\mathcal H_{s,t}$. The estimate therefore applies to $r\mapsto\varphi_{r,t}$. Hence both Wiener integrals are well defined in the local first Gaussian chaos; in particular they are measurable with respect to the interval increment sigma-algebra. Their $L^2$ estimates follow from the Wiener isometry and the displayed embedding, and all finite-$p$ estimates follow from first-chaos hypercontractivity.

Splitting either Wiener integral at $u$ and using
$\varphi_{s,r}=\varphi_{s,u}+\varphi_{u,r}$ for $r\ge u$ and
$\varphi_{r,t}=\varphi_{r,u}+\varphi_{u,t}$ for $r\le u$ proves the two Chen-defect identities. The defect terms satisfy the complementary moment bounds
\[
  \norm{\varphi_{s,u}B^H_{u,t}}_{L^p}
  +\norm{B^H_{s,u}\varphi_{u,t}}_{L^p}
  \lesssim_p |u-s|^{\vartheta}|t-u|^H.
\]
The standard dyadic Kolmogorov argument for two-parameter multiplicative functionals, applied to the increment and defect bounds, yields every exponent $\rho<\vartheta+H$ and all finite moments of the resulting H\"older seminorm. Starting with rational intervals and extending by limits through subintervals preserves interval-local measurability. This is exactly the estimate used in the proof of \cite[Theorem~3.7, equations~(40)--(41)]{ChevyrevFerrucci2026}.
\end{proof}

\begin{theorem}[Automatic square-integrability]\label{thm:auto-L2}
Let $H\in(1/4,1/2]$, let $m=\lfloor1/H\rfloor$, and let $0<\eta<H$ satisfy $1/\eta<m+1$.  Let $B\subseteq\Omega_d$ be Borel with $\mu_H^{(d)}(B)=1$, and let $\mathbf X$ be a Borel $\eta$-H\"older rough-path lift of the canonical first level on $B$.  Assume that every rational-interval word coordinate is measurable with respect to its trace local sigma-field.  Then, for every $1\le r\le m$ and every $0\le s\le t\le T$, the coordinate $\pi_r\mathbf X_{s,t}$ determines a unique $\mu_H^{(d)}$-almost-everywhere equivalence class and this class belongs to $L^2(\mu_H^{(d)})$.  The conclusion is coordinatewise; simultaneous group-like reassembly is established in \cref{lem:deterministic-correction-lie,cor:classification}.
\end{theorem}

\begin{proof}
Let $\overline{\mathbf X}$ be the canonical full-law local lift and restrict it to $B$.  At level two set, on $B$,
\[
 D^{(2)}_{s,t}=\mathbf X^{(2)}_{s,t}-\overline{\mathbf X}^{(2)}_{s,t}.
\]
For every rational interval this is trace-local.  Apply \cref{cor:rational-trace-extension} componentwise to obtain genuinely local full-law versions of its rational coordinates.  The two lifts have the same first level, so subtraction of their Chen identities gives rational additivity almost surely.  The random $2\eta$-H\"older bound on $B$ also holds almost surely for the extended rational family.  Since $H>1/4$, $m\in\{2,3\}$, and $1/\eta<m+1$ implies $2\eta>1/2$, \cref{thm:additive-rigidity} gives a deterministic $2\eta$-H\"older path $\varphi^{(2)}$ such that
\begin{equation}\label{eq:level2-difference}
 D^{(2)}_{s,t}=\varphi^{(2)}_t-\varphi^{(2)}_s
\end{equation}
for all rational pairs on a full-measure subset of $B$.  Both original lifts are continuous on $B$, so \cref{cor:rational-trace-extension} extends \eqref{eq:level2-difference} simultaneously to all real pairs on a Borel subset of $B$ of full measure.  The canonical second level is square-integrable, hence so is $\mathbf X^{(2)}$.

If $H>1/3$, the required roughness depth is two and the proof is complete.  Suppose $1/4<H\le1/3$.  On $B$ define
\[
 D^{(3)}_{s,t}=\mathbf X^{(3)}_{s,t}-\overline{\mathbf X}^{(3)}_{s,t}.
\]
The level-three Chen identity and \eqref{eq:level2-difference} give, first on rational triples and then by continuity on $B$,
\begin{equation}\label{eq:D3-defect}
 \delta D^{(3)}_{s,u,t}
 =\varphi^{(2)}_{s,u}\otimes B^H_{u,t}
  +B^H_{s,u}\otimes\varphi^{(2)}_{u,t}.
\end{equation}
Here $2\eta>1/2>1/2-H$, so \cref{lem:wiener-correction} applies with $\vartheta=2\eta$.  Define the full-law local first-chaos process
\begin{equation}\label{eq:J-definition}
 J_{s,t}
 =\int_s^t\varphi^{(2)}_{s,r}\otimes\dd B_r^H
  +\int_s^t\dd B_r^H\otimes\varphi^{(2)}_{r,t}.
\end{equation}
It has the Chen defect in \eqref{eq:D3-defect}.  Thus $R^{(3)}:=D^{(3)}-J\vert_B$ is trace-local and additive on rational intervals.  Apply \cref{cor:rational-trace-extension} to its rational coordinates.  The lemma also gives a version of $J$ which is $\rho$-H\"older for every $\rho<2\eta+H$.  Since $H>\eta$, choose
\[
 3\eta<\rho<2\eta+H.
\]
The process $D^{(3)}$ is $3\eta$-H\"older, and a $\rho$-H\"older process on a compact interval is also $3\eta$-H\"older.  Hence the rational extension of $R^{(3)}$ is local, additive, and $3\eta$-H\"older almost surely.  Since $3\eta>1/2$, \cref{thm:additive-rigidity} gives a deterministic additive path $\varphi^{(3)}$.  Returning to the original continuous coordinates on $B$ yields, simultaneously for all $s\le t$ on a full-measure Borel subset,
\[
 \mathbf X^{(3)}_{s,t}
 =\overline{\mathbf X}^{(3)}_{s,t}+J_{s,t}
  +\varphi^{(3)}_t-\varphi^{(3)}_s.
\]
The terms on the right are respectively square-integrable, first-chaos Gaussian, and deterministic.  This proves square-integrability through level three and completes the proof.
\end{proof}

\begin{lemma}[Lie closure of the deterministic correction]\label{lem:deterministic-correction-lie}
Under the hypotheses of \cref{thm:auto-L2}, use the deterministic paths and the decompositions obtained in its proof.  Put
\[
 \varphi^{(2)}_{s,t}:=\varphi^{(2)}_t-\varphi^{(2)}_s.
\]
Then
\[
 \varphi^{(2)}_{s,t}\in L_2(\R^d)
\]
for every $s\le t$.  If $1/4<H\le1/3$, also put
\[
 \varphi^{(3)}_{s,t}:=\varphi^{(3)}_t-\varphi^{(3)}_s;
\]
then $\varphi^{(3)}_{s,t}\in L_3(\R^d)$.  Define
\[
 \boldsymbol\varphi_{s,t}
 :=
 \begin{cases}
  \exp_{\le2}(\varphi^{(2)}_{s,t}),&m=2,\\[1mm]
  \exp_{\le3}(\varphi^{(2)}_{s,t}+\varphi^{(3)}_{s,t}),&m=3.
 \end{cases}
\]
Then $\boldsymbol\varphi$ is a deterministic multiplicative group-like rough path above the zero first-level path.  It has the graded bounds
\[
 \norm{\pi_2\boldsymbol\varphi_{s,t}}
 \lesssim |t-s|^{2\eta},
 \qquad
 \norm{\pi_3\boldsymbol\varphi_{s,t}}
 \lesssim |t-s|^{3\eta}
 \quad(m=3).
\]
\end{lemma}

\begin{proof}
Let $\Delta_{\mathrm{unsh}}$ be the unshuffle coproduct on the tensor algebra, determined by
\[
 \Delta_{\mathrm{unsh}}v=v\boxtimes1+1\boxtimes v,
 \qquad v\in\R^d,
\]
where $\boxtimes$ separates the two coproduct legs.  A homogeneous tensor is primitive for this coproduct if and only if it belongs to the corresponding homogeneous component of the free Lie algebra; see, for example, \cite[Chapter~1]{Reutenauer1993}.

Write $x_{s,t}=B^H_{s,t}$.  The degree-two group-like identities for $\mathbf X$ and $\overline{\mathbf X}$ are
\[
 \Delta_{\mathrm{unsh}}\mathbf X^{(2)}_{s,t}
 =\mathbf X^{(2)}_{s,t}\boxtimes1
  +1\boxtimes\mathbf X^{(2)}_{s,t}
  +x_{s,t}\boxtimes x_{s,t},
\]
and the same formula with $\mathbf X$ replaced by $\overline{\mathbf X}$.  Subtracting and using
$D^{(2)}_{s,t}=\varphi^{(2)}_{s,t}$ gives
\[
 \Delta_{\mathrm{unsh}}\varphi^{(2)}_{s,t}
 =\varphi^{(2)}_{s,t}\boxtimes1
  +1\boxtimes\varphi^{(2)}_{s,t}.
\]
Thus $\varphi^{(2)}_{s,t}\in L_2(\R^d)$; equivalently, its degree-two tensor is antisymmetric.  The identity holds first on a common full-measure subset of $B$; both sides are deterministic and continuous in $(s,t)$, so it holds for every pair.

Assume now $m=3$.  Subtracting the degree-three group-like identities gives
\begin{equation}\label{eq:D3-coproduct}
\begin{aligned}
 \Delta_{\mathrm{unsh}}D^{(3)}_{s,t}
 ={}&D^{(3)}_{s,t}\boxtimes1
   +1\boxtimes D^{(3)}_{s,t}\\
 &+x_{s,t}\boxtimes\varphi^{(2)}_{s,t}
   +\varphi^{(2)}_{s,t}\boxtimes x_{s,t}.
\end{aligned}
\end{equation}
We next compute the coproduct of the Wiener correction $J$ from \eqref{eq:J-definition}.  If $a\in L_2(\R^d)$ and $v\in\R^d$, multiplicativity of $\Delta_{\mathrm{unsh}}$ gives
\[
 \Delta_{\mathrm{unsh}}(a\otimes v)
 =(a\otimes v)\boxtimes1+1\boxtimes(a\otimes v)
  +a\boxtimes v+v\boxtimes a,
\]
and the same cross terms occur for $v\otimes a$.  The coproduct is finite-dimensional and therefore commutes with the first-chaos Wiener integrals.  Since
\[
 \varphi^{(2)}_{s,r}+\varphi^{(2)}_{r,t}
 =\varphi^{(2)}_{s,t},
\]
the two terms in \eqref{eq:J-definition} yield
\begin{equation}\label{eq:J-coproduct}
\begin{aligned}
 \Delta_{\mathrm{unsh}}J_{s,t}
 ={}&J_{s,t}\boxtimes1+1\boxtimes J_{s,t}\\
 &+x_{s,t}\boxtimes\varphi^{(2)}_{s,t}
   +\varphi^{(2)}_{s,t}\boxtimes x_{s,t}.
\end{aligned}
\end{equation}
Subtract \eqref{eq:J-coproduct} from \eqref{eq:D3-coproduct} and use
$D^{(3)}_{s,t}-J_{s,t}=\varphi^{(3)}_{s,t}$.  We obtain
\[
 \Delta_{\mathrm{unsh}}\varphi^{(3)}_{s,t}
 =\varphi^{(3)}_{s,t}\boxtimes1
  +1\boxtimes\varphi^{(3)}_{s,t},
\]
so $\varphi^{(3)}_{s,t}\in L_3(\R^d)$.  This is equivalent to all degree-three shuffle-annihilation identities; the mixed shuffle terms are exactly the two cross terms in \eqref{eq:J-coproduct}.

The increment $\varphi^{(2)}$ is additive, and $\varphi^{(3)}$ is additive when $m=3$.  In the step-three quotient every bracket between homogeneous degrees at least two has degree at least four and vanishes; the same statement is immediate at step two.  Hence the relevant truncated Baker--Campbell--Hausdorff product of correction increments is ordinary addition, and
\[
 \boldsymbol\varphi_{s,u}\boldsymbol\varphi_{u,t}
 =\boldsymbol\varphi_{s,t}.
\]
Primitivity gives group-likeness, and the H\"older bounds are those already obtained in the proof of \cref{thm:auto-L2}.  Thus $\boldsymbol\varphi$ is a deterministic rough path above zero.
\end{proof}

\begin{corollary}[Classification without a moment assumption]\label{cor:classification}
Under the assumptions of \cref{thm:auto-L2}, there is a Borel set $B_0\subseteq B$ of full measure on which the lift agrees, simultaneously at all times, with a full-law local lift satisfying the Chevyrev--Ferrucci classification.  In particular, on $B_0$,
\[
  \mathbf X^{(2)}_{s,t}
  =\overline{\mathbf X}^{(2)}_{s,t}+\varphi^{(2)}_{s,t}
\]
for all $s\le t$, with deterministic additive $\varphi^{(2)}$. If $1/4<H\le1/3$, then, simultaneously for all $s\le t$, componentwise
\[
\begin{aligned}
  \mathbf X^{\alpha\beta\gamma}_{s,t}
  ={}&\overline{\mathbf X}^{\alpha\beta\gamma}_{s,t}
  +\int_s^t\varphi^{\alpha\beta}_{s,u}\,\dd B_u^{H,\gamma}\\
  &+\int_s^t\varphi^{\beta\gamma}_{u,t}\,\dd B_u^{H,\alpha}
  +\varphi^{\alpha\beta\gamma}_{s,t},
\end{aligned}
\]
where $(0,\varphi^{(2)},\varphi^{(3)})$ is a deterministic rough path above zero. Conversely every such deterministic rough translation gives a full-law local lift.
\end{corollary}

\begin{proof}
The proof of \cref{thm:auto-L2} gives the level-two and, when $m=3$, level-three decompositions on a common completed full-measure event.  By \cref{lem:deterministic-correction-lie}, the deterministic correction is a rough path above zero.  The deterministic rough-translation construction, equivalently the converse construction in \cite[Theorem~3.7]{ChevyrevFerrucci2026}, therefore defines a full-law local lift $\widetilde{\mathbf X}$ with precisely those formulas.  It agrees with $\mathbf X$ on the rational time simplex on a completed full-measure event and hence, by continuity, on the full time simplex there.  Apply \cref{lem:borel-core} and intersect with $B$ to obtain a Borel set $B_0\subseteq B$ of full measure on which the equality holds simultaneously at all times.

The full-law lift $\widetilde{\mathbf X}$ is square-integrable by the canonical, first-chaos, and deterministic decomposition proved above.  The Chevyrev--Ferrucci classification theorem now applies.  Conversely, the same theorem constructs a full-law local lift from every deterministic rough path above zero of the displayed regularity.
\end{proof}

\begin{corollary}[Symmetry classification without $L^2$]\label{cor:symmetry}
Let $1/4<H\le1/2$, put $m=\lfloor1/H\rfloor$, and let $0<\eta<H$ satisfy $1/\eta<m+1$.  For every full-law local $\eta$-H\"older rough-path lift, the Chevyrev--Ferrucci stationarity, scaling, and coordinate-permutation classification holds without an assumed moment condition.  In particular, the canonical lift is the only full-law local lift having all three symmetries, except at $H=1/3$, where the one-parameter family in the cited corollary remains.
\end{corollary}

\begin{proof}
Apply \cref{thm:auto-L2} with $B=\Omega_d$ and then the Chevyrev--Ferrucci symmetry corollary \cite[Corollary 3.11]{ChevyrevFerrucci2026}.
\end{proof}

\begin{proposition}[Finite-step Young extension and geometricity]\label{prop:young-finite-step}
Let $X:[0,T]\to E$ be $\eta$-H\"older with $\eta>1/2$, where $E$ is finite-dimensional, and let $M<\infty$.  The truncated Young signature is the unique multiplicative extension
\[
 \mathbf X^{(M)}=(1,\pi_1\mathbf X^{(M)},\ldots,\pi_M\mathbf X^{(M)})
\]
with first level $X_{s,t}$ and graded bounds
\begin{equation}\label{eq:young-graded-bounds}
 \norm{\pi_r\mathbf X^{(M)}_{s,t}}
 \le C_r|t-s|^{r\eta},
 \qquad 1\le r\le M.
\end{equation}
It is group-like, hence weakly geometric, at exponent $\eta$.  For every $\eta'<\eta$, it is strong geometric in the $\eta'$-H\"older rough-path topology.  If
\[
 X\in c^\eta([0,T];E)
 :=\overline{C^\infty([0,T];E)}^{\,C^\eta},
\]
then it is strong geometric at exponent $\eta$ itself.
\end{proposition}

\begin{proof}
Existence, group-likeness, and \eqref{eq:young-graded-bounds} follow from iterated Young integration and the shuffle identities.  For uniqueness, let $\mathbf X$ and $\widetilde{\mathbf X}$ be two multiplicative extensions satisfying these bounds and having the same first level.  Suppose by induction that their levels below $r\ge2$ agree.  Subtracting the level-$r$ Chen identities shows that
\[
 D^{(r)}_{s,t}
 :=\pi_r(\mathbf X_{s,t}-\widetilde{\mathbf X}_{s,t})
\]
is additive.  For the equal partition $s=t_0<\cdots<t_n=t$,
\[
 \norm{D^{(r)}_{s,t}}
 \le
 \sum_{j=0}^{n-1}\norm{D^{(r)}_{t_j,t_{j+1}}}
 \le
 C|t-s|^{r\eta}n^{1-r\eta}.
\]
Since $r\eta>1$, the right-hand side tends to zero.  Thus $D^{(r)}=0$, and induction through $r=M$ proves uniqueness.

Choose smooth paths $X^n\to X$ in $C^{\eta'}$ for every $\eta'<\eta$; such approximations exist by first approximating uniformly at a regularity exponent strictly between $\eta'$ and $\eta$.  Continuity of iterated Young integration sends their signatures to $\mathbf X^{(M)}$ in the $\eta'$-H\"older rough-path topology, proving lower-exponent strong geometricity.  If $X\in c^\eta$, the approximating sequence may be chosen in $C^\eta$, and the same continuity gives strong geometricity at the exact exponent.  Thus exact-$\eta$ strong geometricity is obtained precisely under the little-H\"older condition $X\in c^\eta$.
\end{proof}

\begin{proof}[Proof of \cref{thm:phase-transition}(ii)--(iii)]
If $1/4<H\le1/2$, apply \cref{thm:auto-L2} and then \cref{cor:classification}.  This proves item~(ii).

Assume $H>1/2$ and $1/2<\eta<H$.  Choose $\alpha$ with $\eta<\alpha<H$.  Almost every fractional-Brownian sample path belongs to $C^\alpha$, hence to the little space $c^\eta$ by smooth approximation in the lower $C^\eta$ norm.  Apply \cref{prop:young-finite-step} samplewise at every prescribed finite tensor step $M$.  The resulting extension is the truncated Young signature, is unique in the class satisfying the graded bounds, and is strong geometric in the $\eta$-H\"older topology.  Young estimates bound its level-$r$ coordinates by a constant multiple of $\norm{B^H}_{C^\eta}^r$, and the fractional-Brownian H\"older norm has moments of every finite order.  This proves item~(iii).
\end{proof}

\section{The Heisenberg equation}

Let
\[
  \mathbb H_3
  =\left\{
  \begin{pmatrix}
  1&a&c\\0&1&b\\0&0&1
  \end{pmatrix}:a,b,c\in\R
  \right\}.
\]
Its multiplication is
\[
  (a,b,c)(a',b',c')=(a+a',b+b',c+c'+ab').
\]

\begin{definition}[Local Heisenberg solution flow]\label{def:heisenberg-flow}
Let $B\subset\Omega_d$ be Borel. A \emph{local Heisenberg solution flow} on $B$ is a family of finite Borel maps
\[
  Y_{s,t}:B\to\mathbb H_3,
  \qquad
  Y_{s,t}(x)=
  \begin{pmatrix}
  1&x^1_{s,t}&C_{s,t}(x)\\
  0&1&x^2_{s,t}\\
  0&0&1
  \end{pmatrix},
\]
such that
\[
  Y_{s,t}=Y_{s,u}Y_{u,t}
\]
for all $s\le u\le t$, and $C_{s,t}$ is measurable with respect to $\sigma(r_{s,t}^{\{1,2\}}|_B)$.
It is $\eta$-H\"older if $\abs{C_{s,t}}\le R(x)\abs{t-s}^{2\eta}$ for an almost surely finite random $R$.
\end{definition}

\begin{lemma}[The central coordinate is a Chen area]\label{lem:heisenberg-chen}
A family $Y$ is multiplicative if and only if
\[
  C_{s,t}=C_{s,u}+C_{u,t}+x^1_{s,u}x^2_{u,t}.
\]
For a smooth path $x$, the unique classical solution of \eqref{eq:heisenberg-rde} has
\[
  C_{s,t}=\int_s^t x^1_{s,u}\,\dd x_u^2.
\]
\end{lemma}

\begin{proof}
The first assertion is the $(1,3)$ entry of matrix multiplication. For the second, write
\[
  Y_t=\begin{pmatrix}1&a_t&c_t\\0&1&b_t\\0&0&1\end{pmatrix}.
\]
The right-driven equation gives $\dd a=\dd x^1$, $\dd b=\dd x^2$, and $\dd c=a\,\dd x^2$, with zero initial increments.
\end{proof}

\begin{proof}[Proof of \cref{thm:heisenberg-main}]
If $H\le1/4$, \cref{lem:heisenberg-chen} turns the central coordinate of any positive-domain local flow into a coordinate forbidden by \cref{thm:phase-transition}(i).

Let $1/4<H\le1/2$.  Project the canonical local rough path through the algebra homomorphism sending the first two generators to $E_{12},E_{23}$.  This gives a Borel local multiplicative $\mathbb H_3$-valued flow solving \eqref{eq:heisenberg-rde} on a full-measure Borel domain.  Let $C$ be any local $\eta$-H\"older central coordinate on a full-measure Borel domain $B$, and let $\overline C$ be the canonical one, restricted to the intersection of their full-measure domains.  For each rational $s<t$, the difference
\[
 D_{s,t}:=C_{s,t}-\overline C_{s,t}
\]
is measurable with respect to $\mathcal F_{[s,t]}\vert_B$.  Apply \cref{cor:rational-trace-extension} to obtain genuinely $\mathcal F_{[s,t]}$-measurable extensions of the rational coordinates.  By \cref{lem:heisenberg-chen}, rational additivity holds on $B$ and hence almost surely for the extensions; the $2\eta$-H\"older bound likewise holds almost surely.  Since $2\eta>1/2$, \cref{thm:additive-rigidity} gives a deterministic $2\eta$-H\"older path $a$ and the identity
\[
 D_{s,t}=a(t)-a(s)
\]
for all rational pairs on a common full-probability event $E$ in the completed Borel sigma-field.  By \cref{lem:borel-core}, choose a Borel full-measure core $E_0\subseteq E$ and put $B_0:=B\cap E_0$.  The original candidate flow and the canonical flow are continuous in $(s,t)$ on $B_0$, so density of rational pairs yields the stated identity simultaneously for every $x\in B_0$ and every real $0\le s\le t\le T$.  This proves item~(ii).

Let $H>1/2$.  Young integration gives the displayed central coordinate and hence the local multiplicative flow.  For any other local $\eta$-H\"older flow, the difference $D=C-\overline C$ is additive and $2\eta$-H\"older.  For an equal partition $s=t_0<\cdots<t_n=t$,
\begin{align*}
 \abs{D_{s,t}}
 &\leq\sum_{j=0}^{n-1}\abs{D_{t_j,t_{j+1}}}\\
 &\leq R\,n\left(\frac{t-s}{n}\right)^{2\eta}
 =R\abs{t-s}^{2\eta}n^{1-2\eta}.
\end{align*}
Since $2\eta>1$, the right-hand side tends to zero.  Thus $D_{s,t}=0$, so no nontrivial deterministic central drift is possible and the Young flow is unique.  This proves item~(iii).
\end{proof}

\begin{corollary}[No universal local pathwise solver below one quarter]\label{cor:no-solver}
Let $H\le1/4$ and $d\ge2$. There is no Borel local pathwise solver, defined on a positive-probability Borel set, which assigns multiplicative solution flows to all constant-coefficient controlled equations and agrees with the first-level driver. It already fails for the single equation \eqref{eq:heisenberg-rde}.
\end{corollary}

\section{Conclusion}\label{sec:conclusion}

The paper separates two questions that are often conflated below the Young threshold.  The deterministic question is whether a rough enhancement can be constructed with quantitative stability.  The scale-coordinate analysis answers this by the chain
\[
 \begin{gathered}
 \text{matched increments}
 \Longrightarrow
 \text{exact reconstruction}\\
 \Longrightarrow
 \text{dyadic H\"older path}
 \Longrightarrow
 \text{Fourier-normal-ordered lift}.
 \end{gathered}
\]
The stochastic question is whether an increment of that lift can be recovered from the driving increments on the same interval.  For fractional Brownian motion the answer changes exactly at one quarter:
\[
 \begin{array}{c|c}
 H>\frac12 & \text{Young-local and canonical},\\[1mm]
 \frac14<H\le\frac12 & \text{full-law local rough lifts exist and are classified},\\[1mm]
 0<H\le\frac14 & \text{every positive-domain measurable local lift is impossible.}
 \end{array}
\]

The negative proof does not depend on the chosen constructive scheme.  Positive Gaussian mass produces four-corner configurations; Borel finiteness and anti-concentration force a uniform Hilbert--Schmidt bound; locality and Chen identify the corresponding rectangle with the Young form; and oscillatory Cameron--Martin packets force the divergent series $\sum k^{-4H}$.  Hence below one quarter the use of global Fourier, Volterra, renormalization, or extension data is not merely a feature of known constructions.  Some nonlocal information is necessary.

The Heisenberg equation shows that the obstruction is already visible in a single linear noncommutative differential equation.  At and below one quarter no positive-probability Borel domain supports a finite local multiplicative solution flow.  In the rough regime $1/4<H\le1/2$, every full-law local H\"older flow differs from the canonical one only by a deterministic central drift.  In the Young regime $H>1/2$, even that drift vanishes and the local flow is unique.

The scale--FNO construction supplies an explicit measurable lift with quantitative ultraviolet control, while the locality theorem is scheme-independent.  Together they show that rough enhancement exists at every positive regularity, but below one quarter it necessarily uses information beyond the increments of the path on the interval itself.

\appendix

\section{Free synthesis beyond the realized range}\label{app:free-synthesis}

The exact reconstruction theorem uses only towers in $\cA_{\omega,\eps}=\Ran\cD_{\omega,\eps}$; on that range, the series is the original Littlewood--Paley reconstruction and no summability hypothesis is required.  Synthesizing an arbitrary bounded annular tower is a different problem.  This appendix states the two additional conditions used for ambient synthesis.

Fix $0\leq\eps\leq\eps_0$.  Let $Y_\infty$ be the space of towers $(u_{-1},z)$ such that
\[
 u_{-1}\in C_b(\R),
 \qquad
 \sup_{j\geq0}\norm{z_j}_\infty<\infty,
\]
and each $z_j$ has Fourier support in the enlarged $j$th annulus selected by $\widetilde\varphi$.

Absolute convergence in $C_b$ requires the Dini condition stated next; membership in $A_\omega$ requires the stronger dyadic Hardy control of \cref{prop:hardy-synthesis}.

\begin{proposition}[Dini synthesis in $C_b$]\label{prop:dini-synthesis}
Assume
\begin{equation}\label{eq:dini-condition}
 \sum_{j=0}^\infty\lambda_j<\infty.
\end{equation}
For every $(u_{-1},z)\in Y_\infty$, the series
\begin{equation}\label{eq:free-synthesis-series}
 \cI_{\omega,\eps}^{\mathrm{free}}(u_{-1},z)
 :=
 u_{-1}+\sum_{j=0}^\infty I_{j,\eps}^\omega z_j
\end{equation}
converges absolutely and uniformly, and
\begin{equation}\label{eq:dini-synthesis-bound}
 \norm{\cI_{\omega,\eps}^{\mathrm{free}}(u_{-1},z)}_\infty
 \leq
 \norm{u_{-1}}_\infty
 +
 C\left(\sum_{j=0}^\infty\lambda_j\right)
 \sup_j\norm{z_j}_\infty.
\end{equation}
Moreover,
\begin{equation}\label{eq:dini-epsilon-stability}
 \norm{\cI_{\omega,\eps}^{\mathrm{free}}(u_{-1},z)
       -\cI_{\omega,0}^{\mathrm{free}}(u_{-1},z)}_\infty
 \leq
 C\eps^2
 \left(\sum_{j=0}^\infty\lambda_j\right)
 \sup_j\norm{z_j}_\infty.
\end{equation}
\end{proposition}

\begin{proof}
By \eqref{eq:inverse-multiplier-bound},
\[
 \sum_{j=0}^\infty\norm{I_{j,\eps}^\omega z_j}_\infty
 \leq
 C\left(\sup_j\norm{z_j}_\infty\right)
 \sum_{j=0}^\infty\lambda_j.
\]
The right-hand side is finite by \eqref{eq:dini-condition}; this proves absolute uniform convergence and \eqref{eq:dini-synthesis-bound}.  Summing \eqref{eq:inverse-collision-stability} over $j$ gives \eqref{eq:dini-epsilon-stability}.
\end{proof}

Membership in $A_\omega$ requires simultaneous control of the low-frequency increment sum and the high-frequency tail.  The following Hardy condition is exactly the estimate used in that split.

\begin{definition}[Dyadic Hardy regularity]\label{def:hardy-regularity}
The modulus sequence $(\lambda_j)$ satisfies the dyadic Hardy condition if there is $C_H>0$ such that
\begin{equation}\label{eq:hardy-condition}
 \sum_{j>J}\lambda_j
 +
 r_J\sum_{j=0}^Jr_j^{-1}\lambda_j
 \leq
 C_H\lambda_J
 \qquad(J\geq0).
\end{equation}
The first term controls the high-frequency tail; the second controls increments of the low-frequency sum.
\end{definition}

\begin{proposition}[Free synthesis in the modulus algebra]\label{prop:hardy-synthesis}
Assume \eqref{eq:hardy-condition}.  Let $u_{-1}\in A_\omega$ and let $(z_j)_{j\geq0}$ be a bounded annular tower.  Then the series
\[
 f:=u_{-1}+\sum_{j=0}^\infty I_{j,\eps}^\omega z_j
\]
converges uniformly and belongs to $A_\omega$.  Uniformly for $0\leq\eps\leq\eps_0$,
\begin{equation}\label{eq:hardy-synthesis-bound}
 \norm f_{C^\omega}
 \leq
 C\left(
 \norm{u_{-1}}_{C^\omega}
 +
 \sup_j\norm{z_j}_\infty
 \right).
\end{equation}
For every H\"older modulus $\omega(r)=r^\alpha$ with $0<\alpha<1$, condition \eqref{eq:hardy-condition} holds.
\end{proposition}

\begin{proof}
Set
\[
 g_j:=I_{j,\eps}^\omega z_j.
\]
The inverse-multiplier bound and frequency localization give
\begin{equation}\label{eq:synthesis-block-bounds}
 \norm{g_j}_\infty
 \leq
 C\lambda_j\norm{z_j}_\infty,
 \qquad
 \norm{g_j'}_\infty
 \leq
 Cr_j^{-1}\lambda_j\norm{z_j}_\infty.
\end{equation}
The first term in \eqref{eq:hardy-condition} with $J=0$ implies $\sum_j\lambda_j<\infty$, so the series converges uniformly.

Let $0<h\leq1$ and choose $J$ such that
\[
 r_{J+1}<h\leq r_J.
\]
For the low-frequency part, the mean-value theorem and \eqref{eq:synthesis-block-bounds} give
\begin{align*}
 \sum_{j\leq J}\norm{g_j(\cdot+h)-g_j}_\infty
 &\leq
 h\sum_{j\leq J}\norm{g_j'}_\infty\\
 &\leq
 Ch\left(\sum_{j\leq J}r_j^{-1}\lambda_j\right)
 \sup_j\norm{z_j}_\infty.
\end{align*}
For the high-frequency part,
\begin{align*}
 \sum_{j>J}\norm{g_j(\cdot+h)-g_j}_\infty
 &\leq
 2\sum_{j>J}\norm{g_j}_\infty\\
 &\leq
 C\left(\sum_{j>J}\lambda_j\right)
 \sup_j\norm{z_j}_\infty.
\end{align*}
Since $h\leq r_J$, condition \eqref{eq:hardy-condition} yields
\[
 \sum_{j\geq0}\norm{g_j(\cdot+h)-g_j}_\infty
 \leq
 C\lambda_J\sup_j\norm{z_j}_\infty.
\]
The choice $r_{J+1}<h$ gives $r_J<2h$.  Applying \cref{lem:concave-scale} with $r=h$ and $\abs z=r_J/h<2$,
\[
 \lambda_J=\omega(r_J)\leq3\omega(h).
\]
Hence
\[
 \left[\sum_{j\geq0}g_j\right]_\omega
 \leq
 C\sup_j\norm{z_j}_\infty.
\]
Adding the $A_\omega$ norm of $u_{-1}$ and the uniform estimate proves \eqref{eq:hardy-synthesis-bound}.

If $\lambda_j=2^{-\alpha j}$, both sums in \eqref{eq:hardy-condition} are geometric:
\[
 \sum_{j>J}2^{-\alpha j}
 \leq
 C_\alpha2^{-\alpha J},
\]
\[
 2^{-J}\sum_{j=0}^J2^{(1-\alpha)j}
 \leq
 C_\alpha2^{-\alpha J}.
\]
Thus the H\"older modulus satisfies the condition.
\end{proof}

The Dini and Hardy hypotheses are used only for extending synthesis from analyzed towers to arbitrary bounded annular towers.
 \section{Changes of frequency coordinates}\label{app:coordinate-change}

A Littlewood--Paley pair is a coordinate choice.  Exact reconstruction compares two such choices on their realized tower ranges, while finite annular overlap identifies the corresponding little scale spaces.  We state these two conclusions separately because the first is algebraic and the second is an analytic high-frequency estimate.

Let $\Delta=(\Delta_j)_{j\geq-1}$ and $\Gamma=(\Gamma_j)_{j\geq-1}$ be two smooth inhomogeneous Littlewood--Paley pairs of the form in \cref{def:lp-pair}.  Superscripts indicate the pair used in the analysis and reconstruction maps.

\begin{theorem}[Strict transition on realized tower ranges]\label{thm:coordinate-transition}
Choose $\eps_0$ below the inversion thresholds of both pairs.  For $0\leq\eps\leq\eps_0$, define
\[
 T_{\Gamma\leftarrow\Delta,\eps}
 :=
 \cD_{\omega,\eps}^{\Gamma}
 \cI_{\omega,\eps}^{\Delta}:
 \cA_{\omega,\eps}^{\Delta}
 \longrightarrow
 \cA_{\omega,\eps}^{\Gamma}.
\]
Then $T_{\Gamma\leftarrow\Delta,\eps}$ is a bijection with inverse $T_{\Delta\leftarrow\Gamma,\eps}$.  For a third pair $\Lambda$,
\begin{equation}\label{eq:coordinate-cocycle}
 T_{\Lambda\leftarrow\Gamma,\eps}
 T_{\Gamma\leftarrow\Delta,\eps}
 =
 T_{\Lambda\leftarrow\Delta,\eps}.
\end{equation}
The transition is isometric for the transported norms and intertwines the transported products:
\begin{equation}\label{eq:coordinate-product-intertwining}
 T_{\Gamma\leftarrow\Delta,\eps}
 (a\star_{\omega,\eps}^{\Delta}b)
 =
 T_{\Gamma\leftarrow\Delta,\eps}(a)
 \star_{\omega,\eps}^{\Gamma}
 T_{\Gamma\leftarrow\Delta,\eps}(b).
\end{equation}
\end{theorem}

\begin{proof}
The analysis--reconstruction identity for the $\Gamma$ pair gives
\[
 \cI_{\omega,\eps}^{\Gamma}
 \cD_{\omega,\eps}^{\Gamma}
 =
 \id_{A_\omega}.
\]
Therefore
\begin{align*}
 T_{\Delta\leftarrow\Gamma,\eps}
 T_{\Gamma\leftarrow\Delta,\eps}
 &=
 \cD_{\omega,\eps}^{\Delta}
 \cI_{\omega,\eps}^{\Gamma}
 \cD_{\omega,\eps}^{\Gamma}
 \cI_{\omega,\eps}^{\Delta}\\
 &=
 \cD_{\omega,\eps}^{\Delta}
 \cI_{\omega,\eps}^{\Delta}
 =
 \id_{\cA_{\omega,\eps}^{\Delta}}.
\end{align*}
The reverse identity and \eqref{eq:coordinate-cocycle} follow by the same cancellation.  Since the transported norm of a tower is the $A_\omega$ norm of its reconstructed function, the transition is isometric.

For \eqref{eq:coordinate-product-intertwining}, both sides equal
\[
 \cD_{\omega,\eps}^{\Gamma}
 \left(
 \cI_{\omega,\eps}^{\Delta}(a)
 \cI_{\omega,\eps}^{\Delta}(b)
 \right).
\]
\end{proof}

The individual representatives $(d_j^\omega f)_j$ depend on the frequency cutoffs.  Their vanishing criterion at critical scale does not.

\begin{proposition}[Independence of the little scale space]\label{prop:little-space-independence}
Define
\[
 c_{\Delta}^\omega
 :=
 \left\{f\in A_\omega:
 \lambda_j^{-1}\norm{\Delta_jf}_\infty\to0
 \right\},
\]
\[
 c_{\Gamma}^\omega
 :=
 \left\{f\in A_\omega:
 \lambda_j^{-1}\norm{\Gamma_jf}_\infty\to0
 \right\}.
\]
Then
\begin{equation}\label{eq:little-space-coordinate-independence}
 c_{\Delta}^\omega=c_{\Gamma}^\omega.
\end{equation}
Moreover,
\begin{equation}\label{eq:coordinate-limsup-equivalence}
 \limsup_{k\to\infty}
 \lambda_k^{-1}\norm{\Gamma_kf}_\infty
 \asymp
 \limsup_{j\to\infty}
 \lambda_j^{-1}\norm{\Delta_jf}_\infty.
\end{equation}
Consequently, the rule
\[
 \mathsf T_{\Gamma\leftarrow\Delta}
 (\delta_\omega^{\Delta}f)
 :=
 \delta_\omega^{\Gamma}f
\]
defines a bounded linear bijection between the realized images of the two corona derivatives.
\end{proposition}

\begin{proof}
Finite annular overlap gives an integer $M\geq1$ such that, for all sufficiently large $k$,
\begin{equation}\label{eq:finite-annular-overlap}
 \Gamma_kf
 =
 \sum_{\abs{j-k}\leq M}\Gamma_k\Delta_jf.
\end{equation}
The multipliers $\Gamma_k\Delta_j$ are uniformly bounded on $C_b(\R)$.  Concavity of $\omega$ gives a constant $C_M$ such that
\begin{equation}\label{eq:finite-scale-comparability}
 C_M^{-1}\lambda_k
 \leq
 \lambda_j
 \leq
 C_M\lambda_k
 \qquad(\abs{j-k}\leq M).
\end{equation}
Using \eqref{eq:finite-annular-overlap} and \eqref{eq:finite-scale-comparability},
\[
 \lambda_k^{-1}\norm{\Gamma_kf}_\infty
 \leq
 C_M
 \sum_{\abs{j-k}\leq M}
 \lambda_j^{-1}\norm{\Delta_jf}_\infty.
\]
Taking the limsup gives
\[
 \limsup_{k\to\infty}
 \lambda_k^{-1}\norm{\Gamma_kf}_\infty
 \leq
 C
 \limsup_{j\to\infty}
 \lambda_j^{-1}\norm{\Delta_jf}_\infty.
\]
Interchanging $\Delta$ and $\Gamma$ gives the reverse inequality and proves \eqref{eq:coordinate-limsup-equivalence}.  Equation \eqref{eq:little-space-coordinate-independence} follows by taking limits.

By \cref{thm:corona-kernel}, the common little space is the kernel of both realized corona derivatives.  The quotient norm is equivalent to the corresponding limsup, so the displayed rule is well defined, bounded, bijective, and has bounded inverse.
\end{proof}

Thus the realized corona classes carry the same critical information, while their bounded representatives depend on the chosen frequency coordinates.
 \section{Low-degree Magnus and BCH formulas}\label{app:low-degree}

The general reconstruction theorem is homogeneous and valid at every finite step.  The formulas through degree four are listed here to make the first nonlinear endpoint corrections explicit.  If the critical depth is below four, apply \cref{prop:smooth-controlled-extension} with $M=4$ before taking the following components.

Fix one smooth based step-four path and suppress the cutoff index.  Write
\[
 A=A_1-\Gamma_2-\Gamma_3-\Gamma_4+O(5),
 \qquad
 \Omega=\Omega_1+\Omega_2+\Omega_3+\Omega_4+O(5).
\]

The homogeneous Magnus recursion \eqref{eq:homogeneous-magnus} gives the local prepared densities.

\begin{proposition}[Prepared densities through degree four]\label{prop:low-degree-magnus}
One has
\begin{equation}\label{eq:dprep2}
 D_2^{\mathrm{prep}}
 =
 \frac12[\Omega_1,A_1],
\end{equation}
\begin{equation}\label{eq:dprep3}
 D_3^{\mathrm{prep}}
 =
 -\frac12[\Omega_1,\Gamma_2]
 +\frac12[\Omega_2,A_1]
 +\frac1{12}[\Omega_1,[\Omega_1,A_1]],
\end{equation}
and
\begin{align}
 D_4^{\mathrm{prep}}
 ={}&
 -\frac12[\Omega_1,\Gamma_3]
 -\frac12[\Omega_2,\Gamma_2]
 +\frac12[\Omega_3,A_1]
 \nonumber\\
 &-
 \frac1{12}[\Omega_1,[\Omega_1,\Gamma_2]]
 +\frac1{12}[\Omega_1,[\Omega_2,A_1]]
 +\frac1{12}[\Omega_2,[\Omega_1,A_1]].
 \label{eq:dprep4}
\end{align}
Consequently,
\[
 \dot\Omega_r=D_r^{\mathrm{prep}}-\Gamma_r
 \qquad(r=2,3,4).
\]
\end{proposition}

\begin{proof}
Up to degree four, \eqref{eq:full-magnus} is
\[
 \dot\Omega
 =
 A+\frac12[\Omega,A]
 +\frac1{12}[\Omega,[\Omega,A]]+O(5),
\]
because the cubic Bernoulli coefficient is zero.  In degree two, the only nonlinear split is $(1,1)$, giving \eqref{eq:dprep2}.  In degree three, the first commutator has splits $(1,2)$ and $(2,1)$, while the double commutator has split $(1,1,1)$.  Substitution of $A_2=-\Gamma_2$ gives \eqref{eq:dprep3}.

In degree four, the first commutator has splits
\[
 (1,3),
 \qquad
 (2,2),
 \qquad
 (3,1),
\]
and the double commutator has splits
\[
 (1,1,2),
 \qquad
 (1,2,1),
 \qquad
 (2,1,1).
\]
Substituting $A_2=-\Gamma_2$ and $A_3=-\Gamma_3$ gives \eqref{eq:dprep4}.  The degree-$r$ linear term is $A_r=-\Gamma_r$, proving the final identity.
\end{proof}

For endpoint formulas, write
\[
 S_r:=\Omega_r(s),
 \qquad
 T_r:=\Omega_r(t),
 \qquad
 \Delta\Omega_r:=T_r-S_r.
\]

\begin{proposition}[Based logarithmic coefficients through degree four]\label{prop:low-degree-bch}
The first four homogeneous components of
\[
 \BCH(-\Omega(s),\Omega(t))
\]
are
\begin{equation}\label{eq:I1}
 I_1(s,t)=\Delta\Omega_1,
\end{equation}
\begin{equation}\label{eq:I2}
 I_2(s,t)
 =
 \Delta\Omega_2-\frac12[S_1,T_1],
\end{equation}
\begin{align}
 I_3(s,t)
 ={}&
 \Delta\Omega_3
 -\frac12[S_1,T_2]
 -\frac12[S_2,T_1]
 \nonumber\\
 &+
 \frac1{12}[S_1,[S_1,T_1]]
 +\frac1{12}[T_1,[S_1,T_1]],
 \label{eq:I3}
\end{align}
and
\begin{align}
 I_4(s,t)
 ={}&
 \Delta\Omega_4
 -\frac12\bigl([S_1,T_3]+[S_2,T_2]+[S_3,T_1]\bigr)
 \nonumber\\
 &+
 \frac1{12}\bigl(
 [S_1,[S_1,T_2]]
 +[S_1,[S_2,T_1]]
 +[S_2,[S_1,T_1]]
 \bigr)
 \nonumber\\
 &+
 \frac1{12}\bigl(
 [T_1,[S_2,T_1]]
 +[T_1,[S_1,T_2]]
 +[T_2,[S_1,T_1]]
 \bigr)
 \nonumber\\
 &-
 \frac1{24}[T_1,[S_1,[S_1,T_1]]].
 \label{eq:I4}
\end{align}
In each degree $r$, the terms following $\Delta\Omega_r$ are the boundary package $B_r(s,t)$ from \eqref{eq:bch-package}.
\end{proposition}

\begin{proof}
Use the BCH expansion through degree four:
\begin{align*}
 \BCH(X,Y)
 ={}&X+Y+\frac12[X,Y]
 +\frac1{12}[X,[X,Y]]
 +\frac1{12}[Y,[Y,X]]\\
 &-\frac1{24}[Y,[X,[X,Y]]]+O(5).
\end{align*}
Substitute
\[
 X=-\sum_{r=1}^4S_r,
 \qquad
 Y=\sum_{r=1}^4T_r,
\]
and collect homogeneous degrees.  Degree one gives \eqref{eq:I1}; the degree-two, degree-three, and degree-four collections give \eqref{eq:I2}, \eqref{eq:I3}, and \eqref{eq:I4}, respectively.
\end{proof}

Combining \cref{prop:low-degree-magnus,prop:low-degree-bch} with \cref{thm:bch-reconstruction} yields, for $r=2,3,4$,
\[
 P_r(t)-P_r(s)
 =
 \int_s^t(D_r^{\mathrm{prep}}-\Gamma_r)(u)\,\dd u,
\]
\[
 I_r^{\mathrm{ren}}(s,t)
 =
 P_r(t)-P_r(s)+B_r(s,t).
\]

\section*{Acknowledgments}

The ideas leading to this work originated in August 2025.  Zongjian Han and Yuanhe Luo are currently working intensively on a fully human-written version of the manuscript.  In view of the recent frequency of information-leak incidents, we felt it necessary to first prepare the main body of this preliminary version with the assistance of AI and to make it publicly available only after careful review by the author.  We will subsequently update the manuscript with the fully human-written version, which will serve as the final version submitted for publication.


\begin{thebibliography}{99}

\bibitem{BahouriCheminDanchin2011}
H.~Bahouri, J.-Y.~Chemin, and R.~Danchin,
\emph{Fourier Analysis and Nonlinear Partial Differential Equations},
Grundlehren der mathematischen Wissenschaften 343, Springer, 2011.

\bibitem{CarberyWright2001}
A.~Carbery and J.~Wright,
\emph{Distributional and $L^q$ norm inequalities for polynomials over convex bodies in $\R^n$},
Math. Res. Lett. \textbf{8} (2001), no.~3, 233--248.

\bibitem{Chen1957}
K.-T.~Chen,
\emph{Integration of paths, geometric invariants and a generalized Baker--Hausdorff formula},
Ann. of Math. \textbf{65} (1957), 163--178.

\bibitem{ChevyrevFerrucci2026}
I.~Chevyrev and E.~Ferrucci,
\emph{Locality of rough path lifts},
arXiv:2606.21049v2, 2026.

\bibitem{ConnesKreimer1998}
A.~Connes and D.~Kreimer,
\emph{Hopf algebras, renormalization and noncommutative geometry},
Comm. Math. Phys. \textbf{199} (1998), 203--242.

\bibitem{CoutinQian2002}
L.~Coutin and Z.~Qian,
\emph{Stochastic analysis, rough path analysis and fractional Brownian motions},
Probab. Theory Related Fields \textbf{122} (2002), 108--140.

\bibitem{Foissy2002}
L.~Foissy,
\emph{Les alg\`ebres de Hopf des arbres enracin\'es d\'ecor\'es, I},
Bull. Sci. Math. \textbf{126} (2002), 193--239.

\bibitem{FrizVictoir2010}
P.~K.~Friz and N.~Victoir,
\emph{Multidimensional Stochastic Processes as Rough Paths},
Cambridge Studies in Advanced Mathematics 120, Cambridge University Press, 2010.

\bibitem{Gebelein1941}
H.~Gebelein,
\emph{Das statistische Problem der Korrelation als Variations- und Eigenwertproblem und sein Zusammenhang mit der Ausgleichsrechnung},
Z. Angew. Math. Mech. \textbf{21} (1941), 364--379.

\bibitem{Hochschild1945}
G.~Hochschild,
\emph{On the cohomology groups of an associative algebra},
Ann. of Math. \textbf{46} (1945), 58--67.

\bibitem{Lyons1998}
T.~J.~Lyons,
\emph{Differential equations driven by rough signals},
Rev. Mat. Iberoam. \textbf{14} (1998), no.~2, 215--310.

\bibitem{LyonsVictoir2007}
T.~Lyons and N.~Victoir,
\emph{An extension theorem to rough paths},
Ann. Inst. H. Poincar\'e Anal. Non Lin\'eaire \textbf{24} (2007), 835--847.

\bibitem{Picard2011}
J.~Picard,
\emph{Representation formulae for the fractional Brownian motion},
S\'eminaire de Probabilit\'es XLIII, Lecture Notes in Mathematics 2006,
Springer, 2011, 3--70; arXiv:0912.3168.

\bibitem{Reutenauer1993}
C.~Reutenauer,
\emph{Free Lie Algebras},
London Mathematical Society Monographs, New Series 7, Oxford University Press, 1993.

\bibitem{Triebel1983}
H.~Triebel,
\emph{Theory of Function Spaces},
Birkh\"auser, 1983.

\bibitem{UnterbergerFNO2010}
J.~Unterberger,
\emph{H\"older-continuous rough paths by Fourier normal ordering},
Comm. Math. Phys. \textbf{298} (2010), 1--36.

\bibitem{UnterbergerFBM2010}
J.~Unterberger,
\emph{A rough path over multidimensional fractional Brownian motion with arbitrary Hurst index by Fourier normal ordering},
Stochastic Process. Appl. \textbf{120} (2010), no.~8, 1444--1472.

\bibitem{Weaver2000}
N.~Weaver,
\emph{Lipschitz Algebras},
World Scientific, 1999.

\bibitem{Young1936}
L.~C.~Young,
\emph{An inequality of the H\"older type, connected with Stieltjes integration},
Acta Math. \textbf{67} (1936), 251--282.

\end{thebibliography}
\end{document}